\documentclass{article}
\usepackage{amsmath,amssymb}
\usepackage{longtable}
\usepackage{amsthm}
\usepackage{fancyhdr}

\makeatletter
\def\iddots{\mathinner{\mkern1mu\raise\p@
    \hbox{.}\mkern2mu\raise4\p@\hbox{.}\mkern2mu
    \raise7\p@\vbox{\kern7\p@\hbox{.}}\mkern1mu}}
\def\adots{\mathinner{\mkern2mu\raise\p@\hbox{.} %% yhmath.sty'©'ç
 \mkern2mu\raise4\p@\hbox{.}\mkern1mu
 \raise7\p@\vbox{\kern7\p@\hbox{.}}\mkern1mu}}
\makeatother

\usepackage{amsmath}
\usepackage{amssymb}

\today
\makeatletter
\def\udots{\mathinner{\mkern1mu\raise\p@
    \vbox{\kern7\p@\hbox{.}}\mkern2mu
            \raise4\p@\hbox{.}\mkern2mu\raise7\p@\hbox{.}\mkern1mu}}
\makeatother

\newtheorem{theo}{Theorem}[section]
\newtheorem{lemma}[theo]{Lemma}
\newtheorem{pr}[theo]{Proposition}

\theoremstyle{definition}
\newtheorem{defn}[theo]{Definition}
\newtheorem{remark}[theo]{Remark}
\newtheorem{ex}[theo]{Example}
\newtheorem{nt}[theo]{Notation}

\def\dim{\mathop{\rm dim}\nolimits}

\newcommand\bigzerou{\smash{\lower.3ex\hbox{\large 0}}} 
         
\newcommand\bigstaru{\smash{\lower.3ex\hbox{\Large $*$}}} 

\numberwithin{equation}{section}

\title{ Lie algebras generated by Lie modules}
\vspace{50truept}
\author{ Nagatoshi Sasano}
\begin{document}
\thispagestyle{empty}

\newpage
\begin{center}{\huge  Lie algebras constructed with Lie modules and their positively and negatively graded modules}\footnote{{\bf 2010 Mathematic Subjects Classification}: Primary 17B70; Secondary 17B65}\end{center}
\vspace{50truept}

\begin{center}{Nagatoshi SASANO}\end{center}
\begin{abstract}
In this paper, we shall give a way to construct a graded Lie algebra $L(\mathfrak{g},\rho,V,{\cal V},B_0)$ from a standard pentad $(\mathfrak{g},\rho,V,{\cal V},B_0)$ which consists of a Lie algebra $\mathfrak{g}$ which has a non-degenerate invariant bilinear form $B_0$ and $\mathfrak{g}$-modules $(\rho, V)$ and ${\cal V}\subset \mathrm {Hom }(V,\mathfrak{k})$ all defined over a field $\mathfrak{k}$.
In general, we do not assume that these objects are finite-dimensional.
We can embed the objects $\mathfrak{g},\rho,V,{\cal V}$ into $L(\mathfrak{g},\rho,V,{\cal V},B_0)$.
Moreover, we construct specific positively and negatively graded modules of $L(\mathfrak{g},\rho,V,{\cal V},B_0)$.
Finally, we give a chain rule on the embedding rules of standard pentads.
\end{abstract}

\section{Introduction}\label{sec;intro}
A standard quadruplet is a quadruplet of the form $(\mathfrak{g},\rho,V,B_0)$, where $\mathfrak{g}$ is a finite-dimensional reductive Lie algebra, $(\rho,V)$ a finite-dimensional representation of $\mathfrak{g}$ and $B_0$ a non-degenerate symmetric invariant bilinear form on $\mathfrak{g}$ all defined over the complex number field $\mathbb{C}$, which satisfies the conditions that $\rho $ is faithful and completely reducible and that $V$ does not have a non-zero invariant element.
In \cite{Sa}, the author proved that any standard quadruplet $(\mathfrak{g},\rho,V,B_0)$ has a graded Lie algebra, denoted by $L(\mathfrak{g},\rho,V,B_0)=\bigoplus _{n\in\mathbb{Z}}V_n$, such that $V_0\simeq \mathfrak{g}$, $V_1\simeq V$ and $V_{-1}\simeq \mathrm {Hom }(V,\mathbb{C})$ (see \cite[Theorem 2.11]{Sa}).
That is, any finite-dimensional reductive Lie algebra and its finite-dimensional faithful and completely reducible representation can be embedded into some (finite or infinite-dimensional) graded Lie algebra.
We call a graded Lie algebra of the form $L(\mathfrak{g},\rho,V,B_0)$ the Lie algebra associated with a standard quadruplet.
Some well-known Lie algebras correspond to some standard quadruplet, for example, finite-dimensional semisimple Lie algebras and loop algebras.
Moreover, the bilinear form $B_0$ can be also embedded into $L(\mathfrak{g},\rho,V,B_0)$, i.e. there exists a non-degenerate symmetric invariant bilinear form on $L(\mathfrak{g},\rho,V,B_0)$ whose restriction to $V_0\times V_0$ coincides with $B_0$ (see \cite[Proposition 3.2]{Sa}).
By the way, H. Rubenthaler obtained some similar results in \cite {arxiv} using the Kac theory in \cite{ka-1}.

\par The first purpose of this paper is to extend the theory of standard quadruplets to the cases where the objects are infinite-dimensional.
For this, we need to consider pentads $(\mathfrak{g},\rho,V,{\cal V},B_0)$ instead of quadruplets, where $\mathfrak{g}$ is a finite or infinite-dimensional Lie algebra, $\rho :\mathfrak{g}\otimes V\rightarrow V$ a representation of $\mathfrak{g}$ on a finite or infinite-dimensional vector space $V$, ${\cal V}$ a $\mathfrak{g}$-submodule of $\mathrm {Hom}(V,\mathfrak{k})$, $B_0$ a non-degenerate invariant bilinear form on $\mathfrak{g}$ all defined over a field $\mathfrak{k}$.
In general, we do not assume that $B_0$ is symmetric.
We define the notion of {\it standard pentads} by the existence of a linear map $\Phi _{\rho }:V\otimes {\cal V}\rightarrow \mathfrak{g}$ satisfying $B_0(a,\Phi _{\rho }(v\otimes \phi ))=\langle \rho (a\otimes v),\phi \rangle$ for any $a\in\mathfrak{g}$, $v\in V$ and $\phi \in {\cal V}$.
A standard quadruplet $(\mathfrak{g},\rho,V,B_0)$ can be naturally regarded as a standard pentad $(\mathfrak{g},\rho,V,\mathrm {Hom}(V,\mathbb{C}),B_0)$, and, thus, we can say that the notion of standard pentads is an extension of the notion of standard quadruplets.
Then, by a similar argument to the argument in \cite{Sa}, we can construct a graded Lie algebra from an arbitrary standard pentad $(\mathfrak{g},\rho,V,{\cal V},B_0)$ denoted by $L(\mathfrak{g},\rho,V,{\cal V},B_0)=\bigoplus _{n\in \mathbb{Z}}V_n$ such that the objects $\mathfrak{g},\rho,V,{\cal V}$ can be embedded into it.
We call such a graded Lie algebra {\it a Lie algebra associated with a standard pentad}.
This is the first main result of this paper.
Of course, the graded Lie algebra associated with a standard quadruplet $(\mathfrak{g},\rho,V,B_0)$ is isomorphic to the graded Lie algebra associated with a standard pentad $(\mathfrak{g},\rho,V,\mathrm {Hom}(V,\mathbb{C}),B_0)$.
Moreover, if the bilinear form $B_0$ of $(\mathfrak{g},\rho,V,{\cal V},B_0)$ is symmetric, then $B_0$ can be also embedded into $L(\mathfrak{g},\rho,V,{\cal V},B_0)$, i.e. there exists a non-degenerate symmetric invariant bilinear form $B_L$ on $L(\mathfrak{g},\rho,V,{\cal V},B_0)$ whose restriction to $V_0\times V_0$ coincides with $B_0$.

\par When $B_0$ is symmetric, we can expect that a Lie algebra of the form $L(\mathfrak{g},\rho,V,{\cal V},B_0)$ (not necessary finite-dimensional) and its representation can be embedded into some graded Lie algebra using $B_L$.
The second purpose is to construct positively graded modules and negatively graded modules of $L(\mathfrak{g},\rho,V,{\cal V},B_0)$ which can be embedded into some graded Lie algebra under some assumptions.
In general, it is known that for any graded Lie algebra $\mathfrak{l}=\bigoplus _{n\in\mathbb{Z}}\mathfrak{l}_n$ and $\mathfrak{l}_0$-module $U$, there exists a positively (respectively negatively) graded $\mathfrak{l}$-module such that the base space (respectively top space) is the given $\mathfrak{l}_0$-module $U$ (see \cite [Theorem 1.2]{Shen}).
In this paper, we shall try to construct such $L(\mathfrak{g},\rho,V,{\cal V},B_0)$-modules from a $\mathfrak{g}$-module $(\pi, U)$ using a similar way to the construction of a Lie algebra associated with a standard pentad.
Precisely, we inductively construct a positively (respectively negatively) graded $L(\mathfrak{g},\rho,V,{\cal V},B_0)$-module $(\tilde{\pi }^+,\tilde{U}^+)$, $\tilde{U}^+=\bigoplus _{m\geq 0}U_m^+$ (respectively $(\tilde{\pi }^-,\tilde{U}^-)$, $\tilde{U}^-=\bigoplus _{m\leq 0}U_m^-$) such that the ``base space'' $U^+_0$ (respectively the ``top space'' $U^-_0$) is the given $\mathfrak{g}$-module $U$.
In general,  the modules $\tilde{U}^+$ and $\tilde{U}^-$ are infinite-dimensional.
We shall try to embed $L(\mathfrak{g},\rho,V,{\cal V},B_0)$ and its module of the form $\tilde{U}^+$ into some graded Lie algebra.
If we assume that $B_0$ is symmetric and that $U$ has a $\mathfrak{g}$-submodule ${\cal U}$ of $\mathrm{Hom }(U,\mathfrak{k})$ such that $(\mathfrak{g},\pi,U,{\cal U},B_0)$ is a standard pentad, then we can embed the objects $L(\mathfrak{g},\rho,V,{\cal V},B_0)$ and $\tilde{U}^+$ into some graded Lie algebra.
Precisely, under these assumptions, we have that a pentad $(L(\mathfrak{g},\rho,V,{\cal V},B_0),\tilde{ \pi}^+ ,\tilde{ U}^+,\tilde{\cal U}^-,B_L)$ is also standard, and, thus, we can embed the objects $L(\mathfrak{g},\rho,V,{\cal V},B_0)$, $\tilde{ U}^+$, $\tilde{\cal U}^-$ into the graded Lie algebra $L(L(\mathfrak{g},\rho,V,{\cal V},B_0),\tilde{ \pi}^+ ,\tilde{ U}^+,\tilde{\cal U}^-,B_L)$.
In this situation, we have a ``chain rule'' of the Lie algebras associated with a standard pentad.
This is the second main result of this paper.

\par This paper consists of three sections.
\par In section \ref{sec;standard_pentads}, we shall study the Lie algebras associated with a standard pentad.
First, in section \ref{sec;prepair_stap}, we define the notion of standard pentads (see Definition \ref{defn;stap}) and construct a graded Lie algebra from a standard pentad $(\mathfrak{g},\rho,V,{\cal V},B_0)$, which is denoted by $L(\mathfrak{g},\rho,V,{\cal V},B_0)=\bigoplus _{n\in \mathbb{Z}}V_n$ (see Theorem \ref{StapLieAlgebraexists}).
In section \ref{sec;standard_pentads_with_sym_bilinear}, we consider some properties of Lie algebras of the form $L(\mathfrak{g},\rho,V,{\cal V},B_0)$ such that $B_0$ is symmetric.
In these cases, we can also embed the bilinear form $B_0$ into $L(\mathfrak{g},\rho,V,{\cal V},B_0)$, i.e. we can obtain a non-degenerate symmetric invariant bilinear form on $L(\mathfrak{g},\rho,V,{\cal V},B_0)$ whose restriction to $V_0\times V_0$ coincides with $B_0$ (see Proposition \ref{pr;stap_bilinear_exist}).
Moreover, the Lie algebra $L(\mathfrak{g},\rho,V,{\cal V},B_0)$ can be characterized by the transitivity and the existence of such a bilinear form (see Theorem  \ref{th;universality_stap}).
Finally, we give two lemmas on derivations on $L(\mathfrak{g},\rho,V,{\cal V},B_0)$ (see Lemmas \ref{lem;derivationext} and \ref{lemma;der_Binv}).

\par In section \ref{sec;graded_modules}, we shall study positively and negatively graded modules of a Lie algebra of the form $L(\mathfrak{g},\rho,V,{\cal V},B_0)$.
First, in sections \ref{sec;const_graded_ext_vec} and \ref{sec;const_graded_ext}, we shall construct positively graded $L(\mathfrak{g},\rho,V,{\cal V},B_0)$-module and negatively graded $L(\mathfrak{g},\rho,V,{\cal V},B_0)$-module from a $\mathfrak{g}$-module $(\pi, U)$, i.e. we shall give another proof of \cite [Theorem 1.2]{Shen} in the special cases where the graded Lie algebra is of the form $L(\mathfrak{g},\rho,V,{\cal V},B_0)$.
In section \ref{sec;const_graded_ext_vec}, we construct a family of $\mathfrak{g}$-modules $\{U^+_m\}_{m\geq 0}$ (respectively $\{U^-_m\}_{m\leq 0}$) from the pentad $(\mathfrak{g},\rho,V,{\cal V},B_0)$ and the $\mathfrak{g}$-module $(\pi, U)$ by induction.
In section  \ref{sec;const_graded_ext}, we define a structure of positively (respectively negatively) graded $L(\mathfrak{g},\rho,V,{\cal V},B_0)$-module on $\tilde{U}^+:=\bigoplus _{m\geq 0}U_m^+$ (respectively $\tilde{U}^-:=\bigoplus _{m\leq 0}U_m^-$).
We call this positively (respectively negatively) graded module of $L(\mathfrak{g},\rho,V,{\cal V},B_0)$ the {\it positive extension} (respectively {\it negative extension}) of $U$ with respect to $(\mathfrak{g},\rho,V,{\cal V},B_0)$ (see Theorems \ref{th;pi_positivelyextend} and \ref{th;pi_negativelyextend}).
These modules are transitive and characterized by their transitivity (see Theorem \ref{th;universalitystap}).
In sections \ref{sec;stapext_1} and \ref{sec;stapext_2}, we try to construct a standard pentad which contains a Lie algebra of the form $L(\mathfrak{g},\rho,V,{\cal V},B_0)$ and its module of the form $\tilde{U}^+$.
For this, we need to assume that $B_0$ is symmetric and $U$ is embedded into some standard pentad $(\mathfrak{g},\pi,U,{\cal U},B_0)$.
In section \ref{sec;stapext_1}, for the $\mathfrak{g}$-submodule ${\cal U}$ of $\mathrm {Hom }(U,\mathfrak{k})$, we shall extend the canonical pairing $U\times {\cal U}$ to $\tilde{U}^+\times \tilde{\cal U}^-$.
Moreover, in section \ref{sec;stapext_2}, we shall construct the $\Phi $-map of $(L(\mathfrak{g},\rho,V,{\cal V},B_0),\tilde{\pi }^+,\tilde{U}^+,\tilde{\cal U}^-,B_L)$ from the $\Phi $-map of the pentad $(\mathfrak{g},\pi,U,{\cal U},B_0)$ inductively.
Consequently, under the assumptions that $(\mathfrak{g},\rho,V,{\cal V},B_0)$ and $(\mathfrak{g},\pi,U,{\cal U},B_0)$ are standard pentads and that their bilinear form $B_0$ is symmetric, we can embed the Lie algebra $L(\mathfrak{g},\rho,V,{\cal V},B_0)$ and its module $\tilde{U}^+$ into a standard pentad $(L(\mathfrak{g},\rho,V,{\cal V},B_0),\tilde{\pi }^+,\tilde{U}^+,\tilde{\cal U}^-,B_L)$.
Finally, in section \ref{sec;chainrule}, we consider the graded Lie algebra $L(L(\mathfrak{g},\rho,V,{\cal V},B_0),\tilde{\pi }^+,\tilde{U}^+,\tilde{\cal U}^-,B_L)$ under the situation of sections \ref{sec;stapext_1} and \ref{sec;stapext_2}.
From the constructions of $L(\mathfrak{g},\rho,V,{\cal V},B_0)$, $\tilde{U}^+$ and $\tilde{\cal U}^-$, we can write this graded Lie algebra using the data $\mathfrak{g},\rho,V,{\cal V},B_0$ and $U,{\cal U}$.
Precisely, we have the following result on the structures of Lie algebras:
$$
L(L(\mathfrak{g},\rho,V,{\cal V},B_0),\tilde{\pi }^+,\tilde{U}^+,\tilde{\cal U}^-,B_L)\simeq L(\mathfrak{g},\rho\oplus \pi,V\oplus U,{\cal V}\oplus {\cal U},B_0).
$$
This is a chain rule in the theory of standard pentads (see Theorem \ref{theo;chainrule}).
\begin{nt}
In this paper, we regard a representation $\rho $ of a Lie algebra $\mathfrak{l}$ on $V$ as a linear map $\rho :\mathfrak{l}\otimes V\rightarrow V$ which satisfies that 
$$
\rho ([a,b]\otimes v)=\rho (a\otimes \rho(b\otimes v))-\rho (b\otimes \rho(a\otimes v))
$$
for any $a,b\in \mathfrak{l}$ and $v\in V$.
\end{nt}
\begin{defn}
In this paper, we say that a Lie algebra $\mathfrak{l}$ is a {\it $\mathbb{Z}$-graded Lie algebra} or simply a {\it graded Lie algebra} if and only if there exist vector subspaces $\mathfrak{l}_n$ of $\mathfrak{l}$ for all $n\in\mathbb{Z}$ such that:
\begin{itemize}
\item{$\mathfrak{l}=\bigoplus _{n\in\mathbb{Z}}\mathfrak{l}_n$\quad and\quad  $[\mathfrak{l}_n,\mathfrak{l}_m]\subset \mathfrak{l}_{n+m}$ for any $n,m\in\mathbb{Z}$,}
\item{$\mathfrak{l}$ is generated by $\mathfrak{l}_{-1}\oplus \mathfrak{l}_0\oplus \mathfrak{l}_1$.}
\end{itemize}
In general, we do not assume that each $\mathfrak{l}_n$ is finite-dimensional (cf. \cite[Definition 1]{ka-1}).
\par 
Moreover, if $\mathfrak{l}$ satisfies the following two conditions, we say that $\mathfrak{l}$ is transitive (see \cite[Definition 2]{ka-1}):
\begin{itemize}
\item{for $x\in \mathfrak{l}_{i}$, $i\geq 0$, $[x,\mathfrak{l}_{-1}]=\{0\}$ implies $x=0$,}
\item{for $x\in \mathfrak{l}_{i}$, $i\leq 0$, $[x,\mathfrak{l}_{1}]=\{0\}$ implies $x=0$.}
\end{itemize}
\end{defn}
\begin{defn}
In this paper, we say that a module $(\varpi ^+, W)$, $W=\bigoplus _{m\geq 0}W_m$ (respectively $(\varpi ^-, W)$, $W=\bigoplus _{m\leq 0}W_m$) of a graded Lie algebra $\bigoplus _{n\in\mathbb{Z}}\mathfrak{l}_n$ is positively graded (respectively negatively graded) when $\varpi ^+(\mathfrak{l}_n\otimes W_m)\subset W_{n+m}$ (respectively $\varpi ^-(\mathfrak{l}_n\otimes W_m)\subset W_{n+m}$) for any $n,m$ (cf. \cite[Definition 0.1]{Shen}), and, moreover, we say that a positively graded module $(\varpi ^+, W)$ (respectively a negatively graded module $(\varpi ^-, W)$) is transitive when the following condition holds (cf. \cite[Definition 1.1]{Shen}):
\begin{align*}
&\text{for $w\in W_{m}$, $m\geq 1$, $\varpi ^+(V_{-1}\otimes w)=\{0\}$ implies $w=0$}\\
(\text{respectively} \quad &\text{for $w\in W_{m}$, $m\leq -1$, $\varpi ^-(V_{1}\otimes w)=\{0\}$ implies $w=0$}).
\end{align*}
\end{defn}

\begin{nt}
In this paper, we denote the set of all natural numbers, integers and complex numbers by $\mathbb{N}$, $\mathbb{Z}$ and $\mathbb{C}$ respectively.
We denote the set of matrices of size $n\times m$ $(n,m\in\mathbb{N})$ whose entries are belong to a ring $R$ by $M(n,m;R)$, the unit matrix and the zero matrix of size $n$ by $I_n$ and $O_n$ respectively.
Moreover, $\delta _{kl}$ stands for the Kronecker delta, $\mathrm{Tr}(A)$ stands for the trace of a square matrix $A$.
\end{nt}

\section{Standard pentads and corresponding Lie algebras}\label{sec;standard_pentads}
\subsection{Standard pentads}\label{sec;prepair_stap}
Let us start with the definitions of $\Phi $-map and standard pentads.
\begin{defn}[$\Phi $-map]\label{defn;phimap}
Let $\mathfrak{g}$ be a non-zero Lie algebra with non-degenerate invariant bilinear form $B_0$, $\rho :\mathfrak{g}\otimes V\rightarrow V$ a representation of $\mathfrak{g}$ on a vector space $V$ and ${\cal V}$ a $\mathfrak{g}$-submodule of ${\rm Hom }(V,\mathfrak{k})$ all defined over a field $\mathfrak{k}$.
We denote the canonical pairing between $V$ and $\mathrm {Hom}(V,\mathfrak{k})$ by $\langle \cdot,\cdot\rangle $ and the canonical representation of $\mathfrak{g}$ on ${\cal V}$ by $\varrho$.
Then, if a pentad $(\mathfrak{g},\rho,V,{\cal V},B_0)$ has a linear map $\Phi _{\rho }:V\otimes {\cal V}\rightarrow \mathfrak{g}$ which satisfies an equation
\begin{align}
B_0(a,\Phi _{\rho }(v\otimes \phi ))=\langle \rho (a\otimes v),\phi \rangle =-\langle v,\varrho (a\otimes \phi )\rangle \label{defn;eq_stap_Phimap}
\end{align}for any $a\in \mathfrak{g}$, $v\in V$ and $\phi \in {\cal V}$ (cf. \cite [Definition 1.9]{Sa}), we call it a $\Phi $-map of the pentad $(\mathfrak{g},\rho,V,{\cal V},B_0)$.
Moreover, when a pentad $(\mathfrak{g},\rho,V,{\cal V},B_0)$ has a $\Phi $-map, we define a linear map $\Psi _{\rho }:{\cal V}\otimes V\rightarrow \mathfrak{g}$ by:
\begin{align}
B_0(a,\Psi _{\rho }(\phi \otimes v))=\langle v,\varrho (a\otimes \phi) \rangle =-\langle \rho (a\otimes v),\phi \rangle.
\end{align}
We call this map $\Psi _{\rho }$ a $\Psi $-map of $(\mathfrak{g},\rho,V,{\cal V},B_0)$.
\end{defn}
In general, a pentad might not have a $\Phi $-map.
If a pentad $(\mathfrak{g},\rho,V,{\cal V},B_0)$ has a $\Phi $-map, then the equation (\ref{defn;eq_stap_Phimap}) determines the linear map $\Phi _{\rho }$ uniquely.
Moreover, we have an equation 
$$
\Phi _{\rho }(v\otimes \phi )+\Psi _{\rho }(\phi \otimes v)=0
$$
for any $v\in V$ and $\phi \in {\cal V}$.
\begin{defn}[Standard pentads]\label{defn;stap}
We retain to use the notation of Definition \ref{defn;phimap}.
If a pentad $(\mathfrak{g},\rho,V,{\cal V},B_0)$ satisfies the following conditions, we call it a {\it standard pentad}:
\begin{align}&\text{the restriction of $\langle \cdot,\cdot\rangle$ to $V\times {\cal V}$ is non-degenerate,}\label{defn;condi_stap_pairing}\\&\text{there exists a $\Phi $-map from $V\otimes {\cal V}$ to $\mathfrak{g}$.}\label{defn;condi_stap_Phimap}\end{align}
\end{defn}
\begin{lemma}
Under the notation of Definitions \ref{defn;phimap} and \ref{defn;stap}, we have the following claims:
\begin{align}
&\text{if $V$ is finite-dimensional, then a vector space ${\cal V}$ satisfying (\ref{defn;condi_stap_pairing}) coincides with ${\rm Hom }(V,\mathfrak{k})$,}\label{lem;vfin}\\
&\text{if $\mathfrak{g}$ is finite-dimensional, then any pentad $(\mathfrak{g},\rho,V,{\cal V},B_0)$ satisfies the condition (\ref{defn;condi_stap_Phimap}).}\label{lem;gfin}
\end{align}
In particular, if both $\mathfrak{g}$ and $V$ are finite-dimensional, then any quadruplet $(\mathfrak{g},\rho,V,B_0)$ can be naturally regarded as a standard pentad $(\mathfrak{g},\rho,V,{\rm Hom }(V,\mathfrak{k}),B_0)$.
\end{lemma}
\begin{proof}
The claim (\ref{lem;vfin}) is clear.
Let us show the claim (\ref{lem;gfin}).
If $\mathfrak{g}$ is finite-dimensional, then the dual space of $\mathfrak{g}$ can be identified with $\mathfrak{g}$.
Precisely, if $\mathfrak{g}$ is finite-dimensional, then any linear map $f :\mathfrak{g}\rightarrow \mathfrak{k}$ corresponds to some element $A\in \mathfrak{g}$ such that 
$$
f (a)=B_0(a,A)
$$
for any $a\in\mathfrak{g}$.
Thus, for any $v\in V$ and $\phi \in {\cal V}$, there exists an element of $\mathfrak{g}$ which corresponds to a linear map $\mathfrak{g}\rightarrow \mathfrak{k}$ defined by
$$
a\mapsto \langle \rho (a\otimes v),\phi \rangle.
$$
It means that the pentad $(\mathfrak{g},\rho,V,{\cal V},B_0)$ has the $\Phi $-map.
\end{proof}
\begin{remark}
If $V$ is infinite-dimensional, then a submodule ${\cal V}$ of $\mathrm{Hom }(V,\mathfrak{k})$ satisfying the condition (\ref{defn;condi_stap_pairing}) does not necessary coincide with $\mathrm {Hom}(V,\mathfrak{k})$.
\end{remark}
\begin{remark}\label{rem;mght_not_stap}
In general, a Lie algebra $\mathfrak{g}$ and its module $(\rho,V)$ might not have a $\mathfrak{g}$-submodule ${\cal V}\subset \mathrm{Hom}(V,\mathfrak{k})$ and a bilinear form $B_0$ such that a pentad $(\mathfrak{g},\rho,V,{\cal V},B_0)$ is standard.
\end{remark}
\begin{ex}
Let $\mathfrak{g}=\mathfrak{sl}_2(\mathbb{C})$, $K$ be the Killing form on $\mathfrak{g}$ and ${\cal L}(\mathfrak{g})=\mathbb{C}[t,t^{-1}]\otimes \mathfrak{g}$ be the loop algebra (see \cite [Ch.7]{ka-2}).
Let $K_{\cal L}$ be a bilinear form on ${\cal L}(\mathfrak{g})$ defined by:
$$
K_{\cal L}(t^n\otimes X,t^m\otimes Y):=\delta _{n+m,0}K(X,Y).
$$
Clearly, the bilinear form $K_{\cal L}$ is non-degenerate and invariant.
Thus, we can regard ${\cal L}(\mathfrak{g})$ itself as a ${\cal L}(\mathfrak{g})$-submodule of $\mathrm {Hom }( {\cal L}(\mathfrak{g}),\mathbb{C})$ via the non-degenerate invariant bilinear form $K_{\cal L}$.
Then, a pentad $({\cal L}(\mathfrak{g}),\mathrm {ad}, {\cal L}(\mathfrak{g}),{\cal L}(\mathfrak{g}),K_{\cal L})$, where $\mathrm {ad}$ stands for the adjoint representation, is standard.
In fact, we have the condition (\ref{defn;condi_stap_pairing}) clearly, and, we can identify the bracket product ${\cal L}(\mathfrak{g})\times {\cal L}(\mathfrak{g})\rightarrow {\cal L}(\mathfrak{g})$ with the $\Phi $-map of $({\cal L}(\mathfrak{g}),\mathrm {ad}, {\cal L}(\mathfrak{g}),{\cal L}(\mathfrak{g}),K_{\cal L})$, denoted by $\Phi _{\mathrm {ad}}^1$.
\par However, a pentad $({\cal L}(\mathfrak{g}),\mathrm {ad}, {\cal L}(\mathfrak{g}),\mathrm {Hom}({\cal L}(\mathfrak{g}),\mathbb{C}),K_{\cal L})$ is not standard since it does not have the $\Phi $-map.
In fact, if we assume that this pentad might have the $\Phi $-map, denoted by $\Phi _{\mathrm {ad}}^2$, and put
\begin{align*}
H_0:=\begin{pmatrix}
1&0\\0&-1
\end{pmatrix},\quad
X_0:=\begin{pmatrix}
0&1\\0&0
\end{pmatrix},
\quad
Y_0:=\begin{pmatrix}
0&0\\1&0
\end{pmatrix}
\in\mathfrak{g},\\
\quad \phi _{Y_0}\in \mathrm {Hom}({\cal L}(\mathfrak{g}),\mathbb{C}),\quad \langle t^n\otimes X, \phi _{Y_0}\rangle := K(Y_0,X),
\end{align*}
then an element $\Phi _{\mathrm {ad}}^2((1\otimes X_0)\otimes \phi _{Y_0})\in {\cal L}(\mathfrak{g})$ satisfies the equation 
\begin{align}
&K_{\cal L}(t^n\otimes H_0,\Phi _{\mathrm {ad}}^2 ((1\otimes X_0)\otimes \phi _{Y_0}))=\langle [t^n\otimes H_0,1\otimes X_0],\phi _{Y_0}\rangle\notag \\
&\quad =\langle t^n\otimes 2X_0,\phi _{Y_0}\rangle\notag\\
&\quad =K(Y_0,2X_0)\notag \\
&\quad =8 \label{eq;1}
\end{align}
for any $n\in \mathbb{Z}$.
The Lie algebra ${\cal L}(\mathfrak{g})$ does not have an element satisfying (\ref {eq;1}) for any $n\in\mathbb{Z}$, and, thus, the pentad $({\cal L}(\mathfrak{g}),\mathrm {ad}, {\cal L}(\mathfrak{g}),\mathrm {Hom}({\cal L}(\mathfrak{g}),\mathbb{C}),K_{\cal L})$ does not have the $\Phi $-map.
\end{ex}

On the $\Phi $-map and $\Psi $-map of a standard pentad, we have similar properties to ones of the $\Phi $-map and $\Psi $-map of a standard quadruplet (see \cite{Sa}).

\begin{pr}\label{pr;phopsi_liehom}
The $\Phi $-map and the $\Psi $-map of a standard quadruplet $(\mathfrak{g},\rho,V,{\cal V},B_0)$ are homomorphisms of Lie modules.
 (cf. \cite [Proposition 1.3]{Sa}).
\end{pr}
\begin{proof}
We can prove it by the same way to \cite[Proposition 1.3]{Sa}.
\end{proof}
\begin{defn}\label{defn;Phi_rhov}
Let $(\mathfrak{g},\rho,V,{\cal V},B_0)$ be a standard pentad.
For each element $v\in V$ and $\phi \in {\cal V}$, we define linear maps $\Phi _{\rho, v}\in\text{Hom}({\cal V},\mathfrak{g})$ and $\Psi _{\rho, \phi}\in\text{Hom}(V,\mathfrak{g})$  by:
\begin{align*}
&\Phi _{\rho, v} (\psi):={\Phi }_{\rho } (v\otimes \psi), \hspace {10truept}\Psi _{\rho, \phi} (u):={\Psi }_{\rho } (\phi \otimes u)
\end{align*}
for any $u\in V$ and $\psi \in {\cal V}$.
Moreover, we define the following linear maps:
\begin{align*}
\Phi _{\rho }^{\circ}:\ &V\rightarrow \mathrm {Hom} ({\cal V},\mathfrak{g})&
\Psi _{\rho }^{\circ}:\ &{\cal V}\rightarrow \mathrm {Hom} (V,\mathfrak{g})&\notag\\
&v\mapsto \Phi _{\rho, v} ,&
&\phi \mapsto \Psi _{\rho, \phi}.&
\end{align*}
To simplify, we denote $\Phi _{\rho ,v}(\psi )$ and $\Psi _{\rho,\phi }(u)$ by $v(\psi )$ and $\phi (u)$ respectively.
\end{defn}
\begin{defn}\label{defn_p_q}
Let $(\mathfrak{g},\rho,V,{\cal V},B_0)$ be a standard pentad.
Put $V_0:=\mathfrak{g}$, $V_1:=V$ and $V_{-1}:={\cal V}$ and denote the canonical representations of $\mathfrak{g}$ on $V_0$ and $V_{\pm 1}$ by $\rho _0$ and $\rho _{\pm 1}$.
We define homomorphisms of $\mathfrak{g}$-modules $p_0$ and $q_0$ by:
\begin{align*}
p_0:&V_1\otimes V_0\rightarrow V_1\notag\\
&v_1\otimes a\mapsto -\rho _1(a\otimes v_1),\\
q_0:&V_{-1}\otimes V_0\rightarrow V_{-1}\notag\\
&\phi _{-1}\otimes b\mapsto -\rho _{-1}(b\otimes \phi _{-1}).
\end{align*}
Moreover, we define homomorphisms of $\mathfrak{g}$-modules $p_1$ and $q_{-1}$ by:
\begin{align*}
p_1:&V_1\otimes V_1\rightarrow \mathrm {Hom}(V_{-1},V_1)\notag\\
&v_1\otimes u_1\mapsto (\eta  _{-1}\mapsto \rho _1(v_1(\eta _{-1})\otimes u_1)-\rho _1(u_1(\eta _{-1})\otimes v_1)),\\
q_{-1}:&V_{-1}\otimes V_{-1}\rightarrow \mathrm {Hom}(V_1,V_{-1})\notag\\
&\phi _{-1}\otimes \psi _{-1}\mapsto (\xi _1\mapsto \rho _{-1}(\phi _{-1}(\xi _1)\otimes \psi _{-1})-\rho _{-1}(\psi _{-1}(\xi _1)\otimes \phi _{-1})),
\end{align*}
where $v_1(\eta _{-1})\in V_0$ and $\phi _{-1}(\xi _1)\in V_0$ stand for $\Phi _{\rho ,v_1}(\eta _{-1})$ and $\Psi _{\rho ,\phi _{-1}}(\xi _1)$ respectively.
\par Moreover, suppose that $i\geq 2$ and there exist $\mathfrak{g}$-modules $(\rho _{i-1},V_{i-1})$ and $(\rho _{-i+1},V_{-i+1})$
and homomorphisms of $\mathfrak{g}$-modules $p_{i-1}:V_{1}\otimes V_{i-1}\rightarrow {\rm Hom}(V_{-1},V_{i-1})$ and $q_{-i+1}:V_{-1}\otimes V_{-i+1}\rightarrow {\rm Hom}(V_{1},V_{-i+1})$.
Then, we put $V_{i}:={\rm Im}\ p_{i-1}$, $V_{-i}:={\rm Im}\ q_{-i+1}$ and define linear maps $p_{i}$, $q_{-i}$ by:
\begin{align*}
p_{i}:\ &V_{1}\otimes V_{i}\rightarrow \text{Hom} (V_{-1},V_{i})\notag\\
&v_1\otimes u_{i}\mapsto (\eta _{-1}\mapsto \rho _{i}(v_1 (\eta _{-1})\otimes u_{i})+p_{i-1}( v_1\otimes u_{i}(\eta _{-1}))) ,\\
q_{-i}:\ &V_{-1}\otimes V_{-i}\rightarrow \text{Hom} (V_1,V_{-i})\notag\\
&\phi _{-1}\otimes \psi _{-i}\mapsto (\xi _1\mapsto \rho _{-i}(\phi _{-1}(\xi _1)\otimes \psi _{-i})+q_{-i+1} (\phi _{-1}\otimes \psi _{-i}(\xi _1))),
\end{align*}
where $u_i(\eta _{-1})\in V_{i-1}$ and $\psi _{-i}(\xi _1)\in V_{-i+1}$ are the images of $\eta _{-1}$ and $\xi _1$ via $u_i$ and $\psi _{-i}$ respectively.
Then, the linear maps $p_{i}$ and $q_{-i}$ are homomorphisms of $\mathfrak{g}$-modules (cf. \cite[Proposition 1.10]{Sa}).
We denote the images of $p_{i}$ and $q_{-i}$ by $V_{i+1}$ and $V_{-i-1}$ and the canonical representations of $\mathfrak{g}$ on $V_{i+1}$ and $V_{-i-1}$ by $\rho _{i+1}$ and $\rho _{-i-1}$ respectively.
Thus, inductively, we obtain $\mathfrak{g}$-modules $V_n$ and representations $\rho _n$ of $\mathfrak{g}$ on $V_n$ for all $n\in\mathbb{Z}$.
We call $V_n$ the {\it $n$-graduation} of $(\mathfrak{g},\rho,V,{\cal V},B_0)$.
\end{defn}
\begin{remark}
For any $v_1\in V_1$ and $\phi _{-1}\in V_{-1}$, we have
\begin{align*}
&p_1(v_1\otimes v_1)(\eta _{-1})=\rho _1(v_1(\eta _{-1})\otimes v_1)-\rho _1(v_1(\eta _{-1})\otimes v_1)=0,\\
&q_{-1}(\phi _{-1}\otimes \phi _{-1})(\xi_1)=\rho _{-1}(\phi _{-1}(\xi _1)\otimes \phi _{-1})-\rho _{-1}(\phi _{-1}(\xi _1)\otimes \phi _{-1})=0.
\end{align*} 
\end{remark}
In general, we do not assume that $\rho $ and $\varrho $ are surjective, i.e. we do not assume that $V_1=\mathrm{Im}\ p_0$ and $V_{-1}=\mathrm {Im}\ q_0$.
In particular cases where these linear maps are surjective, we have the following proposition.
\begin{pr}\label{defn;case_rho_sur}
If $\rho :\mathfrak{g}\otimes V\rightarrow V$ and $\varrho :\mathfrak{g}\otimes {\cal V}\rightarrow {\cal V}$ are surjective, then $\Phi ^{\circ}_{\rho }$ and $\Psi ^{\circ}_{\rho}$ are injective, and, thus, $V$ and ${\cal V}$ can be regarded as $\mathfrak{g}$-submodules of $\mathrm{Hom}(V_{-1},V_0)$ and $\mathrm {Hom }(V_1,V_0)$ respectively.
\end{pr}
\begin{proof}
To show this proposition, we use the condition (\ref{defn;condi_stap_pairing}).
Let us show that the linear map $\Phi _{\rho }^{\circ}$ is injective.
We take an arbitrary element $v\in V$ which satisfies that $\Phi _{\rho ,v}=0$.
Then we have
\begin{align}
0=B_0(a,\Phi _{\rho, v}(\phi ))=\langle \rho (a\otimes v),\phi \rangle =-\langle v,\varrho (a\otimes \phi) \rangle
\end{align}for all $a\in \mathfrak{g}$ and $\phi \in {\cal V}$.
By the condition (\ref{defn;condi_stap_pairing}) and the assumption that $\varrho $ is surjective, we have that $v=0$.
Therefore, we obtain that $\Phi _{\rho }^{\circ}$ is injective.
Similarly, we can show that $\Psi _{\rho }^{\circ}$ is injective.
\end{proof}
\begin{defn}
We define the following bilinear maps
\begin{align*}
 [\cdot, \cdot  ]^0_n:V_0\times V_n\rightarrow V_n,\quad 
  [\cdot,\cdot   ]^1_n :V_1\times V_n\rightarrow V_{n+1},\quad 
  [\cdot,\cdot   ]^{-1}_n :V_{-1}\times V_n\rightarrow V_{n-1}
\end{align*} by:
\begin{align*}
  [a_0,z_n  ]^0_n:=\rho _n(a_0\otimes z_n),
\end{align*}
\begin{align*}
  [x_1 ,z_n  ]_{n}^1:=\begin{cases}p_n(x_1\otimes z_n)&(n\geq 0)\\-z_n (x_1) &(n\leq -1)\end{cases},
\end{align*}
\begin{align*}
  [y _{-1},z_n  ]_{n}^{-1}:=\begin{cases}-z_n (y _{-1} )&(n\geq 1)\\q_n(y _{-1}\otimes z_n)&(n\leq 0)\end{cases}
\end{align*}where $a_0\in V_0$, $x_1\in V_1$, $y _{-1}\in V_{-1}$ and $z_n\in V_n$.
Note that $z_n(x_1)$ stands for $\Psi _{\rho, z_{-1}}(x_1)$ when $n=-1$ and the image of $x_1$ via $z_n\in \mathrm{Hom}(V_1,V_{n+1})$ when $n\leq -2$.
Moreover, for $i\geq 1$, we define the following bilinear maps 
\begin{align*}
  [\cdot,\cdot   ]^{i+1}_n:V_{i+1}\times V_n\rightarrow V_{i+n+1},\quad 
  [\cdot,\cdot   ]^{-i-1}_n:V_{-i-1}\times V_n\rightarrow V_{-i+n-1}
\end{align*}
by:
\begin{align}
  [p_i(x_1\otimes z_i) ,w_n  ]_{n}^{i+1}:=& [x_1,  [z_i,w_n  ]_{n}^i ]_{i+n}^1- [z_i,  [x_1,w_n  ]_{n}^1 ]_{n+1}^i\label{def:bra_n+1}\\
&(x_1\in V_1, z_i\in V_i, w_n\in V_n)\notag
\end{align}and
\begin{align}
  [q_{-i}(y _{-1}\otimes \omega _{-i}) ,w_n  ]_{n}^{-i-1}:=& [y _{-1},  [\omega _{-i},w_n  ]_{n}^{-i} ]_{-i+n}^{-1}- [\omega _{-i},  [y _{-1},w_n  ]_{n}^{-1}]_{n-1}^{-i}\label{def:bra_-n-1}\\
&(y _{-1}\in V_{-1},\omega_{-i}\in V_{-i},w_n\in V_{n})\notag
\end{align}
inductively.
Then the bilinear maps (\ref{def:bra_n+1}) and (\ref{def:bra_-n-1}) are well-defined.
It can be shown by the same argument to the argument of \cite[Propositions 2.5 and 2.6]{Sa}.
Consequently, we can define a bilinear map $ [\cdot,\cdot ]^n_m:V_n\times V_m\rightarrow V_{n+m}$ for any $n,m\in \mathbb{Z}$.
\end{defn}

\begin{defn}\label{defn;bracket_stap}
For a standard pentad $(\mathfrak{g},\rho,V,{\cal V},B_0)$, we denote a direct sum of its $n$-graduations by $L(\mathfrak{g},\rho,V,{\cal V},B_0)$, i.e. 
\begin{align*}
L(\mathfrak{g},\rho,V,{\cal V},B_0):=\bigoplus _{n\in \mathbb{Z}}V_n.
\end{align*}
Moreover, we define a bilinear map $[\cdot,\cdot]:L(\mathfrak{g},\rho,V,{\cal V},B_0)\times L(\mathfrak{g},\rho,V,{\cal V},B_0)\rightarrow L(\mathfrak{g},\rho,V,{\cal V},B_0)$ by
\begin{align}
[x_n,y_m]:=[x_n,y_m]^n_m
\end{align}for any $n,m\in \mathbb{Z}$, $x_n\in V_n$ and $y_m\in V_m$.
\end{defn}
\begin{pr}\label{pr_lie_axiom}
This bilinear map $[\cdot,\cdot]$ satisfies the following equations
\begin{align}
&[x,y]+[y,x]=0,\\
&[x,[y,z]]+[z,[x,y]]+[y,[z,x]]=0
\end{align}for any $x,y,z\in L(\mathfrak{g},\rho,V,{\cal V},B_0)$.
\end{pr}
\begin{proof}
We can prove it by the same argument to the argument of \cite[Propositions 2.9 and 2.10]{Sa}.
\end{proof}
As a corollary of Proposition \ref{pr_lie_axiom}, we have the following theorem.
\begin{theo}[Lie algebra associated with a standard pentad]\label{StapLieAlgebraexists}
Let $(\mathfrak{g},\rho,V,{\cal V},B_0)$ be a standard pentad over a field $\mathfrak{k}$.
Then the vector space $L(\mathfrak{g},\rho,V,{\cal V},B_0)=\bigoplus _{n\in \mathbb{Z}}V_n$ is a graded Lie algebra with a bracket product $[\cdot ,\cdot]$ defined in Definition \ref{defn;bracket_stap}.
We call this graded Lie algebra {\rm the Lie algebra associated with $(\mathfrak{g},\rho,V,{\cal V},B_0)$}
(cf. \cite [Theorem 2.11]{Sa}).
\end{theo}
\begin{proof}
If the characteristic of $\mathfrak{k}$ is not $2$, our claim follows from {\rm Proposition} \ref{pr_lie_axiom} immediately.
\par Assume that the characteristic of $\mathfrak{k}$ is $2$.
To prove our claim, it is sufficient to show that the bilinear form $[\cdot ,\cdot]$ satisfies the equation 
\begin{align}
[x,x]=0\label{skeq_ch_two}
\end{align}for any $x\in L(\mathfrak{g},\rho,V,{\cal V},B_0)$.
Let us show the equation (\ref{skeq_ch_two}) for the cases where $x=x_n\in V_n$ for some $n\in \mathbb{Z}$.
For $n=0$, the restriction of $[\cdot,\cdot]:L(\mathfrak{g},\rho,V,{\cal V},B_0)\times L(\mathfrak{g},\rho,V,{\cal V},B_0)\rightarrow L(\mathfrak{g},\rho,V,{\cal V},B_0)$ to $V_0\times V_0$ coincides with the bracket product of $\mathfrak{g}$.
Thus we have $[x_0,x_0]=0$ for any $x_0\in V_0$.
For $n\geq 1$, by Proposition \ref{pr_lie_axiom}, we have
\begin{align}
&[[x_{n},x_{n}],\phi _{-1}]
=[[x_{n},\phi _{-1}],x_{n}]+[x_{n},[x_{n},\phi _{-1}]]=0
\end{align}for any $\phi _{-1}\in V_{-1}$ and $x_n\in V_n$.
Thus, we have the equation $[x_{n},x_{n}]=0$ in $V_{2n}=\mathrm {Hom }(V_{-1},V_{2n-1})$ for any $x_{n}\in V_{n}$.
Similarly, we have the same equation $[x_{n},x_{n}]=0$ for any $n\leq -1$ and $x_{n}\in V_{n}$.
Therefore, we have the equation (\ref{skeq_ch_two}) for the cases where $x=x_n\in V_n$ for some $n\in \mathbb{Z}$.
\par Now, we take an arbitrary element $x\in L(\mathfrak{g},\rho,V,{\cal V},B_0)$.
There exist integers $n_1,\ldots ,n_k$ and elements $x_{n_1}\in V_{n_1},\ldots ,x_{n_k}\in V_{n_k}$ such that $x=x_{n_1}+\cdots +x_{n_k}$.
Then we have
\begin{align*}
[x,x]
=[x_{n_1}+\cdots +x_{n_k},x_{n_1}+\cdots +x_{n_k}]
=\sum _{i=1}^k[x_{n_i},x_{n_i}]+\sum _{i\neq j}\biggl ( [x_{n_i},x_{n_j}]+[x_{n_j},x_{n_i}]\biggr )
=0.
\end{align*}
Thus we have  the equation (\ref{skeq_ch_two}) for any $x\in L(\mathfrak{g},\rho,V,{\cal V},B_0)$.
This completes the proof.
\end{proof}
\begin{remark}
Note that we can prove Theorem \ref {StapLieAlgebraexists} without the assumption that the bilinear form $B_0$ is symmetric.
\end{remark}
Note that $V_0=\mathfrak{g}$ and that the $V_0$-modules $V_0$, $V_1$, $V_{-1}$ are isomorphic to $\mathfrak{g}$, $V$, ${\cal V}$ respectively.
In this sense, we can say that the objects $\mathfrak{g}$, $(\rho, V)$, $(\varrho ,{\cal V})$ can be embedded into $L(\mathfrak{g},\rho,V,{\cal V},B_0)$.
\par In particular, when $\rho $ and $\varrho $ are faithful and surjective, we have a similar result on the structure of a graded Lie algebra of the form $L(\mathfrak{g},\rho,V,{\cal V},B_0)$ to the result which is obtained by H. Rubenthaler in \cite [Proposition 3.4.2]{arxiv}.
We can show the following proposition by Proposition \ref{defn;case_rho_sur} immediately.
\begin{pr}\label{pr;stap_transitive}
Let $(\mathfrak{g},\rho,V,{\cal V},B_0)$ be a standard pentad.
If both $\rho :\mathfrak{g}\otimes V\rightarrow V$ and $\varrho :\mathfrak{g}\otimes {\cal V}\rightarrow {\cal V}$ are faithful and surjective, then the graded Lie algebra $L(\mathfrak{g},\rho,V,{\cal V},B_0)$ is transitive.
\end{pr}

\subsection{Standard pentads with a symmetric bilinear form}\label{sec;standard_pentads_with_sym_bilinear}
In the previous section, we proved that for any standard pentad $(\mathfrak{g},\rho,V,{\cal V},B_0)$, there exists a graded Lie algebra such that $\mathfrak{g}$, $\rho$, $V$ and ${\cal V}$ can be embedded into it.
In this section, we discuss cases where $B_0$ is symmetric.
In these cases, we can also embed $B_0$ into $L(\mathfrak{g},\rho,V,{\cal V},B_0)$ and we can obtain some useful properties.
\begin{pr}\label{pr;stap_bilinear_exist}
Let $(\mathfrak{g},\rho,V,{\cal V},B_0)$ be a standard pentad such that $B_0$ is symmetric.
We define a symmetric bilinear form $B_L$ on $L(\mathfrak{g},\rho,V,{\cal V},B_0)$ inductively as follows:
\begin{align*}
\begin{cases}
B_L(a,b)=B_0(a,b),\\B_L(v,\phi )=\langle v,\phi \rangle ,\\ B_L(p_i(v_1\otimes u_i),q_{-i}(\phi _{-1}\otimes \psi _{-i}))=B_L(u_i,[q_{-i}(\phi _{-1}\otimes \psi _{-i}),v_1]),\\B_L(x_n,y_m)=0
\end{cases}
\end{align*}for any $a,b\in V_0$, $v\in V$, $\phi \in {\cal V}$, $i\geq 1$, $v_1\in V_1$, $\phi _{-1}\in V_{-1}$, $u_i\in V_i$, $\psi _{-i}\in V_{-i}$, $n,m\in \mathbb{Z}$, $n+m\neq 0$, $x_n\in V_n$ and $y_m\in V_m$.
Then $B_L$ is a non-degenerate symmetric invariant bilinear form on $L(\mathfrak{g},\rho,V,{\cal V},B_0)$
(cf. \cite[Proposition 3.2]{Sa}).
\end{pr}
\begin{proof}
Note that it is clear that the restriction of $B_L$ to $V_0\times V_0$ and $V_1\times V_{-1}$ is well-defined.
Let us show the well-definedness of $B_L$ on $V_2\times V_{-2}$.
For any $v_1,u_1\in V_1$ and $\phi _{-1},\psi _{-1}\in V_{-1}$, we have
\begin{align}
&B_L(u_1,[q_{-1}(\phi _{-1}\otimes \psi _{-1}),v_1]) =B_L(u_1,[[\phi _{-1},v_1],\psi _{-1}]+[\phi _{-1},[\psi _{-1},v_1]])\notag\\
&\quad =\langle u_1,[[\phi _{-1},v_1],\psi _{-1}]+[\phi _{-1},[\psi _{-1},v_1]]\rangle \notag\\
&\quad =B_0([\phi _{-1},v_1],\psi _{-1}(u_1))-B_0([\psi _{-1},v_1],\phi _{-1}(u_1)) \notag\\
&\quad =B_0([\phi _{-1},v_1],\psi _{-1}(u_1))-B_0(\phi _{-1}(u_1),[\psi _{-1},v_1]) \quad \text{(by the assumption that $B_0$ is symmetric)}\notag\\
&\quad =B_0([v_1,\phi _{-1}],u_1(\psi _{-1}))-B_0(u_1(\phi _{-1}),[v_1,\psi _{-1}]) \notag\\
&\quad =\langle [[v_1,\phi _{-1}],u_1]+[v_1,[u_1,\phi _{-1}]],\psi _{-1}\rangle \notag\\
&\quad =B_L([p_1(v_1\otimes u_1),\phi _{-1}],\psi _{-1}).
\end{align}
Thus, if $v_1^1,\ldots ,v_1^l,u_1^1,\ldots ,u_1^l\in V_1$ and $\phi _{-1}^1,\ldots ,\phi _{-1}^k, \psi _{-1}^1,\ldots ,\psi _{-1}^k\in V_{-1}$ satisfy the equation 
$$
\sum _{s=1}^lp_1(v_1^s\otimes u_1^s)=0,\quad \sum _{t=1}^kq_{-1}(\phi _{-1}^t\otimes \psi _{-1}^t)=0,
$$
then 
\begin{align*}
&\sum _{s=1}^lB_L(u_1^s,[q_{-1}(\phi _{-1}\otimes \psi _{-1}),v_1^s])=\sum _{s=1}^lB_L([p_1(v_1^s\otimes u_1^s),\phi _{-1}],\psi _{-1})=0,\\
 &\sum _{t=1}^kB_L(u_1,[q_{-1}(\phi _{-1}^t\otimes \psi _{-1}^t),v_1])=0
\end{align*}
for any $v_1,u_1\in V_1$ and $\phi _{-1},\psi _{-1}\in V_{-1}$, that is, we have the well-definedness of $B_L$ on $V_2\times V_{-2}$.
This $B_L\mid _{V_2\times V_{-2}}$ is $\mathfrak{g}$-invariant.
Moreover, by a similar argument, we have the well-definedness of $B_L$ on $V_i\times V_{-i}$ for each $i\geq 3$ by induction (see \cite[section 1.2]{Sa}).
Consequently, we can show the well-definedness of $B_L$ on the whole $L(\mathfrak{g},\rho,V,{\cal V},B_0)$ and that $B_L$ is non-degenerate symmetric invariant by the same argument as the argument in \cite[section 1.2 and Proposition 3.2]{Sa}.
\end{proof}
\begin{remark}
We need the assumption that $B_0$ is symmetric to show that the bilinear form $B_L$ is $L(\mathfrak{g},\rho,V,{\cal V},B_0)$-invariant.
Precisely, we need this assumption to show an equation 
\begin{align*}
B_L(v_1,[\phi _{-1},a])=B_L([v_1,\phi _{-1}],a)
\end{align*}for any $a\in V_0$, $v_1\in V_1$, $\phi _{-1}\in V_{-1}$.
\end{remark}
Under the assumption that $B_0$ is symmetric, the graded Lie algebra is characterized by the existence of such a bilinear form.
The following is a proposition concerning the ``universality'' and ``uniqueness'' of Lie algebras associated with a standard pentad with a symmetric bilinear form.

\begin{theo}\label{th;universality_stap}
Let $\mathfrak{L}=\bigoplus _{n\in\mathbb{Z}}\mathfrak{L}_n$ be a graded Lie algebra which has a non-degenerate symmetric invariant bilinear form $B_{\mathfrak{L}}$.
If $\mathfrak{L}$ and $B_{\mathfrak{L}}$ satisfy the following conditions, then a pentad $(\mathfrak{L}_0,{\rm ad},\mathfrak{L}_1,\mathfrak{L}_{-1},B_{\mathfrak{L}}\mid _{\mathfrak{L}_0\times \mathfrak{L}_0})$ is standard and $\mathfrak{L}$ is isomorphic to $L(\mathfrak{L}_0,{\rm ad},\mathfrak{L}_1,\mathfrak{L}_{-1},B_{\mathfrak{L}}\mid _{\mathfrak{L}_0\times \mathfrak{L}_0})$:
\begin{align}
&\text{$\mathfrak{L}_{i+1}=[\mathfrak{L}_1,\mathfrak{L}_i]$, $\mathfrak{L}_{-i-1}=[\mathfrak{L}_{-1},\mathfrak{L}_{-i}]$ for all $i\geq 1$,}\label{univ:condi_1}\\
&\text{the restriction of $B_{\mathfrak{L}}$ to $\mathfrak{L}_i\times \mathfrak{L}_{-i}$ is non-degenerate for any $i\geq 0$,}\label{univ:condi_2}
\end{align}
where $\mathrm{ad}$ stands for the adjoint representation of $\mathfrak{L}$ on itself
(cf. \cite [Proposition 3.3]{Sa}).
\end{theo}
\begin{proof}
First of all, let us check that the pentad $(\mathfrak{L}_0,{\rm ad},\mathfrak{L}_1,\mathfrak{L}_{-1},B_{\mathfrak{L}}\mid _{\mathfrak{L}_0\times \mathfrak{L}_0})$ is standard.
By (\ref{univ:condi_2}), we can obtain that $B_{\mathfrak{L}}\mid _{\mathfrak{L}_0\times \mathfrak{L}_0}$ is non-degenerate and that $\mathfrak{L}_1$ and $\mathfrak{L}_{-1}$ satisfy the condition (\ref{defn;condi_stap_pairing}).
It is easy to show that we can identify the restriction of the bracket product $[\cdot,\cdot]$ of $\mathfrak{L}$ to $\mathfrak{L}_{1}\times \mathfrak{L}_{-1}\rightarrow \mathfrak{L}_0$ with the $\Phi $-map of the pentad $(\mathfrak{L}_0,{\rm ad},\mathfrak{L}_1,\mathfrak{L}_{-1},B_{\mathfrak{L}}\mid _{\mathfrak{L}_0\times \mathfrak{L}_0})$.
Thus,  the condition  (\ref{defn;condi_stap_Phimap}) holds.
\par We denote the $n$-graduation of $(\mathfrak{L}_0,{\rm ad},\mathfrak{L}_1,\mathfrak{L}_{-1},B_{\mathfrak{L}}\mid _{\mathfrak{L}_0\times \mathfrak{L}_0})$ by $(\mathfrak{L})_n$ for any $n\in \mathbb{Z}$ and a bilinear form on $L(\mathfrak{L}_0,{\rm ad},\mathfrak{L}_1,\mathfrak{L}_{-1},B_{\mathfrak{L}}\mid _{\mathfrak{L}_0\times \mathfrak{L}_0})$ obtained in Proposition \ref{pr;stap_bilinear_exist} by $(B)_{\mathfrak{L}}$.
Let $\sigma _0:(\mathfrak{L})_0\rightarrow \mathfrak{L}_0$ and $\sigma _{\pm 1}:(\mathfrak{L})_{\pm 1}\rightarrow \mathfrak{L}_{\pm 1}$ be the identity maps respectively.
Then the linear maps $\sigma _0$ and $\sigma _{\pm 1}$ satisfy the following equations:
\begin{align}
&[\sigma _0(a),\sigma _{\pm 1}(x_{\pm 1})]=\sigma _{\pm 1}([a,x_{\pm 1}]),\label{th;universarity_stap_eq1}\\
&[\sigma _{1}(x_{1}),\sigma _{-1}(x_{-1})]=\sigma _0([x_1,x_{-1}])\label{th;universarity_stap_eq2}
\end{align}for any $a\in (\mathfrak{L})_0$ and $x_{\pm 1}\in (\mathfrak{L})_{\pm 1}$.
Indeed, the equation (\ref{th;universarity_stap_eq1}) is clear, and, we have
\begin{align}
&B_{\mathfrak{L}}(\sigma _0(b),[\sigma _{1}(x_{1}),\sigma _{-1}(x_{-1})])=B_{\mathfrak{L}}([\sigma _0(b),\sigma _{1}(x_{1})],\sigma _{-1}(x_{-1}))=B_{\mathfrak{L}}(\sigma _1([b,x_1]),\sigma _{-1}(x_{-1}))\notag\\
&\quad =(B)_{\mathfrak{L}}([b,x_1],x_{-1})=(B)_{\mathfrak{L}}(b,[x_1,x_{-1}])=B_{\mathfrak{L}}(\sigma _0(b),\sigma _0([x_1,x_{-1}]))
\end{align}for any $b\in (\mathfrak{L})_0$.
Thus, we can obtain the equation (\ref{th;universarity_stap_eq2}).
\par For each $i\geq 1$, we define linear maps $\sigma _{i+1}:(\mathfrak{L})_{i+1}\rightarrow \mathfrak{L}_{i+1}$ and $\sigma _{-i-1}:(\mathfrak{L})_{-i-1}\rightarrow \mathfrak{L}_{-i-1}$ by:
\begin{align}
&\sigma _{i+1}:p_i(x_1\otimes z_i)\mapsto [\sigma _1(x_1),\sigma _i(z_i)],\label{eq;thm_univ_stap1}\\
&\sigma _{-i-1}:q_{-i }(x_{-1}\otimes z_{-i})\mapsto [\sigma _{-1}(x_{-1}),\sigma _{-i}(z_{-i})]\label{eq;thm_univ_stap2}
\end{align}for any $x_{\pm 1}\in (\mathfrak{L})_{\pm 1}$ and $z_{\pm i}\in (\mathfrak{L})_{\pm i}$ inductively.
Note that it follows from (\ref{th;universarity_stap_eq1}) that the linear maps $\sigma _1$ and $\sigma _{-1}$ on $\rho (\mathfrak{g}\otimes V)$ and $\varrho (\mathfrak{g}\otimes {\cal V})$ defined by the same equations as (\ref{eq;thm_univ_stap1}) and (\ref{eq;thm_univ_stap2}) where $i=0$ coincide with the identity maps respectively.
We can prove that the linear maps $\sigma _n$ $(n\in \mathbb{Z})$ are well-defined and satisfy 
\begin{align}
&[\sigma _{0}(a), \sigma _n(z_n)]=\sigma _{n}([a,z_n]),\\
&[\sigma _{\pm 1}(x_{\pm 1}), \sigma _n(z_n)]=\sigma _{n\pm 1}([x_{\pm 1},z_n])
\end{align}for any $n\in\mathbb{Z}$, $a\in (\mathfrak{L})_{0}$, $x_{\pm 1}\in (\mathfrak{L})_{\pm 1}$ and $z_n\in (\mathfrak{L})_n$ by a similar argument to the argument of \cite[Proposition 3.3]{Sa}.
Then a linear map $\sigma :L(\mathfrak{L}_0,{\rm ad},\mathfrak{L}_1,\mathfrak{L}_{-1},B_{\mathfrak{L}}\mid _{\mathfrak{L}_0\times \mathfrak{L}_0})\rightarrow \mathfrak{L}$ defined by
\begin{align}
\sigma (z_n):=\sigma _n(z_n),
\end{align}where $n\in\mathbb{Z}$ and $z_n\in (\mathfrak{L})_n\subset L(\mathfrak{L}_0,{\rm ad},\mathfrak{L}_1,\mathfrak{L}_{-1},B_{\mathfrak{L}}\mid _{\mathfrak{L}_0\times \mathfrak{L}_0})$, is an isomorphism of Lie algebras.
We can also prove this by a similar argument to the argument of \cite[Proposition 3.3]{Sa}.
\end{proof}

As a corollary of Theorem \ref{th;universality_stap}, we can say that the theory of standard pentads is an extension of the theory of standard quadruplets.
\begin{pr}\label{pr;staq_is_stap}
Let $(\mathfrak{g},\rho,V,B_0)$ be a standard quadruplet (see \cite [Definition 1.9]{Sa}).
Then the Lie algebra $L(\mathfrak{g},\rho,V,B_0)$ associated with $(\mathfrak{g},\rho,V,B_0)$ (see \cite[Theorem 2.11]{Sa}) is isomorphic to the Lie algebra $L(\mathfrak{g},\rho,V,\mathrm {Hom }(V,\mathbb{C}),B_0)$.
\end{pr}

\begin{defn}\label{defn;quad_iso}
Let $(\mathfrak{g}^{1},\rho^{1},V^{1},{\cal V}^1,B_0^{1})$ and $(\mathfrak{g}^{2},\rho^{2},V^{2},{\cal V}^2,B_0^{2})$ be standard pentads.
We say that these pentads are {\it equivalent} if and only if there exists an isomorphism of Lie algebras $\tau:\mathfrak{g}^{1}\rightarrow \mathfrak{g}^{2}$, linear isomorphisms $\sigma:V^{1}\rightarrow V^{2}$, $\varsigma:{\cal V}^{1 }\rightarrow {\cal V}^{2}$ and a non-zero element $c\in\mathfrak{k}$ such that 
\begin{align}
&\sigma(\rho ^{1}(a^{1}\otimes x^{1}))=\rho^{2}(\tau (a^{1})\otimes \sigma ( x^{1}))\label{quad_iso_sigma_space},\\
&\varsigma(\varrho ^{1}(a^{1}\otimes y^{1}))=\varrho^{2}(\tau (a^{1})\otimes \varsigma (y^{1}))\label{quad_iso_sigma_space_circ},\\
&\langle x^1,y^1\rangle ^1=\langle \sigma (x^1),\varsigma (y^{1})\rangle ^2\label{tau_taucirc_compati},\\
&B_0^{1}(a^{1},b^{1} )=cB_0^{2}(\tau (a^{1}),\tau (b^{1} ))\label{quad_iso_sigma_bilinear}
\end{align}
where $a^{1},b^{1}\in\mathfrak{g}^{1}$, $x^{1}\in V^{1}$, $y^1\in {\cal V}^{1}$ and $\langle \cdot,\cdot\rangle ^i$ stands for the pairing between $V^i$ and ${\cal V}^{ i}$ $(i=1,2)$.
We denote this equivalence relation by
\begin{align}
(\mathfrak{g}^{1},\rho^{1},V^{1},{\cal V}^{ 1},B_0^{1})\simeq(\mathfrak{g}^{2},\rho^{2},V^{2},{\cal V}^{ 2},B_0^{2}).
\end{align}
\end{defn}
\begin{remark}
Note that if $V$ is finite-dimensional, then linear isomorphisms $\tau $, $\sigma $ satisfying (\ref{quad_iso_sigma_space}) induce a linear isomorphism from ${\cal V}^{1}=\mathrm{Hom}(V^1,\mathfrak{k})$ to ${\cal V}^{2}=\mathrm{Hom}(V^2,\mathfrak{k})$ satisfying (\ref{quad_iso_sigma_space_circ}) and (\ref{tau_taucirc_compati}).
\end{remark}
\begin{pr}\label{pr;stapequi_lieiso}
If standard pentads $(\mathfrak{g}^{1},\rho^{1},V^{1},{\cal V}^1,B_0^{1})$ and $(\mathfrak{g}^{2},\rho^{2},V^{2},{\cal V}^2,B_0^{2})$ are equivalent, then the Lie algebras associated with them are isomorphic, i.e. we have
\begin{align}
L(\mathfrak{g}^{1},\rho^{1},V^{1},{\cal V}^1,B_0^{1})\simeq L(\mathfrak{g}^{2},\rho^{2},V^{2},{\cal V}^2,B_0^{2})
\end{align}
(cf. \cite [Proposition 3.6]{Sa}).
\end{pr}
\begin{proof}
We denote the $n$-graduation of $(\mathfrak{g}^{i},\rho^{i},V^{i},{\cal V}^i,B_0^{i})$ by $V_n^i$ $(i=1,2)$ for all $n\in\mathbb{Z}$ and the bilinear forms on $L(\mathfrak{g}^{i},\rho^{i},V^{i},{\cal V}^i,B_0^{i})$ defined in Proposition \ref{pr;stap_bilinear_exist} by $B_L^i$ $(i=1,2)$.
Under the notation of Definition \ref{defn;quad_iso}, we define linear maps $\sigma _0:=\tau :V_0^1\rightarrow V_0^2$, $\sigma _1:=\frac{1}{c}\sigma :V_1^1\rightarrow V_1^2$ and $\sigma _{-1}:=\varsigma :V_{-1}^1\rightarrow V_{-1}^2$.
Then, these linear maps $\sigma _0$ and $\sigma _{\pm 1}$ satisfy the same equations as (\ref{th;universarity_stap_eq1}) and (\ref{th;universarity_stap_eq2}).
In fact, the equation (\ref{th;universarity_stap_eq1}) is clear, and, we have
\begin{align*}
&B_0^2(\sigma _0(a^1_0),[\sigma _1(x^1_1),\sigma _{-1}(y^1_{-1})])=B_L^2(\sigma _1([a^1_0,x^1_1]),\sigma _{-1}(y^1_{-1})) \notag\\
&\quad =\frac{1}{c}B_L^1([a^1_0,x^1_1],y^1_{-1})= \frac{1}{c}B_0^1(a^1_0,[x^1_1,y^1_{-1}])=B_0^2(\sigma _0(a^1_0),\sigma _0([x^1_1,y^1_{-1}]))
\end{align*}
for any $a^1_0\in V_0^1$, $x^1_1\in V_1^1$ and $y^1_{-1}\in V_{-1}^1$.
Thus, we have the equation (\ref{th;universarity_stap_eq2}).
Then, by the same argument as the argument in proof of Theorem \ref{th;universality_stap}, we can construct an isomorphism of Lie algebras from $L(\mathfrak{g}^{1},\rho^{1},V^{1},{\cal V}^1,B_0^{1})$ to $L(\mathfrak{g}^{2},\rho^{2},V^{2},{\cal V}^2,B_0^{2})$.
\end{proof}
\begin{remark}
The converse of Proposition \ref{defn;quad_iso} is not true.
In fact, we have an example of two non-equivalent pentads such that the corresponding Lie algebras are isomorphic (see \cite[pp. 398--399]{Sa}).
\end{remark}

\begin{defn}
Let $(\mathfrak{g}^{1},\rho^{1},V^{1},{\cal V}^1,B_0^{1})$ and $(\mathfrak{g}^{2},\rho^{2},V^{2},{\cal V}^2,B_0^{2})$ be standard pentads.
Let $\rho ^{1}\boxplus\rho ^{2}$ and $\varrho ^{1}\boxplus\varrho ^{2}$ be representations of $\mathfrak{g} ^{1}\oplus \mathfrak{g} ^{2}$ on $V ^{1}\oplus V ^{2}$ and ${\cal V}^1\oplus {\cal V}^2$ defined by:
\begin{align*}
& (\rho ^{1}\boxplus \rho ^{2})((a ^{1},a ^{2})  \otimes (v ^{1},v ^{2})):=(\rho ^{1}(a ^{1}\otimes v ^{1}), \rho ^{2}(a ^{2}\otimes v ^{2})),\\
& (\varrho ^{1}\boxplus \varrho ^{2})((b ^{1},b ^{2})  \otimes (\phi ^{1},\phi ^{2})):=(\varrho ^{1}(b ^{1}\otimes \phi ^{1}), \varrho ^{2}(b ^{2}\otimes \phi ^{2}))
\end{align*}
where $a ^{i}, b^i\in \mathfrak{g} ^{i}$, $v ^{i}\in V ^{i}$, $\phi  ^{i}\in {\cal V} ^{i}$ $(i=1,2)$.
Let $B_0 ^{1}\oplus B_0 ^ {2}$ be a bilinear form on $\mathfrak{g} ^{1}\oplus \mathfrak{g} ^{2}$ defined by:
\begin{align}
(B_0 ^{1}\oplus B_0^ {2})((a ^{1}, a ^{2}),(b ^{1},b ^{2})):=B_0 ^{1}(a ^{1},b ^{1})+B_0 ^{2}(a ^{2},b ^{2})
\end{align}
where $a ^{i}, b^{i}\in\mathfrak{g} ^{i}$ $(i=1,2)$.
Then, clearly, a pentad $(\mathfrak{g} ^{1}\oplus \mathfrak{g} ^{2},\rho ^{1}\boxplus \rho ^{2}, V ^{1}\oplus V ^{2}, {\cal V}^1\oplus {\cal V}^2,B_0 ^{1}\oplus B_0 ^{2})$ is also a standard pentad.
We call it a ${\it direct\ sum}$ of $(\mathfrak{g}^{1},\rho^{1},V^{1},{\cal V}^1,B_0^{1})$ and $(\mathfrak{g}^{2},\rho^{2},V^{2},{\cal V}^2,B_0^{2})$ and denote it by $(\mathfrak{g}^{1},\rho^{1},V^{1},{\cal V}^1,B_0^{1})\oplus(\mathfrak{g}^{2},\rho^{2},V^{2},{\cal V}^2,B_0^{2})$.
\end{defn}
\begin{pr}
Let $(\mathfrak{g}^{1},\rho^{1},V^{1},{\cal V}^1,B_0^{1})$ and $(\mathfrak{g}^{2},\rho^{2},V^{2},{\cal V}^2,B_0^{2})$ be standard pentads.
Then the Lie algebra $L((\mathfrak{g}^{1},\rho^{1},V^{1},{\cal V}^1,B_0^{1})\oplus(\mathfrak{g}^{2},\rho^{2},V^{2},{\cal V}^2,B_0^{2}))$ is isomorphic to $L(\mathfrak{g}^{1},\rho^{1},V^{1},{\cal V}^1,B_0^{1})\oplus L(\mathfrak{g}^{2},\rho^{2},V^{2},{\cal V}^2,B_0^{2})$
(cf. \cite [Proposition 3.9]{Sa}).
\end{pr}
\begin{proof}
We retain to use the notation of Proposition \ref{pr;stapequi_lieiso}.
Then, we have the following $\mathbb{Z}$-grading of $L(\mathfrak{g}^{1},\rho^{1},V^{1},{\cal V}^1,B_0^{1})\oplus L(\mathfrak{g}^{2},\rho^{2},V^{2},{\cal V}^2,B_0^{2})$:
\begin{align}
L(\mathfrak{g}^{1},\rho^{1},V^{1},{\cal V}^1,B_0^{1})\oplus L(\mathfrak{g}^{2},\rho^{2},V^{2},{\cal V}^2,B_0^{2})=\bigoplus _{n\in\mathbb{Z}}(V_n^1\oplus V_n^2).
\end{align}
By Theorem \ref{th;universality_stap}, we have our claim.
\end{proof}

\begin{defn}
Let $(\mathfrak{g},\rho,V,{\cal V},B_0)$ be a standard pentad.
We say that $(\mathfrak{g},\rho,V,{\cal V},B_0)$ is {\it decomposable} if and only if there exist standard pentads $(\mathfrak{a},\rho _{\mathfrak{a}},V_{\mathfrak{a}},{\cal V}_{\mathfrak{a}},B_{0,\mathfrak{a}})$ and $(\mathfrak{b},\rho _{\mathfrak{b}},V_{\mathfrak{b}},{\cal V}_{\mathfrak{b}},B_{0,\mathfrak{b}})$ such that 
\begin{align}
&(\dim \mathfrak{a}+\dim V_{\mathfrak{a}})(\dim \mathfrak{b}+\dim V_{\mathfrak{b}})\neq 0,\\
&(\mathfrak{g},\rho,V,{\cal V},B_0)\simeq (\mathfrak{a},\rho _{\mathfrak{a}},V_{\mathfrak{a}},{\cal V}_{\mathfrak{a}},B_{0,\mathfrak{a}})\oplus (\mathfrak{b},\rho _{\mathfrak{b}},V_{\mathfrak{b}},{\cal V}_{\mathfrak{b}},B_{0,\mathfrak{b}}).
\end{align}
If $(\mathfrak{g},\rho,V,{\cal V},B_0)$ is not decomposable, we say that $(\mathfrak{g},\rho,V,{\cal V},B_0)$ is {\it indecomposable}.
\end{defn}

\begin{defn}\label{defn;stap_irr}
Let $(\mathfrak{g},\rho,V,{\cal V},B_0)$ be a standard pentad.
We say that $(\mathfrak{g},\rho,V,{\cal V},B_0)$ is {\it reducible} if and only if there exist an ideal $\mathfrak{a}$ of $\mathfrak{g}$ and $\mathfrak{g}$-submodules $V_{\mathfrak{a}}$ and ${\cal V}_{\mathfrak{a}}$ of $V$ and ${\cal V}$ satisfying that:
\begin{align}
&\text{$\{0\}\neq {\cal V}_{\mathfrak{a}}\oplus \mathfrak{a}\oplus V_{\mathfrak{a}}\subsetneq {\cal V}\oplus \mathfrak{a}\oplus V$,}\label{defn;stap_irr_condi_1}\\
&\text{$\rho (\mathfrak{a}\otimes V), \rho (\mathfrak{g}\otimes V_{\mathfrak{a}}) \subset V_{\mathfrak{a}}$ and $\varrho (\mathfrak{a}\otimes {\cal V}), \varrho (\mathfrak{g}\otimes {\cal V}_{\mathfrak{a}})\subset {\cal V}_{\mathfrak{a}}$,}\label{defn;stap_irr_condi_2}\\
&\text{$\Phi _{\rho }(V_{\mathfrak{a}}\otimes {\cal V}),\Phi _{\rho }(V\otimes {\cal V}_{\mathfrak{a}})\subset \mathfrak{a}$.}\label{defn;stap_irr_condi_3}
\end{align}
And, we say that $(\mathfrak{g},\rho,V,{\cal V},B_0)$ is {\it irreducible} if and only if it is not reducible.
\end{defn}
\begin{remark}
If a standard pentad is irreducible, then it is indecomposable.
\end{remark}
\begin{pr}\label{pr;irr_imply_sur}
Let $(\mathfrak{g},\rho,V,{\cal V},B_0)$ be an irreducible standard pentad.
Then the representations $\rho :\mathfrak{g}\otimes V\rightarrow V$, $\varrho :\mathfrak{g}\otimes {\cal V}\rightarrow {\cal V}$ and the $\Phi $-map $\Phi _{\rho }:V\otimes {\cal V}\rightarrow \mathfrak{g}$ are surjective.
\end{pr}
\begin{proof}
If $\varrho (\mathfrak{g}\otimes {\cal V})\oplus \Phi _{\rho }(V\otimes {\cal V})\oplus \rho (\mathfrak{g}\otimes V)=\{0\}$, it follows that $\dim {\cal V}=\dim \mathfrak{g}=\dim V=0$ from the assumption that $(\mathfrak{g},\rho,V,{\cal V},B_0)$ is irreducible.
In particular, we have $\varrho (\mathfrak{g}\otimes {\cal V})={\cal V}=\{0\}$ and $\rho (\mathfrak{g}\otimes V)=V=\{0\}$.
If $\varrho (\mathfrak{g}\otimes {\cal V})\oplus \Phi _{\rho }(V\otimes {\cal V})\oplus \rho (\mathfrak{g}\otimes V)\neq \{0\}$, since it satisfies the conditions (\ref{defn;stap_irr_condi_2}) and (\ref{defn;stap_irr_condi_3}), we have $\varrho (\mathfrak{g}\otimes {\cal V})\oplus \Phi _{\rho }(V\otimes {\cal V})\oplus \rho (\mathfrak{g}\otimes V)={\cal V}\oplus \mathfrak{g}\oplus V$.
\end{proof}

\begin{pr}\label{indec_ker_V_1}
Let $(\mathfrak{g},\rho,V,{\cal V},B_0)$ be an irreducible standard pentad whose representation $\rho $ is faithful and denote the Lie algebra associated with it by $L(\mathfrak{g},\rho,V,{\cal V},B_0)=\bigoplus _{n\in\mathbb{Z}}V_n$.
Let $N$ (respectively $M$) be an integer such that $V _{N+1}$ is not $\{0\}$ (respectively $V_{-M-1}$ is not $\{0\}$).
Then for any non-zero element $z_N\in V_N$ (respectively $\omega _{-M}\in V_{-M}$), there exists an element $x_1\in V_1$ such that $[x_1,z_N]\neq 0$ (respectively $y_{-1}\in V_{-1}$ such that $[y_{-1},\omega _{-M}]\neq 0$) (cf. \cite [Proposition 3.11]{Sa}).
\end{pr}
\begin{proof}
When $N\leq -1$, we have our claim by Propositions \ref{defn;case_rho_sur}, \ref{pr;stap_transitive} and \ref{pr;irr_imply_sur}.
When $N=0$, we have our claim by the assumption that $\rho $ is faithful.
Assume that $N\geq 1$ and $V_{N+1}\neq \{0\}$ and put $\mathfrak{a}_N:=\{a_N\in V_N\mid [x_1,a_N]=0\text{ for any $x_1\in V_{1}$}\}$ and $\mathfrak{a}_n:=\{a_n\in V_n\mid [x_1,a_n]\in \mathfrak{a}_{n+1}\text{ for any $x_1\in V_{1}$}\}$ for $n\leq N-1$ inductively.
Then $\mathfrak{a}_n$ is a $V_0$-submodule of $V_n$ for each $n$, i.e. $[V_0,\mathfrak{a}_n]\subset \mathfrak{a}_n$, and, we have that $[V_{\pm 1},\mathfrak{a}_n]\subset \mathfrak{a}_{n\pm 1}$ for any $n\in\mathbb{Z}$ (see \cite[the proof of Proposition 3.11]{Sa}).
In particular, $\mathfrak{a}_{-1}\oplus \mathfrak{a}_0\oplus \mathfrak{a}_1$ satisfies the conditions (\ref{defn;stap_irr_condi_2}) and (\ref{defn;stap_irr_condi_3}).
If $\mathfrak{a}_{-1}\oplus \mathfrak{a}_0\oplus \mathfrak{a}_1={\cal V}\oplus \mathfrak{g}\oplus V$, then we have $\mathfrak{a}_N=V_N$ and a contradiction to the assumption that $V_{N+1}\neq \{0\}$.
Thus we have $\mathfrak{a}_1=\{0\}$, and, thus, $\mathfrak{a}_2=\{0\},\ldots ,\mathfrak{a}_{N}=\{0\}$ by the transitivity of $L(\mathfrak{g},\rho,V,{\cal V},B_0)$.
Similarly, we have our result for $M$ such that $V_{-M-1}\neq \{0\}$.
\end{proof}
\begin{pr}\label{L_simple}
Let $(\mathfrak{g},\rho,V,{\cal V},B_0)$ be an irreducible standard pentad whose representation $\rho $ is faithful.
If the Lie algebra $L(\mathfrak{g},\rho,V,{\cal V},B_0)$ is finite-dimensional, then $L(\mathfrak{g},\rho,V,{\cal V},B_0)$ is simple
(cf. \cite [Proposition 3.12]{Sa}).
Moreover, if $(\mathfrak{g},\rho,V,{\cal V},B_0)$ is defined over $\mathbb{C}$ and $L(\mathfrak{g},\rho,V,{\cal V},B_0)$ is a finite-dimensional simple Lie algebra, then a triplet $(\mathfrak{g},\rho,V)$ corresponds to some prehomogeneous vector space of parabolic type (see \cite [Theorem 3.13]{Sa}).
\end{pr}
\begin{proof}
We can show this by Proposition \ref{indec_ker_V_1} and the same argument to the argument of \cite[Proposition 3.12 and Theorem 3.13]{Sa}.
\end{proof}
A prehomogeneous vector space of parabolic type (abbrev. a PV of parabolic type) is a PV which can be obtained from a $\mathbb{Z}$-graded finite-dimensional semisimple Lie algebra.
PVs of parabolic type are classified by H. Rubenthaler (see \cite {ru-1,ru-2,ru-3}).
\begin{ex}\label{ex;sa_sl}
Let $m\geq 2$ and $\mathfrak{g}=\mathfrak{gl}_1(\mathbb{C})\oplus \mathfrak{sl}_m(\mathbb{C})$, $\rho=\Lambda _1$ a representation of $\mathfrak{g}$ on $\mathbb{C}^m$ defined by 
$$
\Lambda _1 ((a,A)\otimes v):=av+Av\quad (a\in \mathfrak{gl}_1, A\in \mathfrak{sl}_m, v\in V),
$$
$B_0=\kappa _m$ a bilinear form on $\mathfrak{g}$ defined by
$$
\kappa _m((a,A),(a^{\prime },A^{\prime })):=\frac{m}{m+1}aa^{\prime }+\mathrm {Tr}(AA^{\prime })\quad (a,a^{\prime }\in \mathfrak{gl}_1,A,A^{\prime }\in \mathfrak{sl}_m).
$$
Then, a pentad $(\mathfrak{g},\rho,V,\mathrm {Hom}(V,\mathbb{C}),B_0)=(\mathfrak{gl}_1\oplus \mathfrak{sl}_m,\Lambda _1,\mathbb{C}^m,\mathbb{C}^m,\kappa _m)$ is a standard pentad which has a $(m^2+2m)$-dimensional graded simple Lie algebra $L(\mathfrak{gl}_1\oplus \mathfrak{sl}_m,\Lambda _1,\mathbb{C}^m,\mathbb{C}^m,\kappa _m)=V_{-1}\oplus V_0\oplus V_1$ (see \cite [Example 1.14]{Sa}).
This Lie algebra $L(\mathfrak{gl}_1\oplus \mathfrak{sl}_m,\Lambda _1,\mathbb{C}^m,\mathbb{C}^m,\kappa _m)$ is isomorphic to $\mathfrak{sl}_{m+1}$.
Indeed, from the classification of PVs of parabolic type (see \cite {ru-1,ru-2,ru-3}) and the dimension of  $L(\mathfrak{g},\rho,V,\mathrm {Hom}(V,\mathbb{C}),B_0)$, it is isomorphic to $\mathfrak{sl}_{m+1}$.
\end{ex}

\begin{ex}\label{ex;1}
Put $\mathfrak{g}:=\mathfrak{gl}_1(\mathbb{C})\oplus \mathfrak{gl}_1(\mathbb{C})\oplus \mathfrak{sl}_2(\mathbb{C})$,  $V:=\mathbb{C}^2=M(2,1;\mathbb{C})$, ${\cal V}:=\mathbb{C}^2$ and define representations $\rho :\mathfrak{g}\otimes V\rightarrow V$, $\varrho :\mathfrak{g}\otimes {\cal V}\rightarrow {\cal V}$ by:
\begin{align*}
\rho ((a,b,A)\otimes v):=bv+Av, \quad \varrho ((a,b,A)\otimes \phi ):=-b\phi -{}^tA\phi 
\end{align*}
for any $(a,b,A)\in \mathfrak{g}$, $v\in V$, $\phi \in {\cal V}$.
We can identify ${\cal V}$ with $\mathrm {Hom }(V,\mathbb{C})$ via the following bilinear map $\langle \cdot,\cdot\rangle _V:V\times {\cal V}\rightarrow \mathbb{C}$ defined by:
\begin{align*}
\langle v,\phi \rangle _V:={}^t v\phi .
\end{align*}
Let $B_0$ be a bilinear form on $\mathfrak{g}$ defined by:
\begin{align*}
&B_0((a,b,A),(a^{\prime },b^{\prime},A^{\prime })):=\frac{3}{4}aa^{\prime }+bb^{\prime}+\frac{1}{2}(ab^{\prime }+a^{\prime }b)+ \mathrm{Tr}(AA^{\prime}).
\end{align*}
Then, a pentad $(\mathfrak{g},\rho,V, {\cal V},B_0)$ is a standard pentad whose $\Phi $-map is given by:
\begin{align*}
\Phi _{\rho }(v\otimes \phi )=(-{}^tv\phi ,\frac{3}{2}{}^tv\phi ,v{}^t\phi -\frac{1}{2}{}^tv\phi I_2).
\end{align*}
The Lie algebra $L(\mathfrak{g},\rho,V, {\cal V},B_0)$ is isomorphic to $\mathfrak{gl }_1\oplus \mathfrak{sl}_3$.
Indeed, if we put $\mathfrak{g}_V^1:=\mathbb{C}\cdot (1,0,O_2)$, $\mathfrak{g}_V^2:=\mathbb{C}\cdot (-\frac{2}{3},1,O_2)\oplus \mathfrak{sl}_2$, then we have
\begin{align*}
L(\mathfrak{g},\rho,V, {\cal V},B_0)&\simeq  L((\mathfrak{g}_V^1,\rho \mid _{\mathfrak{g}_V^1},\{0\}, \{0\},B_0\mid _{\mathfrak{g}_V^1\times \mathfrak{g}_V^1})\oplus (\mathfrak{g}_V^2,\rho \mid _{\mathfrak{g}_V^2},V, {\cal V},B_0\mid _{\mathfrak{g}_V^2\times \mathfrak{g}_V^2}))\\
&\simeq \mathfrak{g}_V^1\oplus L(\mathfrak{g}_V^2,\rho \mid _{\mathfrak{g}_V^2},V, {\cal V},B_0\mid _{\mathfrak{g}_V^2\times \mathfrak{g}_V^2})\simeq \mathfrak{g}_V^1\oplus {\cal V}\oplus \mathfrak{g}_V^2\oplus V\\
&\simeq \mathfrak{gl}_1\oplus \mathfrak{sl}_3
\end{align*}
from Example \ref{ex;sa_sl}.
Moreover, under this identification, the bilinear form $B_L$ on $ L(\mathfrak{g}_V^2,\rho \mid _{\mathfrak{g}_V^2},V, {\cal V},B_0\mid _{\mathfrak{g}_V^2\times \mathfrak{g}_V^2})$ is given by $B_L(\hat{A},\hat{A}^{\prime})=\mathrm {Tr}(\hat{A}\hat{A}^{\prime})$ $(\hat{A},\hat{A}^{\prime}\in \mathfrak{sl}_3)$.
In fact, if we put 
$$
h:=(0,0,\begin{pmatrix}1&0\\0&-1\end{pmatrix})\in \mathfrak{g}_V^2,
$$
then $B_0(h,h)=2$.
On the other hand, we can obtain $\mathrm {Tr}(\mathrm {ad}\ h\ \mathrm {ad}\ h)=12$, where $\mathrm {ad}$ stands for the adjoint representation of $ L(\mathfrak{g}_V^2,\rho \mid _{\mathfrak{g}_V^2},V, {\cal V},B_0\mid _{\mathfrak{g}_V^2\times \mathfrak{g}_V^2})$, by a direct calculation.
Since any non-degenerate invariant bilinear form on $\mathfrak{sl}_3$ is a scalar multiple of the Killing form, we can obtain that $B_L$ is $1/6$ times the Killing form of $\mathfrak{sl}_3$, i.e. $B_L(\hat{A},\hat{A}^{\prime})=\mathrm {Tr}(\hat{A}\hat{A}^{\prime})$.
\end{ex}

\begin{pr}\label{pr;L_irrgraded_alg}
Let $(\mathfrak{g},\rho,V,{\cal V},B_0)$ be a standard pentad whose representation $\rho $ is faithful.
Under this assumption, the pentad $(\mathfrak{g},\rho,V,{\cal V},B_0)$ is irreducible if and only if the Lie algebra $L(\mathfrak{g},\rho,V,{\cal V},B_0)$ does not have a non-zero proper graded ideal.
\end{pr}
\begin{proof}
Assume that $(\mathfrak{g},\rho,V,{\cal V},B_0)$ is reducible.
Under the notation of Definition \ref{defn;stap_irr}, we put $\mathfrak{a}_{-1}:={\cal V}_{\mathfrak{a}}$, $\mathfrak{a}_0:=\mathfrak{a}$, $\mathfrak{a}_1:=V_{\mathfrak{a}}$.
Moreover, we put $\mathfrak{a}_n:=[V_1,\mathfrak{a}_{n-1}]$ for all $n\geq 2$ and $\mathfrak{a}_m:=[V_{-1},\mathfrak{a}_{m+1}]$ for all $m\leq -2$ inductively.
Then a direct sum $\mathfrak{A}:=\bigoplus _{n\in\mathbb{Z}}\mathfrak{a}_n$ is a non-zero proper graded ideal of $L(\mathfrak{g},\rho,V,{\cal V},B_0)$.
In fact, by the assumption that $[V_i,\mathfrak{a}_j]\subset \mathfrak{a}_{i+j}$ for any $-1\leq i,j,i+j\leq 1$, we can easily show that $[V_{0},\mathfrak{A}],[V_{\pm 1},\mathfrak{A}]\subset \mathfrak{A}$ by induction.
Since $L(\mathfrak{g},\rho,V,{\cal V},B_0)$ is generated by $V_0$ and $V_{\pm 1}$, we have $[L(\mathfrak{g},\rho,V,{\cal V},B_0),\mathfrak{A}]\subset \mathfrak{A}$.
Thus, $\mathfrak{A}$ is a graded ideal.
Since $\{0\}\neq \mathfrak{a}_{-1}\oplus \mathfrak{a}_0\oplus \mathfrak{a}_1\subsetneq {\cal V}\oplus \mathfrak{g}\oplus V$, we have $\{0\}\neq \mathfrak{A}\subsetneq L(\mathfrak{g},\rho,V,{\cal V},B_0)$.
\par Conversely, assume that $(\mathfrak{g},\rho,V,{\cal V},B_0)$ is irreducible.
Let $\mathfrak{b}=\sum _{n\in\mathbb{Z}}(\mathfrak{b}\cap V_n)$ be a non-zero graded ideal of $L(\mathfrak{g},\rho,V,{\cal V},B_0)$ and put $\mathfrak{b}_n:=\mathfrak{b}\cap V_n$.
Then, by {\rm Proposition} \ref{indec_ker_V_1}, we can obtain that $\mathfrak{b}_0\neq \{0\}$.
In fact, since $\mathfrak{b}\neq \{0\}$, there exists an integer $n\in \mathbb{Z}$ and a non-zero element $z_n\in \mathfrak{b}_n$.
For example, if $n\geq 1$, then there exist $n$ elements $y_{-1}^{1},\ldots ,y_{-1}^n\in V_{-1}$ such that $[y_{-1}^n,[\cdots [y_{-1}^1,z_n]\cdots ]]\in \mathfrak{b}_0\setminus \{0\}$.
Since $\mathfrak{b}_{-1}\oplus \mathfrak{b}_0\oplus \mathfrak{b}_1$ satisfies the conditions (\ref{defn;stap_irr_condi_2}) and (\ref{defn;stap_irr_condi_3}), it coincides with $V_{-1}\oplus V_0\oplus V_1$, and, thus, $\mathfrak{b}=L(\mathfrak{g},\rho,V,{\cal V},B_0)$.
\end{proof}
The following lemmas are to construct a derivation on $L(\mathfrak{g},\rho,V,{\cal V},B_0)$.
They are used in Theorem \ref{theo;chainrule}.
\begin{lemma}\label{lem;derivationext}
Let $(\mathfrak{g},\rho,V,{\cal V},B_0)$ be a standard pentad, $L(\mathfrak{g},\rho,V,{\cal V},B_0)=\bigoplus _{n\in\mathbb{Z}}V_n$ be the Lie algebra associated with it.
Let $\alpha _i:V_i\rightarrow V_i$ $(i=0,\pm 1)$ be linear maps which satisfy
\begin{align}
\alpha _{i+j}([a_i,b_j])=[\alpha _i(a_i),b_j]+[a_i,\alpha _j(b_j)]\label{der_eq1}
\end{align}
for any $-1\leq i,j,i+j\leq 1$ and elements $a_i\in V_i$, $b_j\in V_j$.
Then, there exists a linear map $\alpha :L(\mathfrak{g},\rho,V,{\cal V},B_0)\rightarrow L(\mathfrak{g},\rho,V,{\cal V},B_0)$ such that $\alpha $ is a derivation on $L(\mathfrak{g},\rho,V,{\cal V},B_0)$ and its restriction to $V_i$ $(i=0,\pm 1)$ coincides with $\alpha _i$.
\end{lemma}
\begin{proof}
First, let us construct linear maps $\alpha _i:V_i\rightarrow V_i$ for all $i\in \mathbb{Z}$ by induction.
Let $i\geq 1$ and assume that the integer $i$ satisfies the condition that we have linear maps $\alpha _j:V_j\rightarrow V_j$ for all $0\leq j\leq i$ which satisfy the following equations:
\begin{align*}
&\alpha _{j}([a_0,b_j])=[\alpha _{0}(a_0),b_j]+[a_0,\alpha _j(b_j)],\\
&\alpha _{j}([x_1,b_{j-1}])=[\alpha _{1}(x_1),b_{j-1}]+[x_1,\alpha _{j-1}(b_{j-1})],\\
&\alpha _{j-1}([y_{-1},b_j])=[\alpha _{-1}(y_{-1}),b_j]+[y_{-1},\alpha _j(b_j)]
\end{align*}
for any $0\leq j\leq i$, $a_0\in V_0$, $x_1\in V_1$, $y_{-1}\in V_{-1}$, $b_j\in V_j$ and $b_{j-1}\in V_{j-1}$.
By the assumption (\ref{der_eq1}), when $i=1$ the given linear maps $\alpha _0$, $\alpha _{\pm 1}$ satisfy these equations.
Then we define a linear map $\alpha _{i+1}:V_{i+1}\rightarrow V_{i+1}$ by:
\begin{align}
\alpha _{i+1}([x_1,b_i]):=[\alpha _1(x_1),b_i]+[x_1,\alpha _i(b_i)]
\end{align}for any $x_1\in V_1$ and $b_i\in V_i$.
Let us check the well-definedness of $\alpha _{i+1}$.
In fact, for any $y_{-1}\in V_{-1}$, $x_1\in V_1$ and $b_i\in V_i$, we have
\begin{align}
&[y_{-1},[\alpha _1(x_1),b_i]+[x_1,\alpha _i(b_i)]]=[y_{-1},[\alpha _1(x_1),b_i]]+[y_{-1},[x_1,\alpha _i(b_i)]]\notag\\
&\quad =[[y_{-1},\alpha _1(x_1)],b_i]+[\alpha _1(x_1),[y_{-1},b_i]]+[[y_{-1},x_1],\alpha _i(b_i)]+[x_1,[y_{-1},\alpha _i(b_i)]]\notag\\
&\quad =[\alpha _0([y_{-1},x_1]),b_i]-[[\alpha _{-1}(y_{-1}),x_1],b_i]+[\alpha _1(x_1),[y_{-1},b_i]]\notag\\
&\qquad+[[y_{-1},x_1],\alpha _i(b_i)]+[x_1,\alpha _{i-1}([y_{-1},b_i])]-[x_1,[\alpha _{-1}(y_{-1}),b_i]]\notag\\
&\quad =\alpha _i([[y_{-1},x_1],b_i])+\alpha _i([x_1,[y_{-1},b_i]])-[\alpha _{-1}(y_{-1}),[x_1,b_i]]\notag\\
&\quad =\alpha _i([y_{-1},[x_1,b_i]])-[\alpha _{-1}(y_{-1}),[x_1,b_i]].\label{proof;alphader_welldef}
\end{align}
Thus, if $x_1^1,\ldots ,x_1^l\in V_1$ and $b_i^1,\ldots ,b_i^l\in V_i$ satisfy $\sum _{s=1}^l [x_1^s,b_i^s]=0$, then we have
\begin{align*}
\sum _{s=1}^l[y_{-1},[\alpha _1(x_1^s),b_i^s]+[x_1^s,\alpha _i(b_i^s)]]=0
\end{align*}for any $y_{-1}\in V_{-1}$.
Therefore, we have $\sum _{s=1}^l([\alpha _1(x_1^s),b_i^s]+[x_1^s,\alpha _i(b_i^s)])=0$ and the well-definedness of $\alpha _{i+1}$.
Moreover, $\alpha _{i+1}$ satisfies the following equations:
\begin{align}
&\alpha _{i+1}([a_0,b_{i+1}])=[\alpha _{0}(a_0),b_{i+1}]+[a_0,\alpha _{i+1}(b_{i+1})],\label{der_eq;alpha_i+1_1}\\
&\alpha _{i+1}([x_1,b_{i}])=[\alpha _{1}(x_1),b_{i}]+[x_1,\alpha _{i}(b_{i})],\label{der_eq;alpha_i+1_2}\\
&\alpha _{i}([y_{-1},b_{i+1}])=[\alpha _{-1}(y_{-1}),b_{i+1}]+[y_{-1},\alpha _{i+1}(b_{i+1})]\label{der_eq;alpha_i+1_3}
\end{align}for any $a_0\in V_0$, $x_1\in V_1$, $y_{-1}\in V_{-1}$, $b_i\in V_i$ and $b_{i+1}\in V_{i+1}$.
In fact, for any $a_0\in V_0$, $x_1\in V_1$, and $b_{i}\in V_i$, we have
\begin{align}
&\alpha _{i+1}([a_0,[x_1,b_{i}]])=\alpha _{i+1}([[a_0,x_1],b_i])+\alpha _{i+1}([x_1,[a_0,b_i]])\notag\\
&\quad =[\alpha _{1}([a_0,x_1]),b_i]+[[a_0,x_1],\alpha _i(b_i)]+[\alpha _1(x_1),[a_0,b_i]]+[x_1,\alpha _{i}([a_0,b_i])]\notag\\
&\quad =[[\alpha _{0}(a_0),x_1],b_i]+[[a_0,\alpha _1(x_1)],b_i]+[[a_0,x_1],\alpha _i(b_i)]\notag\\
&\qquad+[\alpha _1(x_1),[a_0,b_i]]+[x_1,[\alpha _{0}(a_0),b_i]]+[x_1,[a_0,\alpha _{i}(b_i)]]\notag\\
&\quad =[\alpha _{0}(a_0),[x_1,b_i]]+[a_0,[\alpha _1(x_1),b_i]]+[a_0,[x_1,\alpha _i(b_i)]]\notag\\
&\quad =[\alpha _{0}(a_0),[x_1,b_i]]+[a_0,\alpha _{i+1}([x_1,b_i])].\label{der_eq;alpha_i+1_7}
\end{align}Thus, we can obtain the equation (\ref{der_eq;alpha_i+1_1}).
The equation (\ref{der_eq;alpha_i+1_2}) is clear.
The equation (\ref{der_eq;alpha_i+1_3}) follows from (\ref{proof;alphader_welldef}).
Thus, inductively, we can obtain linear maps $\alpha _i$ for all $i\geq 0$, and, similarly, we can construct linear maps $\alpha _{-i}:V_{-i}\rightarrow V_{-i}$ for all $i\geq 0$.
Consequently, we have linear maps $\alpha _n:V_n\rightarrow V_n$ for all $n\in\mathbb{Z}$ which satisfy
\begin{align}
&\alpha _{n}([a_0,b_{n}])=[\alpha _{0}(a_0),b_{n}]+[a_0,\alpha _{n}(b_{n})],\label{der_eq;alpha_i+1_4}\\
&\alpha _{n+1}([x_1,b_{n}])=[\alpha _{1}(x_1),b_{n}]+[x_1,\alpha _{n}(b_{n})],\label{der_eq;alpha_i+1_5}\\
&\alpha _{n-1}([y_{-1},b_{n}])=[\alpha _{-1}(y_{-1}),b_{n}]+[y_{-1},\alpha _{n}(b_{n})]\label{der_eq;alpha_i+1_6}
\end{align}
for any $a_0\in V_0$, $x_1\in V_1$, $y_{-1}\in V_{-1}$ and $b_n\in V_n$.
\par We define a linear map $\alpha :L(\mathfrak{g},\rho,V,{\cal V},B_0)\rightarrow L(\mathfrak{g},\rho,V,{\cal V},B_0)$ by:
\begin{align}
\alpha (a_n):=\alpha _n(a_n)
\end{align}for any $n\in\mathbb{Z}$ and $a_n\in V_n$.
Then $\alpha $ is a derivation of Lie algebras.
In fact, we can show the following equation 
\begin{align}
\alpha ([a_n,b_m])=[\alpha (a_n),b_m]+[a_n,\alpha (b_m)]
\end{align}for any $n,m\in\mathbb{Z}$, $a_n\in V_n$ and $b_m\in V_m$ by the equations (\ref{der_eq;alpha_i+1_4}), (\ref{der_eq;alpha_i+1_5}), (\ref{der_eq;alpha_i+1_6}) inductively.
\end{proof}
\begin{lemma}\label{lemma;der_Binv}
Let $(\mathfrak{g},\rho,V,{\cal V},B_0)$ be a standard pentad and $\alpha $ be a derivation on $L(\mathfrak{g},\rho,V,{\cal V},B_0)$.
If $\alpha $ satisfies the equation
\begin{align}
B_L(\alpha (z),\omega )=-B_L(z,\alpha (\omega ))
\end{align}
for any $z=z_n\in V_n$ $(n=0,\pm 1)$ and $\omega \in L(\mathfrak{g},\rho,V,{\cal V},B_0)$, then we have the same equation for any $z,\omega \in L(\mathfrak{g},\rho,V,{\cal V},B_0)$.
\end{lemma}
\begin{proof}
We argue our claim in the cases where $z=z_n\in V_n$ for some $n$ and prove it by induction on $n$.
Suppose that $n\geq 0$.
If $n=0,1$, then our claim follows from the assumption.
Suppose that $n\geq 2$.
Then, by the induction hypothesis, we have
\begin{align*}
&B_L(\alpha ([x_1,z_{n-1}]),\omega )=B_L([\alpha (x_1),z_{n-1}],\omega ) +B_L([x_1,\alpha (z_{n-1})],\omega )\notag\\
&\quad =-B_L(z_{n-1},[\alpha (x_1),\omega ]) -B_L(\alpha (z_{n-1}),[x_1,\omega ])\notag\\
&\quad =-B_L(z_{n-1},[\alpha (x_1),\omega ]) +B_L(z_{n-1},\alpha ([x_1,\omega ]))\notag\\
&\quad =B_L(z_{n-1},[x_1,\alpha (\omega )])\notag\\
&\quad =-B_L([x_1,z_{n-1}],\alpha (\omega ))
\end{align*}
for any $x_1\in V_1$, $z_{n-1}\in V_{n-1}$.
Since $V_{n}=[V_1,V_{n-1}]$, we have our claim for $n$.
Thus, by induction, we have our claim for all $n\geq 0$.
Similarly, we can show our claim for $n\leq -1$.
\end{proof}

\section{Graded modules of $L(\mathfrak{g},\rho,V,{\cal V},B_0)$}\label{sec;graded_modules}
\subsection{A construction of vector spaces $\tilde{U}^+$ and $\tilde{U}^-$}\label{sec;const_graded_ext_vec}
As mentioned in section \ref{sec;intro}, the purpose of this and the next section is to construct a positively graded module and a negatively graded module of $L(\mathfrak{g},\rho,V,{\cal V},B_0)$ from a given $\mathfrak{g}$-module $U$, which will be denoted by $\tilde{U}^+$ and $\tilde{U}^-$.
First, we construct $\tilde{U}^+ $ and $\tilde{U}^-$ as vector spaces by induction.
\begin{defn}\label{defn;r_-1}
Let $(\mathfrak{g},\rho,V,{\cal V},B_0)$ be a standard pentad and $L(\mathfrak{g},\rho,V,{\cal V},B_0)=\bigoplus _{n\in\mathbb{Z}}V_n$ be the Lie algebra associated with it.
Let $\pi:\mathfrak{g}\otimes U\rightarrow U$ be a representation of $\mathfrak{g}=V_0$ on a vector space $U$ over $\mathfrak{k}$. 
We put $U_{0}^+=U_{0}^-:=U$, $\pi _{0}^+=\pi _0^-:=\pi$ and define linear maps $r_0^+:V_{1}\otimes U_{0}^+\rightarrow {\rm Hom}(V_{-1},U_{0}^+)$ and $r_{0}^-:V_{-1}\otimes U_{0}^-\rightarrow {\rm Hom}(V_1,U_{0}^-)$ by:
\begin{align}
r_0^+:\ &V_1\otimes U_0^+\rightarrow {\rm Hom }(V_{-1},U_0^+)\notag\\
&x_1\otimes u_0\mapsto (\eta _{-1}\mapsto \pi _0^+([\eta _{-1},x_1]\otimes u_0)),\\
r_{0}^-:\ &V_{-1}\otimes U_0^-\rightarrow {\rm Hom }(V_{1},U_0^-)\notag\\
&y_{-1}\otimes u_0\mapsto (\xi _{1}\mapsto \pi _0^-([\xi _{1},y_{-1}]\otimes u_0)).
\end{align}
\end{defn}
\begin{pr}\label{pr;r_-1_hom}
The maps $r_{ 0}^+$ and $r_{0}^-$ are homomorphisms of $\mathfrak{g}$-modules. 
\end{pr}
\begin{proof}
We prove for $r_0^+$.
For any elements $a\in\mathfrak{g}$, $x _1\in V_1$, $\eta _{-1}\in V_{-1}$ and $u_{0}\in U_{0}^+$, we have
\begin{align*}
&r_0^+([a,x_1]\otimes u_0+x_1\otimes \pi _0^+(a\otimes u_0))(\eta _{-1})=\pi _0^+([\eta _{-1},[a,x_1]]\otimes u_0)+\pi _0^+([\eta _{-1},x_1]\otimes \pi _0^+(a\otimes u_0))\\
&\quad =\pi _0^+([a,[\eta _{-1},x_1]]\otimes u_0)-\pi _0^+([[a,\eta _{-1}],x_1]]\otimes u_0)+\pi _0^+([\eta _{-1},x_1]\otimes \pi _0^+(a\otimes u_0))\\
&\quad =\pi _0^+(a\otimes \pi _0^+([\eta _{-1},x_1]\otimes u_0))-\pi _0^+([[a,\eta _{-1}],x_1]]\otimes u_0)\\
&\quad =\pi _0^+(a\otimes (r_0^+(x_1\otimes u_0)(\eta _{-1})))-r_0^+(x_1\otimes u_0)([a,\eta _{-1}]).
\end{align*}
Thus $r_{0}^+$ is a homomorphism of $\mathfrak{g}$-modules.
Similarly, we can prove that $r_{0}^-$ is a homomorphism of $\mathfrak{g}$-modules. 
\end{proof}
It follows from ${\rm Proposition}$ $\ref{pr;r_-1_hom}$ that the linear spaces $U_{1}^+:={\rm Im}\ r_{0}^+$ and $U_{-1}^-:={\rm Im}\ r_{0}^-$ have the canonical $\mathfrak{g}$-module structures.
We denote these canonical representations by $\pi _{1}^+$ and $\pi _{-1}^-$ respectively.
Moreover, we inductively construct $\mathfrak{g}$-modules $U_2^+,U_3^+,\ldots$ by using the following proposition.
\begin{pr}\label{pr;r_hom}
Assume that there exist $\mathfrak{g}$-modules $(\varpi ^+,W ^+)$, $(\varpi ^-,W ^-)$ and $\mathfrak{g}$-module homomorphisms $\lambda ^+:V_1\otimes W ^+\rightarrow {\rm Hom }(V_{-1},W ^+)$ and $\lambda ^-:V_{-1}\otimes W ^-\rightarrow {\rm Hom }(V_{1},W ^-)$.
We put $\hat{W}^+:={\rm Im}\ \lambda ^+$, $\hat{W}^+:={\rm Im}\ \lambda ^-$ and denote the canonical representations of $\mathfrak{g}$ on them by $\hat{\varpi }^+$ and $\hat{\varpi }^-$ respectively.
Then the following linear maps are $\mathfrak{g}$-module homomorphisms:
\begin{align}
\hat{\lambda}^+:&V_{1}\otimes \hat{W}^+\rightarrow {\rm Hom}(V_{-1},\hat{W}^+)\notag\\
&x_1\otimes \hat{w}^+\mapsto (\eta _{-1}\mapsto \hat{\varpi }^+([\eta _{-1},x_1]\otimes \hat{w}^+)+\lambda ^+(x_1\otimes \hat{w}^+(\eta _{-1}))),\\
\hat{\lambda}^-:&V_{-1}\otimes \hat{W}^-\rightarrow {\rm Hom}(V_1,\hat{W}^-)\notag\\
&y_{-1}\otimes \hat{w}^-\mapsto (\xi _1\mapsto \hat{\varpi }^-([\xi _1,y_{-1}]\otimes \hat{w}^-)+\lambda ^-(y_{-1}\otimes \hat{w}^-(\xi _1))).
\end{align}
\end{pr}
\begin{proof}
We can prove it by a similar argument to the argument of \cite [Proposition 1.10]{Sa}.
Take any elements $a\in\mathfrak{g}$, $x_1\in V_1$, $\eta _{-1}\in V_{-1}$ and $\hat{w}^+\in \hat{W}^+$.
Then we have
\begin{align*}
&(\hat{\lambda }^+([a,x_{1}]\otimes \hat{w}^+)+\hat{\lambda }^+(x_{1}\otimes \hat{\varpi }^+(a\otimes \hat{w}^+))) (\eta _{-1})\\
&\quad =\hat{\varpi }^+([\eta _{-1},[a,x_1]]\otimes \hat{w}^+)+\lambda ^+([a,x_1]\otimes \hat{w}^+(\eta _{-1}))\\&\qquad+ \hat{\varpi }^+([\eta _{-1},x_1]\otimes \hat{\varpi }^+(a\otimes \hat{w}^+))+\lambda ^+(x_1\otimes (\hat{\varpi }^+(a\otimes \hat{w}^+)(\eta _{-1})))\\
&\quad =\hat{\varpi }^+([a,[\eta _{-1},x_1]]\otimes \hat{w}^+)+\hat{\varpi }^+([\eta _{-1},x_1]\otimes  \hat{\varpi }^+(a\otimes \hat{w}^+))+\lambda ^+([a,x_1]\otimes \hat{w}^+(\eta _{-1}))\\
&\qquad+\lambda ^+(x_1\otimes \varpi ^+(a\otimes \hat{w}^+(\eta _{-1})))-\hat{\varpi }^+([[a,\eta _{-1}],x_1]\otimes \hat{w}^+)-\lambda ^+(x_1\otimes \hat{w}^+([a,\eta _{-1}]))\\
&\quad =\hat{\varpi }^+(a\otimes \hat{\varpi }^+([\eta _{-1},x_1]\otimes  \hat{w}^+))+\hat{\varpi }^+(a\otimes  \lambda ^+(x_1\otimes \hat{w}^+(\eta _{-1})))\\
&\qquad-\hat{\varpi }^+([[a,\eta _{-1}],x_1]\otimes  \hat{w}^+)-\lambda ^+(x_1\otimes \hat{w}^+([a,\eta _{-1}]))\\
&\quad =\hat{\varpi }^+(a\otimes  \hat{\lambda }^+(x_1\otimes \hat{w}^+)(\eta _{-1})) -\hat{\lambda }^+(x_1\otimes \hat{w}^+)([a,\eta _{-1}]).
\end{align*}
Thus, $\hat{\lambda}^+$ is a homomorphism of $\mathfrak{g}$-modules.
By the same way, we can prove that $\hat{\lambda}^-$ is also a $\mathfrak{g}$-module homomorphism.
\end{proof}
\begin{defn}
Suppose that $j\geq 1$ and there exist $\mathfrak{g}$-modules $(\pi _{j-1}^+,U_{j-1}^+)$, $(\pi _{-j+1}^-,U_{-j+1}^-)$ and homomorphisms of $\mathfrak{g}$-modules $r_{j-1}^+:V_{1}\otimes U_{j-1}^+\rightarrow {\rm Hom}(V_{-1},U_{j-1}^+)$ and $r_{-j+1}^-:V_{-1}\otimes U_{-j+1}^-\rightarrow {\rm Hom}(V_1,U_{-j+1}^-)$.
Put $U_{j}^+:={\rm Im}\ r_{j-1}^+$ and $U_{-j}^-:={\rm Im}\ r_{-j+1}^-$.
Then we define linear maps $r_{j}^+$ and $r_{-j}^-$ by:
\begin{align}
r_{j}^+:\ &V_{1}\otimes U_{j}^+\rightarrow \text{Hom} (V_1,U_{j}^+)\notag\\
&x_{1}\otimes u_{j}^+\mapsto (\eta _1 \mapsto \pi _{j}^+([\eta _{-1},x_1]\otimes u_{j}^+)+r_{j-1}^+(x_1 \otimes u_{j}^+(\eta _{-1}))),\\
r_{-j}^-:\ &V_{-1}\otimes U_{-j}^-\rightarrow \text{Hom} (V_1,U_{-j}^-)\notag\\
&y_{-1}\otimes u_{-j}^-\mapsto (\xi _1 \mapsto \pi _{-j}^-([\xi _1,y_{-1}]\otimes u_{-j}^-)+r_{-j+1}^-(y_{-1}\otimes u_{-j}^-(\xi _1))).
\end{align}
Then, by ${\rm Proposition}$ $\ref{pr;r_hom}$, $r_{j}^+$ and $r_{-j}^-$ are homomorphisms of $\mathfrak{g}$-modules.
We denote by $U_{j+1}^+$ and $U_{-j-1}^-$ the images of $r_{j}^+$ and $r_{-j}^-$ and the canonical representations of $\mathfrak{g}$ on $U_{j+1}^+$ and $U_{-j-1}^-$ by $\pi _{j+1}^+$ and $\pi _{-j-1}^-$ respectively.
Moreover, we put 
\begin{align*}
U_{-j}^+:=\{0\},\quad U^-_{j}:=\{0\}
\end{align*}
for $j\geq 1 $.
We denote the zero representations of $\mathfrak{g}$ on $U_{-j}^+$ and $U_{j}^-$ by $\pi _{-j}^+$ and $\pi _{j}^-$ for all $j\geq 1$.
Thus, inductively, we obtain $\mathfrak{g}$-modules $(\pi _m^+,U_m^+)$, $(\pi _m^-,U_{m}^-)$ for all $m\in\mathbb{Z}$.
Under these notation, we define linear spaces $\tilde{U}^+$ and $\tilde{U}^-$ by:
\begin{align}
\tilde{U}^+:=\bigoplus _{m\in \mathbb{Z}}U_n^+,\quad \tilde{U}^-:=\bigoplus _{m\in\mathbb{Z}}U_n^-.\label{def_U_pm_as_vec}
\end{align}
Throughout this paper, we use these notation.
\end{defn}

\subsection{A construction of representations of $L(\mathfrak{g},\rho,V,{\cal V},B_0)$ on $\tilde{U}^+$ and $\tilde{U}^-$}\label{sec;const_graded_ext}
In this section, we define $L(\mathfrak{g},\rho,V,{\cal V},B_0)$-module structures on vector spaces $\tilde{U}^+$ and $ \tilde{U}^-$ constructed in (\ref{def_U_pm_as_vec}).
For this, we start with the following definition.
\begin{defn}\label{defn;pi}
We define the following linear maps:
\begin{align*}
\pi _{0,m}^+:V_0\otimes U_{m}^+\rightarrow U_{m}^+,
\quad \pi _{1,m}^+:V_1\otimes U_{m}^+\rightarrow U_{m+1}^+,
\quad \pi _{-1,m}^+:V_{-1}\otimes U_{m}^+\rightarrow U_{m-1}^+
\end{align*}
by:
\begin{align}
&\pi _{0,m}^+(a\otimes u_{m}^+):=\pi _{m}^+(a\otimes u_{m}^+) \quad (m\in\mathbb{Z}),\label{defn_pi^0_m}\\
&\pi _{1,m}^+(x_1\otimes u_{m}^+):=\begin{cases}r_{m}^+(x_1\otimes u_m^+)&(m\geq 0)\\0&(m\leq -1)\end{cases},\label{defn_pi^1_m}\\
&\pi _{-1,m}^+(y_{-1}\otimes u_{m}^+):=\begin{cases}u_m^+(y_{-1})&(m\geq 1)\\0&(m\leq 0)\end{cases}\label{defn_pi^-1_m}
\end{align}
where $a\in V_0$, $x_1\in V_1$, $y_{-1}\in V_{-1}$ and $u_{m}^+\in U_{m}^+$.
\end{defn}
By the above definition, we can obtain the following proposition.
\begin{pr}\label{defn;pi_1_-1_compati}
Under the above notation, we have the following equations:
\begin{align}
&\pi ^+_{1,m}([x_1,a]\otimes u_{m}^+)=\pi ^+_{1,m}(x_1\otimes \pi ^+_{0,m}(a\otimes u_{m}^+))-\pi ^+_{0,m+1}(a\otimes \pi ^+_{1,m}(x_1\otimes u_{m}^+)),\label{pr_pi_0+1}\\
&\pi ^{+}_{-1,m}([y_{-1}, a]\otimes u_{m}^+)=\pi ^{+}_{-1,m}(y_{-1}\otimes \pi ^{+}_{0,m}(a\otimes u_{m}^+))-\pi ^{+}_{0,m-1}(a\otimes \pi ^{+}_{-1,m}(y_{-1}\otimes u_{m}^+)),\label{defn_pi_1}\\
&\pi ^+_{1,m-1}(x_1\otimes \pi ^{+}_{-1,m}(y_{-1}\otimes u_{m}^+))=\pi ^+ _{0,m}([x_1,y_{-1}]\otimes u_m)+\pi ^{+}_{-1,m+1}(y_{-1}\otimes \pi ^+_{1,m}(x_1\otimes u_m)).\label{pr_pi_1-1}
\end{align}
\end{pr}
\begin{proof}
Let us show (\ref{pr_pi_1-1}).
The equations (\ref{pr_pi_0+1}) and (\ref{defn_pi_1}) can be shown similarly.
If $m\leq -1$, then (\ref{pr_pi_1-1}) is clear.
If $m=0$, then the left hand side equals to $0$.
For the right hand side, we have 
\begin{align*}
&\pi ^+ _{0,0}([x_1,y_{-1}]\otimes u_0)+\pi ^{+}_{-1,1}(y_{-1}\otimes \pi ^+_{1,0}(x_1\otimes u_0))\\
&\quad =\pi _0^+([x_1,y_{-1}]\otimes u_0^+)+r_0^+(x_1\otimes u_0^+)(y_{-1})=\pi _0^+([x_1,y_{-1}]\otimes u_0^+)+\pi _0^+([y_{-1},x_1]\otimes u_0^+)=0.
\end{align*}Thus we have (\ref{pr_pi_1-1}) when $m=0$.
For $m\geq 1$, the equation (\ref{pr_pi_1-1}) follows from definition.
\end{proof}
\begin{defn}\label{pi^i+-1}
We define the following linear maps for $i\geq 1$ inductively:
\begin{align}
\pi ^+_{i+1,m}\ :\ &V_{i+1}\otimes U_{m}^+\rightarrow U_{i+m+1}^+\notag\\
&p_{i}(x_1\otimes z_i)\otimes u_{m}^+\mapsto \pi ^+_{1,i+m}(x_1\otimes \pi ^+_{i,m}(z_i\otimes u_{m}^+))-\pi ^+_{i,m+1}(z_i\otimes \pi ^+_{1,m}(x_1\otimes u_{m}^+)),\label{defn_pi_i+1}\\
\pi ^+_{-i-1,m}\ :\ &V_{-i-1}\otimes U_{m}^+\rightarrow U_{-i+m-1}^+\notag\\
&q_{-i}(y_{-1}\otimes \omega _{-i})\otimes u_{m}^+\mapsto \pi ^+_{-1,-i+m}(y_{-1}\otimes \pi ^+_{-i,m}(\omega _{-i}\otimes u_{m}^+))\notag\\&\hspace{180truept}-\pi ^+_{-i,m-1}(\omega _{-i}\otimes \pi ^+_{-1,m}(y_{-1}\otimes u_{m}^+)).\label{defn_pi_-i-1}
\end{align}
\end{defn}
Note that the linear maps $\pi ^+_{0,m}$, $\pi ^+_{\pm 1,m}$ defined in {\rm Definition} \ref{defn;pi} satisfy the same equations as (\ref{defn_pi_i+1}) and (\ref{defn_pi_-i-1}) in the cases where $i=0$ by \rm Proposition \ref{defn;pi_1_-1_compati}.
For $i\geq 1$, we must show the well-definedness of ${\rm Definition}$ $\ref{pi^i+-1}$.
To prove it, let us show the following two propositions.
\begin{pr}(The well-definedness of $\pi ^+_{i+1,m}$ given in (\ref{defn_pi_i+1}))\label{pr_well-def_pi_i+1}
Suppose that $i\geq 0$.
Suppose that the linear map $\pi _{i,m}^+$ defined in (\ref{defn_pi_i+1}) is well-defined for any $m\in\mathbb{Z}$ and satisfies the following equations:
\begin{align}
&\pi ^+_{0,i+m }(a\otimes \pi ^+_{i,m}(z_i\otimes u^+_{m}))=\pi ^+_{i,m}([a, z_i]\otimes u^+_{m})+\pi ^+_{i,m}(z_i\otimes \pi^+_{0,m}(a\otimes u^+_{m})),\label{well-def:i_0_hom}\\
&\pi ^+_{i,m -1}(z_i\otimes \pi ^+_{-1,m }(y_{-1}\otimes u^+_{m }))=\pi ^+_{i-1,m }([z_i,y_{-1}]\otimes u^+_{m })+\pi ^+_{-1,i+m }(y_{-1}\otimes \pi ^+_{i,m }(z_i\otimes u^+_{m }))\label{well-def:i_-1_hom}.
\end{align}
If $x_1^1,\ldots ,x_1^l\in V_1$ and $z_i^1,\ldots ,z_i^l\in V_i$ satisfy $\sum _{s=1}^lp_i(x_1^s\otimes z_i^s)=0$, then we have
\begin{align}
\displaystyle \sum _{s=1}^l (\pi ^+_{1,i+m}(x_1^s\otimes \pi ^+_{i,m}(z_i^s\otimes u^+_{m}))-\pi ^+_{i,m+1}(z_i^s\otimes \pi ^+_{1,m}(x_1^s\otimes u^+_{m})))=0 \label{well-defn_pi_i+1_m}
\end{align}
for all $m\in\mathbb{Z}$ and $u^+_{m}\in U_{m}^+$.
In particular, we can obtain the well-definedness of the linear map $\pi ^+_{i+1,m}$ defined in (\ref{defn_pi_i+1}) for any $m\in\mathbb{Z}$.
Moreover, the linear maps $\pi ^+_{i+1,m}$ $(m\in\mathbb{Z})$ satisfy the following equations:
\begin{align}
&\pi ^+_{0,i+m+1 }(a\otimes \pi ^+_{i+1,m}(z_{i+1}\otimes u^+_{m}))\notag\\
&\quad =\pi ^+_{i+1,m}([a, z_{i+1}]\otimes u^+_{m})+\pi ^+_{i+1,m}(z_{i+1}\otimes \pi^+_{0,m}(a\otimes u^+_{m})),\label{well-def:i+1_0_hom}\\
&\pi ^+_{i+1,m -1}(z_{i+1}\otimes \pi ^+_{-1,m }(y_{-1}\otimes u^+_{m }))\notag\\
&\quad =\pi ^+_{i,m }([z_{i+1},y_{-1}]\otimes u^+_{m })+\pi ^+_{-1,i+m+1 }(y_{-1}\otimes \pi ^+_{i+1,m }(z_{i+1}\otimes u^+_{m })).\label{well-def:i+1_-1_hom}
\end{align}
\end{pr}
\begin{proof}
We argue by induction on $i$.
For $i=0$, our claim follows from Proposition \ref{defn;pi_1_-1_compati}.
Suppose that $i\geq 1$.
We fix $i$ and argue (\ref{well-defn_pi_i+1_m}) by induction on $m$.
First, if $m\leq -1$, then the equation (\ref{well-defn_pi_i+1_m}) is clear.
If $m\geq 0$, then we have
\begin{align}
&\pi ^+_{-1,i+m+1}(y_{-1}\otimes \pi ^+_{1,i+m}(x_1\otimes \pi ^+_{i,m}(z_{i}\otimes u^+_{m}))-\pi ^+_{i,m+1}(z_{i}\otimes \pi ^+_{1,m}(x_1\otimes u^+_{m})))\notag\\
&\quad =\pi ^+_{0,i+m}([y_{-1},x_1]\otimes  \pi ^+_{i,m}(z_{i}\otimes u^+_{m}))+\pi ^+_{1,i+m-1}(x_1\otimes \pi ^+_{-1,i+m}(y_{-1}\otimes \pi ^+_{i,m}(z_{i}\otimes u^+_{m})))\notag\\
&\qquad-\pi ^+_{i-1,m+1}([y_{-1},z_{i}]\otimes \pi ^+_{1,m}(x_1\otimes u^+_{m}))-\pi ^+_{i,m}(z_i\otimes \pi ^+_{-1,m+1}(y_{-1}\otimes \pi ^+_{1,m}(x_1\otimes u^+_{m})))\notag\\
&\quad =\pi ^+_{0,i+m}([y_{-1},x_1]\otimes  \pi ^+_{i,m}(z_{i}\otimes u^+_{m}))+\pi ^+_{1,i+m-1}(x_1\otimes \pi ^+_{i-1,m}([y_{-1},z_{i}]\otimes u^+_{m}))\notag\\
&\qquad+\pi ^+_{1,i+m-1}(x_1\otimes \pi ^+_{i,m-1}(z_i\otimes \pi ^+_{-1,m}(y_{-1}\otimes u^+_{m})))-\pi ^+_{i-1,m+1}([y_{-1},z_{i}]\otimes \pi ^+_{1,m}(x_1\otimes u^+_{m}))\notag\\
&\qquad-\pi ^+_{i,m}(z_i\otimes \pi ^+_{0,m}([y_{-1},x_1]\otimes u^+_{m}))-\pi ^+_{i,m}(z_i\otimes \pi ^+_{1,m-1}(x_{1}\otimes \pi ^+_{-1,m}(y_{-1}\otimes u^+_{m})))\notag\\
&\quad =\pi ^+_{i,m}([[y_{-1},x_1],z_{i}]\otimes u^+_{m})+\pi ^+_{i,m}([x_1,[y_{-1},z_{i}]]\otimes u^+_{m})\notag\\
&\qquad+\pi ^+_{1,i+m-1}(x_1\otimes \pi ^+_{i,m-1}(z_i\otimes \pi ^+_{-1,m}(y_{-1}\otimes u^+_{m})))-\pi ^+_{i,m}(z_i\otimes \pi ^+_{1,m-1}(x_{1}\otimes \pi ^+_{-1,m}(y_{-1}\otimes u^+_{m})))\notag\\
&\quad =\pi ^+_{i,m}([y_{-1},[x_1,z_{i}]]\otimes u^+_{m})\notag\\
&\qquad+\pi ^+_{1,i+m-1}(x_1\otimes \pi ^+_{i,m-1}(z_i\otimes \pi ^+_{-1,m}(y_{-1}\otimes u^+_{m})))-\pi ^+_{i,m}(z_i\otimes \pi ^+_{1,m-1}(x_{1}\otimes \pi ^+_{-1,m}(y_{-1}\otimes u^+_{m})))\label{well-def;pi_i+1_hom}
\end{align}for any $x_1\in V_1$, $z_i\in V_i$, $y_{-1}\in V_{-1}$ and $u^+_{m}\in U_{m}^+$.
By the induction hypotheses on $i$ and $m$, if we take elements $x_1^1,\ldots ,x_1^l\in V_1$ and $z_i^1,\ldots ,z_i^l\in V_i$ satisfying $\sum _{s=1}^lp_i(x_1^s\otimes z_i^s)=0$, then we have
\begin{align}
&\displaystyle \sum _{s=1}^l\pi ^+_{i,m}([[x_1^s,z_i^s],y_{-1}]\otimes u^+_m)=0\quad \text{(by the induction hypothesis on $i$),}\\
&\displaystyle \sum _{s=1}^l   (\pi ^+_{1,i+m-1}(x_1^s\otimes \pi ^+_{i,m-1}(z_i^s\otimes \pi ^+_{-1,m}(y_{-1}\otimes u^+_{m}))) -\pi ^+_{i,m}(z_i^s\otimes \pi ^+_{1,m-1}(x_{1}^s\otimes \pi ^+_{-1,m}(y_{-1}\otimes u^+_{m}))))\notag \\&\quad =0\quad \text{(by the induction hypothesis on $m$)}.
\end{align}
Thus, we have
\begin{align}
\sum _{s=1}^l\pi ^+_{-1,i+m+1}(y_{-1}\otimes \pi ^+_{1,i+m}(x_1^s\otimes \pi ^+_{i,m}(z_i^s\otimes u^+_{m}))-\pi ^+_{i,m+1}(z_i^s\otimes \pi ^+_{1,m}(x_1^s\otimes u^+_m)))=0
\end{align}from (\ref{well-def;pi_i+1_hom}).
Since $i+m+1\geq 1$, we can obtain that
\begin{align}
&\displaystyle \sum _{s=1}^l(\pi ^+_{1,i+m}(x_1^s\otimes \pi ^+_{i,m}(z_i^s\otimes r _{m+1}(y_{-1}\otimes u^+_{m+1})))-\pi ^+_{i,m+1}(z_i^s\otimes \pi ^+_{1,m}(x_1\otimes r _{m+1}(y_{-1}\otimes u^+_{m+1}))) )\notag\\&\quad =0\in U^+_{i+m+1}\subset \mathrm {Hom}(V_{-1},U^+_{i+m}).
\end{align}
Therefore we can obtain the well-definedness of the linear map $\pi ^{i+1}_m:V_{i+1}\otimes U_m^+\rightarrow U_{i+m+1}^+$ given in (\ref{defn_pi_i+1}) for any $m$.\par
In order to complete the proof, we must show the equations (\ref{well-def:i+1_0_hom}) and (\ref{well-def:i+1_-1_hom}).
Let us show (\ref{well-def:i+1_0_hom}).
Under the above notation, for any $m\in\mathbb{Z}$, we have
\begin{align*}
&\pi ^+_{0,i+m+1}(a\otimes \pi ^+_{i+1,m}(p_i(x_1\otimes z_{i})\otimes u^+_m))\\
&\quad =\pi ^+_{0 ,i+m+1}(a\otimes \pi ^+_{1,i+m}(x_1\otimes \pi ^+_{i,m}(z_i\otimes u^+_{m})))-\pi ^+_{0,i+m+1}(a\otimes \pi ^+_{i,m+1}(z_i\otimes \pi ^+_{1,m}(x_1\otimes u^+_{m})))\\
&\quad =\pi ^+_{1,i+m}([a,x_1]\otimes \pi ^+_{i,m}(z_i\otimes u^+_{m}))+\pi ^+_{1,i+m}(x_1\otimes \pi ^+_{0,i+m}(a\otimes \pi ^+_{i,m}(z_i\otimes u^+_{m})))\\
&\qquad-\pi ^+_{i,m+1}([a,z_i]\otimes \pi ^+_{1,m}(x_1\otimes u^+_{m}))-\pi ^+_{i,m+1}(z_i\otimes \pi ^+_{0,m+1}(a\otimes \pi ^+_{1,m}(x_1\otimes u^+_{m})))\\
&\quad =\pi ^+_{1,i+m}([a,x_1]\otimes \pi ^+_{i,m}(z_i\otimes u^+_{m}))+\pi ^+_{1,i+m}(x_1\otimes  \pi ^+_{i,m}([a,z_i]\otimes u^+_{m}))\\
&\qquad+\pi ^+_{1,i+m}(x_1\otimes  \pi ^+_{i,m}(z_i\otimes \pi ^+_{0,m}(a\otimes u^+_{m})))-\pi ^+_{i,m+1}([a,z_i]\otimes \pi ^+_{1,m}(x_1\otimes u^+_{m}))\\
&\qquad-\pi ^+_{i,m+1}(z_i\otimes \pi ^+_{1,m}([a,x_1]\otimes u^+_{m}))-\pi ^+_{i,m+1}(z_i\otimes \pi ^+_{1,m}(x_1\otimes \pi ^+_{0 ,m}(a\otimes u^+_{m})))\\
&\quad =\pi ^+_{i+1,m}([[a,x_1],z_i]\otimes u^+_m)+\pi ^+_{i+1,m}([x_1,[a,z_i]]\otimes u^+_m)+\pi ^+_{i+1,m}([x_1,z_i]\otimes \pi ^+_{0,m}(a\otimes u^+_m))\\
&\quad =\pi ^+_{i+1,m}([a,p_i(x_1\otimes z_i)]\otimes u^+_m)+\pi ^+_{i+1,m}(p_i(x_1\otimes z_i)\otimes \pi ^+_{0,m}(a\otimes u^+_m)).
\end{align*}
Thus we have (\ref{well-def:i+1_0_hom}).
The equation (\ref{well-def:i+1_-1_hom}) follows from (\ref{well-def;pi_i+1_hom}).
This completes the proof.
\end{proof}
\begin{pr}(The well-definedness of $\pi ^+_{-i-1,m}$ given in (\ref{defn_pi_-i-1}))\label{pr_well-def_pi_-i-1}
Suppose that $i\geq 0$.
Suppose that the linear map $\pi ^+_{-i,m}$ defined in (\ref{defn_pi_-i-1}) is well-defined for any $m\in\mathbb{Z}$ and satisfies the following equations:
\begin{align}
&\pi ^+_{0,-i+m}(a\otimes \pi ^+_{-i,m}(\omega _{-i}\otimes u_m^+))=\pi ^+_{-i,m}([a,\omega _{-i}]\otimes u_m^+)+\pi ^+_{-i,m}(\omega _{-i}\otimes \pi ^+_{0,m}(a\otimes u_m^+)),\label{well-def:-i_0_hom}\\
&\pi ^+_{-i,m+1}(\omega _{-i}\otimes \pi ^+_{1,m}(x_1\otimes u_m^+))=\pi ^+_{-i+1,m}([\omega _{-i},x_1]\otimes u_m^+)+\pi ^+_{1,-i+m}(x_1\otimes \pi ^+_{-i,m}(\omega _{-i}\otimes u_m^+)).\label{well-def:-i_+1_hom}
\end{align}
If $y_{-1}^1,\ldots ,y_{-1}^l\in V_{-1}$ and $\omega _{-i}^1,\ldots ,\omega _{-i}^l\in V_{-i}$ satisfy $\sum _{s=1}^lq_{-i}(y_{-1}^s\otimes \omega _{-i}^s)=0$, then we have
\begin{align}
\displaystyle \sum _{s=1}^l (\pi ^+_{-1,-i+m}(y_{-1}^s\otimes \pi ^+_{-i,m}(\omega _{-i}^s\otimes u_{m}^+))-\pi ^+_{-i,m-1}(\omega _{-i}^s\otimes \pi ^+_{-1,m}(y_{-1}^s\otimes u_{m}^+)))=0 \label{well-defn_pi_-i-1_m}
\end{align}for all $m\in\mathbb{Z}$ and $u_{m}^+\in U_{m}^+$.
In particular, we can obtain the well-definedness of the linear map $\pi ^+_{-i-1,m}$ defined in (\ref{defn_pi_-i-1}) for any $m\in\mathbb{Z}$.
Moreover, the maps $\pi ^+_{-i-1,m}$ $(m\in\mathbb{Z})$ satisfy the following equations:
\begin{align}
&\pi ^+_{0,-i+m-1 }(a\otimes \pi ^+_{-i-1,m}(\omega _{-i-1}\otimes u_{m}^+))\notag\\
&\hspace{30truept}=\pi ^+_{-i-1,m}([a, \omega _{-i-1}]\otimes u_{m}^+)+\pi ^+_{-i-1,m}(\omega _{-i-1}\otimes \pi^+_{0,m}(a\otimes u_{m}^+)),\label{well-def:-i-1_0_hom}\\
&\pi ^+_{-i-1,m +1}(\omega _{-i-1}\otimes \pi ^+_{1,m }(x_{1}\otimes u_{m }^+))\notag\\
&\hspace{30truept}=\pi ^+_{-i,m }([\omega _{-i-1},x_{1}]\otimes u_{m }^+)+\pi ^+_{1,-i+m-1 }(x_{1}\otimes \pi ^+_{-i-1,m }(\omega _{-i-1}\otimes u_{m }^+)).\label{well-def:-i-1_-1_hom}
\end{align}
\end{pr}
\begin{proof}
If $i=0$, then our claim immediately follows from the definition.
Suppose that $i\geq 1$.
We fix $i$ and discuss by induction on $m$.
If $m\leq 0$, the equation (\ref{well-defn_pi_-i-1_m}) is clear.
Suppose that $m\geq 1$.
Then, for any $x_1\in V_1$, $y_{-1}\in V_{-1}$, $\omega _{-i}\in V_{-i}$ and $u_{m-1}^+\in U_{m-1}^+$, we have
\begin{align}
&\pi ^+_{-1,-i+m}(y_{-1}\otimes \pi ^+_{-i,m}(\omega _{-i}\otimes \pi ^+_{1,m-1}(x_1\otimes u_{m-1}^+)))\notag\\
&\qquad-\pi ^+_{-i,m-1}(\omega _{-i}\otimes \pi ^+_{-1,m}(y_{-1}\otimes \pi ^+_{1,m-1}(x_1\otimes u_{m-1}^+)))\notag\\
&\quad =\pi ^+_{-1,-i+m}(y_{-1}\otimes \pi ^+_{-i+1,m-1}([\omega _{-i},x_1]\otimes u_{m-1}^+))\notag\\
&\qquad+\pi ^+_{-1,-i+m}(y_{-1}\otimes \pi ^+_{1,-i+m-1}(x_1\otimes \pi ^+_{-i,m-1}(\omega _{-i}\otimes u_{m-1}^+)))\notag\\
&\qquad-\pi ^+_{-i,m-1}(\omega _{-i}\otimes \pi ^+_{0,m-1}([y_{-1},x_1]\otimes u_{m-1}^+))\notag\\
&\qquad-\pi ^+_{-i,m-1}(\omega _{-i}\otimes \pi ^+_{1,m-2}(x_1\otimes \pi ^+_{-1,m-1}(y_{-1}\otimes u_{m-1}^+)))\notag\\
&\quad =\pi ^+_{-1,-i+m}(y_{-1}\otimes \pi ^+_{-i+1,m-1}([\omega _{-i},x_1]\otimes u_{m-1}^+))-\pi ^+_{-i,m-1}(\omega _{-i}\otimes \pi ^+_{0,m-1}([y_{-1},x_1]\otimes u_{m-1}^+))\notag\\
&\qquad+\pi ^+_{0,-i+m-1}([y_{-1},x_1]\otimes \pi ^+_{-i,m-1}(\omega _{-i}\otimes u_{m-1}^+))\notag\\
&\qquad+\pi ^+_{1,-i+m-2}(x_1\otimes \pi ^+_{-1,-i+m-1}(y_{-1}\otimes \pi ^+_{-i,m-1}(\omega _{-i}\otimes u_{m-1}^+)))\notag\\
&\qquad-\pi ^+_{-i+1,m-2}([\omega _{-i},x_1]\otimes \pi ^+_{-1,m-1}(y_{-1}\otimes u_{m-1}^+))\notag\\
&\qquad-\pi ^+_{1,-i+m-2}(x_1\otimes \pi ^+_{-i,m-2}(\omega _{-i}\otimes \pi ^+_{-1,m-1}(y_{-1}\otimes u_{m-1}^+)))\notag\\
&\quad =-\pi ^+_{-i,m-1}([[\omega _{-i},x_1],y_{-1}]\otimes u_{m-1}^+)+\pi ^+_{-i,m-1}([[y_{-1},x_1],\omega _{-i}]\otimes u_{m-1}^+)\notag\\
&\qquad+\pi ^+_{1,-i+m-2}(x_1\otimes \pi ^+_{-1,-i+m-1}(y_{-1}\otimes \pi ^+_{-i,m-1}(\omega _{-i}\otimes u_{m-1}^+)))\notag\\
&\qquad-\pi ^+_{1,-i+m-2}(x_1\otimes \pi ^+_{-i,m-2}(\omega _{-i}\otimes \pi ^+_{-1,m-1}(y_{-1}\otimes u_{m-1}^+)))\notag\\
&\quad =\pi ^+_{-i,m-1}([[y_{-1},\omega _{-i}],x_1]\otimes u_{m-1}^+)\notag\\
&\qquad+\pi ^+_{1,-i+m-2}(x_1\otimes \pi ^+_{-1,-i+m-1}(y_{-1}\otimes \pi ^+_{-i,m-1}(\omega _{-i}\otimes u_{m-1}^+)))\notag\\
&\qquad-\pi ^+_{1,-i+m-2}(x_1\otimes \pi ^+_{-i,m-2}(\omega _{-i}\otimes \pi ^+_{-1,m-1}(y_{-1}\otimes u_{m-1}^+))).\label{well-def;pi_-i-1_hom}
\end{align}
By the induction hypotheses on $i$ and $m$, if we take elements $y_{-1}^1,\ldots ,y_{-1}^l\in V_{-1}$ and $\omega _{-i}^1,\ldots ,\omega _{-i}^l\in V_{-i}$ satisfying $\sum _{s=1}^lq_{-i}(y_{-1}^s\otimes \omega _{-i}^s)=0$, then we have
\begin{align}
&\displaystyle \sum _{s=1}^l\pi ^+_{-i,m-1}([[y_{-1}^s,\omega _{-i}^s],x_1]\otimes u_{m-1}^+)=0\quad \text{(by the induction hypothesis on $i$),}\\
&\displaystyle \sum _{s=1}^l ( \pi ^+_{1,-i+m-2}(x_1\otimes \pi ^+_{-1,-i+m-1}(y_{-1}\otimes \pi ^+_{-i,m-1}(\omega _{-i}\otimes u_{m-1}^+)))\notag\\
&\qquad-\pi ^+_{1,-i+m-2}(x_1\otimes \pi ^+_{-i,m-2}(\omega _{-i}\otimes \pi ^+_{-1,m-1}(y_{-1}\otimes u_{m-1}^+))) )=0\notag\\&\hspace{185truept}\text{(by the induction hypothesis on $m$)}.
\end{align}
Thus, we have
\begin{align}
\sum _{s=1}^l(&\pi ^+_{-1,-i+m}(y_{-1}^s\otimes \pi ^+_{-i,m}(\omega _{-i}^s\otimes \pi ^+_{1,m-1}(x_1\otimes u_{m-1}^+)))\notag\\
&\qquad-\pi ^+_{-i,m-1}(\omega _{-i}^s\otimes \pi ^+_{-1,m}(y_{-1}^s\otimes \pi ^+_{1,m-1}(x_1\otimes u_{m-1}^+))))=0
\end{align}from (\ref{well-def;pi_-i-1_hom}).
Since $\pi ^+_{1,m-1}:V_1\otimes U_{m-1}^+\rightarrow U^+_m$ is surjective, we can obtain the equation (\ref{well-defn_pi_-i-1_m}).
Therefore we can obtain the well-definedness of the linear map $\pi ^+_{-i-1,m}:V_{-i-1}\otimes U_m^+\rightarrow U_{-i+m-1}^+$ given in (\ref{defn_pi_-i-1}) for any $m$.\par
The equation (\ref{well-def:-i-1_0_hom}) can be shown by a similar way to the proof of Proposition \ref{pr_well-def_pi_i+1}.
Moreover, the equation (\ref{well-def:-i-1_-1_hom}) follows from (\ref{well-def;pi_-i-1_hom}).
\end{proof}
\begin{defn}
By the above propositions, Propositions \ref{pr_well-def_pi_i+1} and \ref{pr_well-def_pi_-i-1}, we define a linear map $\tilde{\pi }^{+ } :L(\mathfrak{g},\rho,V,{\cal V},B_0)\otimes \tilde{U}^+\rightarrow \tilde{U}^+$ by:
\begin{align*}
\tilde{\pi }^{+ } (z_n\otimes u_m^+):=\pi ^+_{n,m}(z_n\otimes u_m^+)
\end{align*}
where $n,m\in\mathbb{Z}$, $z_n\in V_n$ and $u_m^+\in U_m^+$.
\end{defn}
This linear map $\tilde{\pi }^+$ satisfies the following equations:
\begin{align*}
&\tilde{\pi }^{+ } ([a,z_n]\otimes u_m^+)=\tilde{\pi }^{+ } (a\otimes \tilde{\pi }^{+ } (z_n\otimes u_m^+)) -\tilde{\pi }^{+ } (z_n\otimes \tilde{\pi }^{+ } (a\otimes u_m^+)),\\
&\tilde{\pi }^{+ } ([x_1,z_n]\otimes u_m^+)=\tilde{\pi }^{+ } (x_1\otimes \tilde{\pi }^{+ } (z_n\otimes u_m^+)) -\tilde{\pi }^{+ } (z_n\otimes \tilde{\pi }^{+ } (x_1\otimes u_m^+)),\\
&\tilde{\pi }^{+ } ([y _{-1},z_n]\otimes u_m^+)=\tilde{\pi }^{+ } (y _{-1}\otimes \tilde{\pi }^{+ } (z_n\otimes u_m^+)) -\tilde{\pi }^{+ } (z_n\otimes \tilde{\pi }^{+ } (y _{-1}\otimes u_m^+))
\end{align*}
for any $n,m\in\mathbb{Z}$, $a\in V_0$, $x_1\in V_1$, $y _{-1}\in V_{-1}$, $z_n\in V_n$ and $u_m^+\in U_m^+$.
Moreover, we have the following proposition on $\tilde{\pi}^+$.
\begin{pr}\label{pr_tilde_pi_rep}
The map $\tilde{\pi }^{+ }$ satisfies the following equation:
\begin{align}
\tilde{\pi }^{+ }([x,y]\otimes u)=\tilde{\pi }^{+ }(x\otimes \tilde{\pi }^{+ }(y\otimes u))-\tilde{\pi }^{+ }(y\otimes \tilde{\pi }^{+ }(x\otimes u))
\end{align}for any $x,y\in L(\mathfrak{g},\rho,V,{\cal V},B_0)$ and $u^+\in \tilde{U}^+$.
\end{pr}
\begin{proof}
To prove our claim, it is sufficient to show the case where $x=z_n\in V_n$ for some $n\in \mathbb{Z}$.
We argue by induction on $n$.
\par Assume that $n\geq 0$.
For $n=0,1$, our result has been shown.
For $n\geq 2$.
We can assume that $z_n=p_{n-1}(x_1\otimes z_{n-1})$ for some $x_1\in V_1$ and $z_{n-1}\in V_{n-1}$ without loss of generality.
Then, by the induction hypothesis, we have 
\begin{align*}
&\tilde{\pi }^{+ }([p_{n-1}(x_1\otimes z_{n-1}),y]\otimes u^+)=\tilde{\pi }^{+ }([x_1,[z_{n-1},y]]\otimes u^+) -\tilde{\pi }^{+ }([z_{n-1},[x_1,y]]\otimes u^+)\notag\\
&\quad =\tilde{\pi }^{+ }(x_1\otimes \tilde{\pi}^+([z_{n-1},y]\otimes u^+))-\tilde{\pi }^{+ }([z_{n-1},y]\otimes \tilde{\pi }^{+ }(x_1\otimes u^+))\notag\\
&\qquad-\tilde{\pi }^{+ }(z_{n-1}\otimes \tilde{\pi}^+([x_1,y]\otimes u^+))+\tilde{\pi }^{+ }([x_1,y]\otimes \tilde{\pi }^{+ }(z_{n-1}\otimes u^+))\notag\\
&\quad =\tilde{\pi }^{+ }(x_1\otimes \tilde{\pi}^+(z_{n-1}\otimes \tilde{\pi}^+(y\otimes u^+)))-\tilde{\pi }^{+ }(x_1\otimes \tilde{\pi}^+(y\otimes \tilde{\pi }^{+ }(z_{n-1}\otimes u^+)))\notag\\
&\qquad-\tilde{\pi }^{+ }(z_{n-1}\otimes \tilde{\pi}^+(y\otimes \tilde{\pi}^+(x_1\otimes u^+)))+\tilde{\pi }^{+ }(y\otimes \tilde{\pi}(z_{n-1}\otimes \tilde{\pi }^{+ }(x_1\otimes u^+)))\notag\\
&\qquad-\tilde{\pi }^{+ }(z_{n-1}\otimes \tilde{\pi}^+(x_1\otimes \tilde{\pi}^+(y\otimes u^+)))+\tilde{\pi }^{+ }(z_{n-1}\otimes \tilde{\pi}^+(y\otimes \tilde{\pi }^{+ }(x_1\otimes u^+)))\notag\\
&\qquad+\tilde{\pi }^{+ }(x_1\otimes \tilde{\pi}^+(y\otimes \tilde{\pi}^+(z_{n-1}\otimes u^+)))-\tilde{\pi }^{+ }(y\otimes \tilde{\pi}^+(x_1\otimes \tilde{\pi}^+(z_{n-1}\otimes u^+)))\notag\\
&\quad =\tilde{\pi }^{+ }(x_1\otimes \tilde{\pi}^+(z_{n-1}\otimes \tilde{\pi}^+(y\otimes u^+)))-\tilde{\pi }^{+ }(z_{n-1}\otimes \tilde{\pi}^+(x_1\otimes \tilde{\pi }^{+ }(y\otimes u^+)))\notag\\
&\qquad-\tilde{\pi }^{+ }(y\otimes \tilde{\pi}^+(x_1\otimes \tilde{\pi}^+(z_{n-1}\otimes u^+)))+\tilde{\pi }^{+ }(y\otimes \tilde{\pi}^+(z_{n-1}\otimes \tilde{\pi }^{+ }(x_1\otimes u^+)))\notag\\
&\quad =\tilde{\pi }^{+ }([x_1,z_{n-1}]\otimes \tilde{\pi}^+(y\otimes u^+))-\tilde{\pi }^{+ }(y\otimes \tilde{\pi}^+([x_1,z_{n-1}]\otimes u^+))\notag\\
&\quad =\tilde{\pi }^{+ }(p_{n-1}(x_1\otimes z_{n-1})\otimes \tilde{\pi}^+(y\otimes u^+))-\tilde{\pi }^{+ }(y\otimes \tilde{\pi}^+(p_{n-1}(x_1\otimes z_{n-1})\otimes u^+)).
\end{align*}
Thus, we have our result for any $n\geq 0$.
\par Similarly, we can obtain our result for any $n\leq -1$. 
This completes the proof.
\end{proof}
From \rm Proposition \ref{pr_tilde_pi_rep}, we have the following theorem.
\begin{theo}\label{th;pi_positivelyextend}
The vector space $\tilde{U}^+=\bigoplus _{m\in\mathbb{Z}}U_m ^+=\bigoplus _{m\geq 0}U_m ^+$ has a structure of a positively graded $L(\mathfrak{g},\rho,V,{\cal V},B_0)$-module whose representation is $\tilde{\pi }^{+ }$.
We call the module $(\tilde{\pi }^{+ },\tilde{U}^+)$ {\rm the positive extension of $U$ with respect to a standard pentad $(\mathfrak{g},\rho,V,{\cal V},B_0)$}.
(This is a special case of \cite [Theorem 1.2]{Shen}.)
\end{theo}
By the same argument, we can obtain a negatively graded Lie module of $L(\mathfrak{g},\rho,V,{\cal V},B_0)$.
\begin{defn}\label{defn;pi^-}
We define the following linear maps:
\begin{align*}
\pi _{0,m}^-:V_0\otimes U_{m}^-\rightarrow U_{m}^-,
\quad \pi _{1,m}^-:V_1\otimes U_{m}^-\rightarrow U_{m+1}^-,
\quad \pi _{-1,m}^-:V_{-1}\otimes U_{m}^-\rightarrow U_{m-1}^-
\end{align*}
by:
\begin{align}
&\pi _{0,m}^-(a\otimes u_{m}^-):=\pi _{m}^-(a\otimes  u_{m}^-),\label{defn_pi^0_m-}\\
&\pi _{1,m}^-(x_1\otimes u_{m}^-):=\begin{cases}0&(m\geq 0)\\u_m^-(x_1)&(m\leq -1)\end{cases},\label{defn_pi^1_m-}\\
&\pi _{-1,m}^-(y_{-1}\otimes u_{m}^-):=\begin{cases}0&(m\geq 1)\\r_{m}^-(y_{-1}\otimes u_m^-)&(m\leq 0)\end{cases}\label{defn_pi^-1_m-}
\end{align}
where $m\in\mathbb{Z}$, $a\in V_0$, $x_1\in V_1$, $y_{-1}\in V_{-1}$ and $u_{m}^-\in U_{m}^-$.
\end{defn}
\begin{theo}\label{th;pi_negativelyextend}
The vector space $\tilde{U}^-=\bigoplus _{m\in\mathbb{Z}}U_m ^-=\bigoplus _{m\leq 0}U_m ^-$ has a structure of a negatively graded $L(\mathfrak{g},\rho,V,{\cal V},B_0)$-module whose representation is $\tilde{\pi}^{- }$.
We call the module $(\tilde{\pi }^{- },\tilde{U}^-)$ {\rm the negative extension of $U$ with respect to a standard pentad $(\mathfrak{g},\rho,V,{\cal V},B_0)$}.
(This is a special case of \cite [Theorem 1.2]{Shen}.)
\end{theo}
Note that an arbitrary module of $L(\mathfrak{g},\rho,V,{\cal V},B_0)$ is not necessary written in the form of $\tilde{U}^+$ or $\tilde{U}^-$.
For example, the adjoint representation of a loop algebra $L(\mathfrak{sl}_2,\mathrm{ad},\mathfrak{sl}_2,\mathfrak{sl}_2,K_{\mathfrak{sl}_2})={\cal L}(\mathfrak{sl}_2(\mathbb{C}))=\mathbb{C}[t,t^{-1}]\otimes \mathfrak{sl}_2(\mathbb{C})=\bigoplus _{n\in\mathbb{Z}}\mathbb{C}t^n\otimes \mathfrak{sl}_2$, where $K_{\mathfrak{sl}_2}$ is the Killing form of $\mathfrak{sl}_2$, cannot be written in the form of positively or negatively graded module.
Indeed, ${\cal L}(\mathfrak{sl}_2(\mathbb{C}))$ does not have a non-zero element which commutes with any element  of the form $t^{ }\otimes X$ or $t^{-1}\otimes X$ $(X\in \mathfrak{sl}_2)$.

\begin{pr}\label{pr;stapliemod}
Under the notation of Theorems \ref{th;pi_positivelyextend} and \ref{th;pi_negativelyextend}, $L(\mathfrak{g},\rho,V,{\cal V},B_0)$-modules $\tilde{U}^+$ and $\tilde{ U}^-$ have the following properties:
\begin{align}
&\text{$\tilde{U}^+$ and $\tilde{ U}^-$ are transitive,}\label{pr;stapliemod_property_1}\\
&\text{$\tilde{U}^+$ and $\tilde{ U}^-$ are $L(\mathfrak{g},\rho,V,{\cal V},B_0)$-irreducible if and only if $U=U^+_0=U^-_0$ is $\mathfrak{g}$-irreducible.}\label{pr;stapliemod_property_2}
\end{align}
(This is a special case of \cite [Theorem 1.1]{Shen}.)
\end{pr}
\begin{proof}
By the definition, we can show (\ref{pr;stapliemod_property_1}) immediately.
\par Let us show (\ref{pr;stapliemod_property_2}).
Assume that $U$ is an irreducible $\mathfrak{g}$-module.
Let $\underline{W}$ be an arbitrary non-zero $L(\mathfrak{g},\rho,V,{\cal V},B_0)$-submodule of $\tilde{U}^+$.
Then we have that $\underline{W}\cap U^+_0\neq \{0\}$ (cf. \cite [Corollary 1.2]{Shen}).
In fact, take a non-zero element $\underline {w}\in \underline{W}$.
Then there exist integers $0\leq m_1<\cdots <m_k$ and $\underline{w}_{m_1}\in \underline{W}\cap U^+_{m_1},\ldots ,\underline{w}_{m_k}\in \underline{W}\cap U^+_{m_k}$ such that $\underline{w}=\underline {w}_{m_1}+\cdots +\underline {w}_{m_k}$.
Since $\tilde{U}^+$ is transitive, we can take $y_{-1}^1,\ldots ,y_{-1}^{m_k}\in V_{-1}$ such that $0\neq \tilde{\pi }^+(y_{-1}^1\otimes \cdots \otimes \tilde{\pi }^+(y_{-1}^{m_k}\otimes w)\cdots )\in \underline{W}\cap U^+_0$.
By the assumption that $U$ is irreducible, we have $\underline {W}\cap U^+_0=U$.
Since $\tilde{U}^+$ is generated by $U=U_0^+$ and $V_0,V_1$, we have that $\underline {W}$ coincides with $\tilde{U}^+$.
\par Conversely, assume that $\tilde{U}^+$ is an irreducible $L(\mathfrak{g},\rho,V,{\cal V},B_0)$-module.
Take a non-zero $\mathfrak{g}$-submodule $W$ of $U$.
Then a submodule $\underline{W}$ of $\tilde{U}^+$ which is generated by $V_0,V_1, W$ is a non-zero $L(\mathfrak{g},\rho,V,{\cal V},B_0)$-submodule of $\tilde{U}^+$.
Thus, $\underline {W}=\tilde{U}^+$, and, in particular, $W=\underline{W}\cap U_0^+=U$.
Similarly, we can show (\ref{pr;stapliemod_property_2}) for the negative extension $\tilde{U}^-$.
\end{proof}

\begin{ex}\label{ex;2}
We retain to use the notation of {\rm Example} \ref{ex;1}.
Put $U:=\mathbb{C}$ and define a representation $\pi:\mathfrak{g}\otimes U\rightarrow U$ by:
\begin{align*}
\pi ((a,b,A)\otimes u):=au
\end{align*}for any $u\in U$.
Then, the positive extension $\tilde{U}^+$ of $U$ with respect to $(\mathfrak{g},\rho,V,{\cal V},B_0)$ is $3$-dimensional irreducible representation of $L(\mathfrak{g},\rho,V,{\cal V},B_0)=V_{-1}\oplus V_0\oplus V_1\simeq \mathfrak{gl}_1\oplus \mathfrak{sl}_3$.
In fact, for any $v\in V_1=V$, $\phi \in V_{-1}={\cal V}$ and $u\in U$, we have 
\begin{align*}
\tilde{\pi }^+(\phi \otimes \tilde{\pi }^+(v\otimes u))=-\tilde{\pi }^+(\Phi _{\rho }(v\otimes \phi )\otimes u)=-\pi ((-{}^tv\phi ,\frac{3}{2}{}^tv\phi ,v{}^t\phi -\frac{1}{2}{}^tv\phi I_2)\otimes u)={}^tv\phi u.
\end{align*}
Thus, the element $\tilde{\pi }^+(v\otimes u)$ can be identified with $uv\in V_1=V$ via $\langle \cdot,\cdot\rangle _V$, in particular, $U_1^+$ is $2$-dimensional.
Moreover, we have
\begin{align*}
&\tilde{\pi }^+(\phi \otimes\tilde{\pi }^+(v^{\prime }\otimes  \tilde{\pi }^+(v\otimes u)))\\
&\quad =-\tilde{\pi }^+((-{}^tv^{\prime }\phi ,\frac{3}{2}{}^tv^{\prime }\phi ,v^{\prime }{}^t\phi -\frac{1}{2}{}^tv^{\prime }\phi I_2)\otimes \tilde{\pi }^+(v\otimes u))+\tilde{\pi }^+(v^{\prime } \otimes\tilde{\pi }^+(\phi \otimes  \tilde{\pi }^+(v\otimes u)))\\
&\quad =-\tilde{\pi }^+({}^tv^{\prime }\phi\cdot v \otimes u)-\tilde{\pi }^+({}^tv\phi \cdot v^{\prime }\otimes u)+\tilde{\pi }^+(v\otimes {}^tv^{\prime }\phi u)+\tilde{\pi }^+(v^{\prime }\otimes {}^tv\phi u)\\
&\quad =0
\end{align*}for any $v,v^{\prime }\in V_1$, $\phi \in V_{-1}$ and $u\in U$.
Therefore, the positive extension $\tilde{U}^+=U_0^+\oplus U_1^+$ is a $3$-dimensional irreducible representation (see Proposition \ref{pr;stapliemod}).
\end{ex}

The positive and negative extensions of $U$ are characterized by the transitivity.
\begin{theo}\label{th;universalitystap}
Let $(\mathfrak{g},\rho,V,{\cal V},B_0)$  be a standard pentad.
Let $(\underline{\pi},\underline{U})=(\underline {\pi }, \bigoplus _{m\geq 0}\underline {U}_m)$ (respectively $(\underline{\varpi},\underline{\cal U})=(\underline {\varpi }, \bigoplus _{m\leq 0}\underline {\cal U}_m)$) be a positively graded Lie module (respectively a negatively graded Lie module) of $L(\mathfrak{g},\rho,V,{\cal V},B_0)$.
If the $L(\mathfrak{g},\rho,V,{\cal V},B_0)$-module $(\underline{\pi},\underline{U})$ (respectively $(\underline{\varpi},\underline{\cal U})$) is transitive and generated by $V_0$, $V_1$ and $\underline {U}_0$ (respectively generated by $V_0$, $V_{-1}$ and $\underline{\cal U}_0$), then $\underline {U}$ is isomorphic to the positive extension of $\underline {U}_0$ with respect to $(\mathfrak{g},\rho,V,{\cal V},B_0)$ (respectively $\underline {\cal U}$ is isomorphic to the negative extension of $\underline {\cal U}_0$ with respect to $(\mathfrak{g},\rho,V,{\cal V},B_0)$).
(This is a special case of \cite [Theorem 1.2]{Shen}.)
\end{theo}
\begin{proof}
We denote the positive extension of $\underline {U}_0$ with respect to $(\mathfrak{g},\rho,V,{\cal V},B_0)$ by 
\begin{align*}
\widetilde{\underline{U}_0}^+=\bigoplus _{m\geq 0}(\underline {U})^+_m
\end{align*}
and the canonical representation of $L(\mathfrak{g},\rho,V,{\cal V},B_0)$ on $\widetilde{\underline {U}_0}^+$ by $\widetilde{\underline {\pi }}^+$.
Note that $(\underline{U})^+_0=\underline{U}_0$.
We let $\tau _0:(\underline{U})^+_0\rightarrow \underline{U}_0$ be the identity map on $(\underline{U})^+_0=\underline{U}_0$ and define linear maps $\tau _i:(\underline{U})^+_i\rightarrow \underline{U}_i$ by 
\begin{align*}
\tau _i(r_{i-1}^+(x_1\otimes \underline{u}_{i-1}^+)):=\underline{\pi }(x_1\otimes \tau _{i-1}(\underline{u}_{i-1}^+))
\end{align*}
for $i\geq 1$ and any $x_1\in V_1$ and $\underline{u}_{i-1}^+\in (\underline {U})_{i-1}^+$ inductively.
These $\tau _i$'s are well-defined and satisfy the following equation:
\begin{align}
&\tau _{i+j}(\widetilde{\underline {\pi }}^+(a_j\otimes \underline {u}_i^+))=\underline {\pi }(a_j\otimes \tau _{i}(\underline {u}_i^+))\label{eq;univ_ijcondi}
\end{align}for $j=0,\pm 1$ and any $a_j\in V_j$, $\underline {u}_i^+\in (\underline {U})_{i}^+$.
Let us show it by induction on $i$.
It is clear that the equation (\ref{eq;univ_ijcondi}) holds when $i=0$ and $j=0,-1$.
In order to show the equation (\ref{eq;univ_ijcondi}) for $i=0$ and $j=1$, let us show that $\tau _1$ is well-defined.
Take an arbitrary element $y_{-1}\in V_{-1}$, then we have
\begin{align}
&\underline{\pi }(y_{-1}\otimes \underline{\pi }(x_1\otimes  \tau _{0}(\underline{u}_{0}^+)))=\underline{\pi }([y_{-1},x_1]\otimes \tau _{0}(\underline{u}_{0}^+))+\underline{\pi }(x_1\otimes \underline{\pi }(y_{-1}\otimes \tau _{0}(\underline{u}_{0}^+)))\notag\\
&\quad =\underline{\pi }([y_{-1},x_1]\otimes \tau _{0}(\underline{u}_{0}^+))=\tau _0(\widetilde{\underline {\pi }}^+([y_{-1},x_1]\otimes \underline{u}_{0}^+))=\tau _0 (\widetilde{\underline {\pi }}^+(y_{-1}\otimes r_0^+(x_1\otimes \underline{u}_{0}^+))).\label{eq;universality_welldef_i1}
\end{align}
Thus, if $x_1^1,\ldots, x_1^l\in V_1$ and $\underline{u}_{0}^{+1},\ldots ,\underline{u}_{0}^{+l}\in (\underline {U})_{0}^+$ satisfy $\sum _{s=1}^lr_0^+(x_1^s\otimes \underline{u}_{0}^{+s})=0$, then we have
\begin{align*}
\sum _{s=1}^l\underline{\pi }(y_{-1}\otimes \underline{\pi }(x_1^s\otimes  \tau _{0}(\underline{u}_{0}^{+s})))=0
\end{align*}
for any $y_{-1}\in V_{-1}$.
Since $(\underline{\pi},\underline{U})$ is transitive, it follows that $\sum _{s=1}^l\underline{\pi }(x_1^s\otimes \tau _{0}(\underline{u}_{0}^{+s}))=0$, and, thus, we have the well-definedness of $\tau _1$.
By the equation (\ref{eq;universality_welldef_i1}), we can obtain the equation (\ref{eq;univ_ijcondi}) where $i=0$ and $j=1$.
\par Let $i\geq 1$ and assume that $\tau _0,\ldots ,\tau _{i}$ are well-defined and that $\tau _{i}$ satisfies the equation (\ref{eq;univ_ijcondi}) for $j=0,-1$.
Then for any $y_{-1}\in V_{-1}$, we have
\begin{align}
&\underline{\pi }(y_{-1}\otimes \underline{\pi }(x_1\otimes \tau _{i}(\underline{u}_{i}^+)))=\underline{\pi }([y_{-1},x_1]\otimes \tau _{i}(\underline{u}_{i}^+))+\underline{\pi }(x_1\otimes \underline{\pi }(y_{-1}\otimes \tau _{i}(\underline{u}_{i}^+)))\notag\\
&\quad =\tau _{i}(\widetilde{\underline {\pi }}^+([y_{-1},x_1]\otimes \underline{u}_{i}^+))+\tau _{i}(\widetilde{\underline {\pi }}^+(x_1\otimes \widetilde{\underline {\pi }}^+(y_{-1}\otimes \underline{u}_{i}^+)))\notag\\
&\quad =\tau _{i}(\widetilde{\underline {\pi }}^+(y_{-1}\otimes \widetilde{\underline {\pi }}^+(x_1\otimes \underline{u}_{i}^+)))=\tau _{i} (\widetilde{\underline {\pi }}^+(y_{-1}\otimes r_0^+(x_1\otimes \underline{u}_{i}^+))).\label{eq;universality_welldef_i+1}
\end{align}
Thus, by the same argument to the argument of the case where $i=0$ and $j=1$, we have the well-definedness of $\tau _{i+1}$, i.e. $\tau _{i}$ satisfies satisfies the equation (\ref{eq;univ_ijcondi}) for $j=1$, and that $\tau _{i+1}$ satisfies the equation (\ref{eq;univ_ijcondi}) for $j=-1$.
Moreover, by a similar argument to the argument of (\ref{eq;universality_welldef_i+1}), we have that $\tau _{i+1}$ satisfies the equation (\ref{eq;univ_ijcondi}) for $j=0$.
Therefore, by induction on $i$, we can obtain the well-definedness of $\tau _i$ and the equation (\ref{eq;univ_ijcondi}) for all $i\geq 0$ and $j=0,\pm 1$.
\par We define a linear map $\tau :\widetilde{\underline{U}_0}^+\rightarrow \underline{U}$ by
\begin{align}
\tau (\underline {u}_i^+):=\tau _i(\underline {u}_i^+)
\end{align}for any $i\geq 0$ and $\underline{u}_i^+\in (\underline {U})_i^+$.
This $\tau $ is an isomorphism of $L(\mathfrak{g},\rho,V,{\cal V},B_0)$-modules.
In fact, by the assumption that $\underline {U}$ is generated by $V_1$ and $\underline {U}_0$, we have the surjectivity of $\tau $.
Moreover, by the equation (\ref{eq;univ_ijcondi}) in the cases where $i\geq 1$ and $j=-1$ and the definition of $\tau _0$ and the transitivity of the positive extension of $\widetilde{\underline{U}_0}^+$, we have the injectivity of $\tau $.
Thus, $\tau $ is bijective.
Moreover, since $L(\mathfrak{g},\rho,V,{\cal V},B_0)$ is generated by $V_0$, $V_{\pm 1}$, it follows that $\tau $ is a homomorphism of $L(\mathfrak{g},\rho,V,{\cal V},B_0)$-modules from the equation (\ref{eq;univ_ijcondi}).
Therefore $\underline {U}$ is isomorphic to $\widetilde{\underline{U}_0}^+$  as $L(\mathfrak{g},\rho,V,{\cal V},B_0)$-modules.
\par By the same argument, we can prove our claim for $(\underline{\varpi},\underline{\cal U})$.
\end{proof}
As an application of Theorem \ref{th;universalitystap}, we have the following proposition.
\begin{pr}
Let $(\mathfrak{g},\rho,V,{\cal V},B_0)$ be a standard pentad and $U,W$ (respectively ${\cal U}, {\cal W}$) be $\mathfrak{g}$-modules.
Then the positive extension of $U\oplus W$ (respectively the negative extension of ${\cal U}\oplus {\cal W}$) with respect to $(\mathfrak{g},\rho,V,{\cal V},B_0)$ is isomorphic to a direct sum of positive extensions of $U$ and $W$ (respectively negative extensions of ${\cal U}$ and ${\cal W}$) with respect to $(\mathfrak{g},\rho,V,{\cal V},B_0)$, i.e.
\begin{align*}
(\widetilde{U\oplus W})^+\simeq \tilde{U}^+\oplus \tilde{W}^+\quad
(respectively\ (\widetilde{{\cal U}\oplus {\cal W}})^-\simeq \tilde{{\cal U}}^-\oplus \tilde{{\cal W}}^-).
\end{align*}
\end{pr}

\subsection {A pairing between $(\tilde{\pi }^{+ },\tilde{U}^+)$ and $(\tilde{\varpi }^{- },\tilde{{\cal U}}^-)$}\label{sec;stapext_1}
In the previous section, we constructed positively and negatively graded $L(\mathfrak{g},\rho,V,{\cal V},B_0)$-modules.
Next, let us try to embed these modules into some graded Lie algebra.
For this, we need to embed $L(\mathfrak{g},\rho,V,{\cal V},B_0)$ and $(\tilde{\pi }^{+ },\tilde{U}^+)$ into some standard pentad.
However, as mentioned in Remark \ref{rem;mght_not_stap}, the objects $L(\mathfrak{g},\rho,V,{\cal V},B_0)$ and $\tilde{U}^+$ might not have a submodule of $\mathrm {Hom}(\tilde{U}^+,\mathfrak{k})$ and a bilinear form on $L(\mathfrak{g},\rho,V,{\cal V},B_0)$ satisfying the conditions (\ref{defn;condi_stap_pairing}) and (\ref{defn;condi_stap_Phimap}).
In the present and the next sections, we only consider the cases where $B_0$ is symmetric and $U$ has a submodule ${\cal U}\subset \mathrm {Hom }(U,\mathfrak{k})$ such that $(\mathfrak{g},\pi,U,{\cal U},B_0)$ is standard.
Then, we can show that a pentad $(L(\mathfrak{g},\rho,V,{\cal V},B_0),\tilde{\pi }^+,\tilde{U }^+,\tilde{\cal U}^-,B_L)$ is standard.
First, in this section, we consider the negative extension $\tilde{\cal U}^-$ of ${\cal U}$ and construct a non-degenerate invariant bilinear form $\tilde{U}^+\times \tilde{\cal U}^-\rightarrow \mathfrak{k}$ under the assumption (\ref{defn;condi_stap_pairing}) inductively (cf. \cite [Remark 1.4]{Shen}).
In the next section, we shall construct the $\Phi $-map of the pentad $(L(\mathfrak{g},\rho,V,{\cal V},B_0),\tilde{\pi }^+,\tilde{U }^+,\tilde{\cal U}^-,B_L)$. 
\begin{defn}\label{defn;pairing_stap_rep}
Let $(\tilde{\pi }^{+ },\tilde{U}^+)$ and $(\tilde{\varpi }^{- },\tilde{{\cal U}}^-)$, ${\cal U}\subset \mathrm {Hom}(U,\mathfrak{k})$ be $\mathfrak{g}$-modules such that the restriction of the canonical pairing $\langle \cdot,\cdot\rangle _0:U\times {\cal U}\rightarrow \mathfrak{k}$ is non-degenerate, and, let $\tilde{U}^+$ and $\tilde{\cal U}^-$ be the positive and negative extensions of $U$ and ${\cal U}$ respectively.
We define a bilinear map $\langle \cdot,\cdot\rangle _{0}^0$ by:
\begin{align}
\langle \cdot,\cdot \rangle _{0}^0:&U_{0}^+\times {\cal U}_{0}^-\rightarrow \mathfrak{k}\notag\\
&(u_{0}^+,w_0^-)\mapsto \langle u_{0}^+,w_{0}^-\rangle _0.
\end{align}
Moreover, for $i\geq 1$, we define a bilinear map $\langle \cdot,\cdot\rangle _{-i}^i$ by:
\begin{align}
\langle \cdot,\cdot \rangle _{-i}^i:&U_{i}^+\times {\cal U}_{-i}^-\rightarrow \mathfrak{k}\notag\\
&(r_{i-1}^+(x_1\otimes u_{i-1}^+),r_{-i+1}^-(y_{-1}\otimes w_{-i+1}^-))\mapsto -\langle \tilde{\pi }^+(y_{-1}\otimes r_{i-1}^+(x_1\otimes u_{i-1}^+)),w_{-i+1}^-\rangle^{i-1}_{-i+1}\label{defn;pairing_i_-i}
\end{align}inductively.
\end{defn}
The well-definedness of Definition \ref{defn;pairing_stap_rep} can be obtained by the following proposition.
\begin{pr}\label{pr;welldef_bra}
Let $j\geq 0$.
Assume that the bilinear map $\langle \cdot,\cdot\rangle _{-j}^{j}$ defined in (\ref{defn;pairing_i_-i}) is well-defined and satisfies the following equations:
\begin{align}
&\langle \tilde{\pi }^+(a\otimes u_{j}^+),w_{-j}^-\rangle ^{j}_{-j}+\langle u_{j}^+,\tilde{\pi }^-(a\otimes w_{-j}^-)\rangle ^{j}_{-j}=0\label{pr;welldef_pair_1},\\
&\langle \tilde{\pi }^+(y_{-1}\otimes \tilde{\pi }^+(x_1\otimes u_{j}^+)),w_{-j}^-\rangle ^{j}_{-j}=\langle u_{j}^+,\tilde{\pi }^-(x_1\otimes \tilde{\pi }^-(y_{-1}\otimes w_{-j}^-))\rangle ^{j}_{-j}\label{pr;welldef_pair_2}\\
&\langle \tilde{\pi }^+(x_1\otimes u_{j-1}^+),w_{-j}^-\rangle ^{j}_{-j}=\begin{cases}-\langle u_{j-1}^+,\tilde{\pi }^-(x_1\otimes w_{-j}^-)\rangle ^{j-1}_{-j+1}&(j\geq 1)\\0&(j=0)\end{cases}\label{pr;welldef_pair_2-2}
\end{align}
for any $a\in \mathfrak{g}=V_0\subset L(\mathfrak{g},\rho,V,{\cal V},B_0)$, $x_1\in V_1$, $y_{-1}\in V_{-1}$, $u_{j-1}^+\in U^+_{j-1}$, $u_{j}^+\in U^+_{j}$ and $w_{-j}^-\in {\cal U}^-_{-j}$.
Then the bilinear map $\langle \cdot,\cdot\rangle ^{j+1}_{-j-1}$ defined in (\ref{defn;pairing_i_-i}) is also well-defined and satisfies the following equations:
\begin{align}
&\langle \tilde{\pi }^+(a\otimes u_{j+1}^+),w_{-j-1}^-\rangle ^{j+1}_{-j-1}+\langle u_{j+1}^+,\tilde{\pi }^-(a\otimes w_{-j-1}^-)\rangle ^{j+1}_{-j-1}=0,\label{pr;welldef_pair_3}\\
&\langle \tilde{\pi }^+(y_{-1}\otimes \tilde{\pi }^+(x_1\otimes u_{j+1}^+)),w_{-j-1}^-\rangle ^{j+1}_{-j-1}=\langle u_{j+1}^+,\tilde{\pi }^-(x_1\otimes \tilde{\pi }^-(y_{-1}\otimes w_{-j-1}^-))\rangle ^{j+1}_{-j-1}\label{pr;welldef_pair_4}\\
&\langle \tilde{\pi }^+(x_1\otimes u_{j}^+),w_{-j-1}^-\rangle ^{j+1}_{-j-1}=-\langle u_j^+,\tilde{\pi }^-(x_1\otimes w_{-j-1}^-)\rangle ^{j}_{-j}\label{pr;welldef_pair_5}
\end{align}
for any $a\in \mathfrak{g}=V_0\subset L(\mathfrak{g},\rho,V,{\cal V},B_0)$, $x_1\in V_1$, $y_{-1}\in V_{-1}$, $u_j^+\in U_{j}^+$, $u_{j+1}^+\in U^+_{j+1}$ and $w_{-j-1}^-\in {\cal U}^-_{-j-1}$.
\end{pr}
\begin{proof}
First, we let $j=0$.
It is clear that $\langle \cdot,\cdot\rangle ^0_0$ satisfies (\ref{pr;welldef_pair_1}) and (\ref{pr;welldef_pair_2-2}).
Let us show that $\langle \cdot,\cdot\rangle ^0_0$ satisfies (\ref{pr;welldef_pair_2}).
Indeed, under the above notation, we have
\begin{align}
\langle \tilde{\pi }^+(y_{-1}\otimes \tilde{\pi }^+(x_1\otimes u_{0}^{+})),w_{0}\rangle ^{0}_{0}=&\langle \tilde{\pi }^+([y_{-1},x_1]\otimes u_{0}^{+}),w_{0}\rangle ^{0}_{0}= \langle u_{0}^{+},\tilde{\pi }^-([x_1,y_{-1}]\otimes w_{0})\rangle ^{0}_{0} \notag\\
=&\langle u_{0}^{+},\tilde{\pi }^-(x_1\otimes \tilde{\pi }^-(y_{-1}\otimes w_0))\rangle ^{0}_{0}.\label{eq;well-pair_0_0}
\end{align}
Thus, the bilinear map $\langle \cdot,\cdot\rangle ^0_0$ satisfies the assumptions of Proposition \ref{pr;welldef_bra}.
\par Next, let us show that the bilinear map $\langle \cdot,\cdot\rangle ^1_{-1}$ is well-defined.
Take arbitrary natural numbers $\nu,\mu\in \mathbb{N}$ and elements $x_1^1,\ldots ,x_1^{\nu}\in V_1$, $u_{0}^{+,1},\ldots ,u_{0}^{+,\nu}\in U_0^+$, $y_{-1}^1,\ldots ,y_{-1}^{\mu }\in V_{-1}$, $w_{0}^{-,1},\ldots ,w_{0}^{-,\mu}\in {\cal U}_0^-$ satisfying 
\begin{align*}
\sum _{s=1}^{\nu }r_{0}^+(x_1^s\otimes u_{0}^{+,s})=0, \quad \sum _{t=1}^{\mu }r_{0}^-(y_{-1}^t\otimes w_{0}^{-,t})=0.
\end{align*}
Then, for any $y_{-1}\in V_{-1}$, $w_{0}^{-}\in {\cal U}_{0}^-$, $x_1\in V_1$ and $u_0^+\in U_0^+$, we have
\begin{align}
\langle \sum _{s=1}^{\nu } r_{0}^+(x_1^s\otimes u_{0}^{+,s})(y_{-1}),w_{0}^-\rangle ^{0}_{0}=0,\label{eq;well-def_pairing_j=1+}
\end{align}
and, by the equation (\ref{eq;well-pair_0_0}), we have
\begin{align}
\sum _{t=1}^{\mu }\langle r_{0}^+(x_1\otimes u_{0}^{+})(y_{-1}^{t}),w_{0}^{-,t}\rangle ^{0}_{0}=\sum _{t=1}^{\mu }\langle u_{0}^{+},r_{0}^-(y_{-1}^{t}\otimes w_{0}^{-,t})(x_1)\rangle ^{0}_{0}=0.\label{eq;well-def_pairing_j=1-}
\end{align}
By (\ref{eq;well-def_pairing_j=1+}) and (\ref{eq;well-def_pairing_j=1-}), we can obtain that $\langle \cdot,\cdot\rangle ^1_{-1}$ is well-defined.
\par Let us consider properties of $\langle\cdot,\cdot\rangle _{-1}^1$.
By (\ref{eq;well-pair_0_0}), we have that $\langle\cdot,\cdot\rangle _{-1}^1$ satisfies 
\begin{align}
\langle \tilde{\pi }^+(x_1\otimes u_0^+),w_{-1}^-\rangle _{-1}^1=-\langle u_0^+,\tilde{\pi }^-(x_1\otimes w_{-1}^-)\rangle _{-1}^1\label{eq;well-pair_0_0_eq1}
\end{align}
for any $x_1\in V_1$, $u_0^+\in U_0^+$ and $w_{-1}^-\in {\cal U}_{-1}^-$, i.e. $\langle \cdot,\cdot\rangle ^1_{-1}$ satisfies the equation (\ref{pr;welldef_pair_5}).
Moreover, $\langle \cdot,\cdot\rangle ^1_{-1}$ satisfies the equations (\ref{pr;welldef_pair_3}) and (\ref{pr;welldef_pair_4}).
In fact, for all $a\in V_0$, $x_1\in V_1$, $y_{-1}\in V_{-1}$, $u_0^+\in U_0^+$ and $w_0^-\in {\cal U}_0^-$, we have
\begin{align}
&\langle \tilde{\pi }^+(a\otimes r_0^+(x_1\otimes u_0^+)),r _{0}^-(y_{-1}\otimes w_0^-)\rangle ^1_{-1}=-\langle \tilde{\pi }^+(y_{-1}\otimes \tilde{\pi }^+(a\otimes \tilde{\pi }^+(x_1\otimes u_0^+))), w_0^-\rangle ^0_{0}\notag\\
&\quad =-\langle \tilde{\pi }^+(a\otimes \tilde{\pi }^+(y_{-1}\otimes \tilde{\pi }^+(x_1\otimes u_0^+))), w_0^-\rangle ^0_{0}+\langle  \tilde{\pi }^+([a,y_{-1}]\otimes \tilde{\pi }^+(x_1\otimes u_0^+)),w_0^-\rangle ^0_{0}\notag\\
&\quad =\langle \tilde{\pi }^+(y_{-1}\otimes \tilde{\pi }^+(x_1\otimes u_0^+)),\tilde{\pi }^-(a\otimes w_0^-)\rangle ^0_{0}-\langle \tilde{\pi }^+(x_1\otimes u_0^+),\tilde{\pi }^-([a,y_{-1}]\otimes w_0^-)\rangle ^1_{-1}\notag\\
&\quad =-\langle \tilde{\pi }^+(x_1\otimes u_0^+),\tilde{\pi }^-(y_{-1}\otimes \tilde{\pi }^-(a\otimes w_0^-))\rangle ^1_{-1}-\langle \tilde{\pi }^+(x_1\otimes u_0^+),\tilde{\pi }^-([a,y_{-1}]\otimes w_0^-)\rangle ^1_{-1}\notag\\
&\quad =-\langle r_0^+(x_1\otimes u_0^+),\tilde{\pi }^-(a\otimes r_{0}^-(y_{-1}\otimes w_0^-))\rangle ^1_{-1}.\label{eq;well-pair_0_0_eq2}
\end{align}
Thus $\langle \cdot,\cdot\rangle ^1_{-1}$ satisfies (\ref{pr;welldef_pair_3}).
And, from (\ref{eq;well-pair_0_0_eq1}) and (\ref{eq;well-pair_0_0_eq2}), we have 
\begin{align}
&\langle \tilde{\pi }^+(y_{-1}\otimes \tilde{\pi }^+(x_1\otimes u_1^+)), w_{-1}^-\rangle ^1_{-1}\notag\\
&\quad =\langle \tilde{\pi }^+([y_{-1},x_1]\otimes u_1^+), w_{-1}^-\rangle ^1_{-1}+\langle \tilde{\pi }^+(x_1\otimes \tilde{\pi }^+(y_{-1}\otimes u_1^+)),w_{-1}^-\rangle^1_{-1}\notag\\
&\quad =-\langle u_1^+, \tilde{\pi }^-([y_{-1},x_1]\otimes w_{-1}^-)\rangle ^1_{-1}-\langle  \tilde {\pi }^+(y_{-1}\otimes u_1^+),\tilde{\pi }^-(x_1\otimes w_{-1}^-)\rangle^0_{0}\notag\\
&\quad =\langle u_1^+, \tilde{\pi }^-([x_1,y_{-1}]\otimes w_{-1}^-)\rangle ^1_{-1}+\langle  u_1^+,\tilde{\pi }^-(y_{-1}\otimes \tilde{\pi }^-(x_1\otimes w_{-1}^-))\rangle^1_{-1}\notag\\
&\quad =\langle u_1^+,\tilde{\pi }^-(x_1\otimes \tilde{\pi }^-(y_{-1}\otimes w_{-1}^-))\rangle ^1_{-1}
\end{align}
for any $x_1\in V_1$, $y_{-1}\in V_{-1}$, $u_1^+\in U_1^+$ and $w_{-1}^-\in {\cal U}_{-1}^-$.
Thus $\langle \cdot,\cdot\rangle ^1_{-1}$ satisfies (\ref{pr;welldef_pair_4}).
\par We let $j\geq 1$.
Suppose that the bilinear map $\langle \cdot,\cdot\rangle ^{j}_{-j}$ is well-defined and satisfies the equations (\ref{pr;welldef_pair_1}), (\ref{pr;welldef_pair_2}) and (\ref{pr;welldef_pair_2-2}).
Let us show the well-definedness of $\langle \cdot,\cdot\rangle ^{j+1}_{-j-1}$.
Take arbitrary natural numbers $\nu,\mu\in \mathbb{N}$ and elements $x_1^1,\ldots ,x_1^{\nu}\in V_1$, $u_{j}^{+,1},\ldots ,u_{j}^{+,\nu}\in U_0^+$, $y_{-1}^1,\ldots ,y_{-1}^{\mu }\in V_{-1}$, $w_{-j}^{-,1},\ldots ,w_{-j}^{-,\mu}\in {\cal U}_0^-$ satisfying
\begin{align}
\sum _{s=1}^{\nu } r_{j}^+(x_1^s\otimes u_{j}^{+,s})=0,\quad \sum _{t=1}^{\mu } r_{-j}^-(y_{-1}^t\otimes w_{-j}^{-,t})=0.
\end{align}
Then, by the equation (\ref{pr;welldef_pair_2}) and the same argument to the argument of (\ref{eq;well-def_pairing_j=1+}) and (\ref{eq;well-def_pairing_j=1-}), we have the following equations:
\begin{align}
\langle \sum _{s=1}^{\nu } r_{j}^+(x_1^s\otimes u_{j}^{+,s})(y_{-1}),w_{-j}^-\rangle ^{j}_{-j}=0,\quad  \sum _{t=1}^{\mu }\langle r_{j}^+(x_1\otimes u_{j}^{+})(y_{-1}^{t}),w_{-j}^{-,t}\rangle ^{j}_{-j}=0.
\end{align}
Thus, we have that the bilinear map $\langle \cdot,\cdot\rangle ^{j+1}_{-j-1}$ is well-defined.
\par From the equation (\ref{pr;welldef_pair_2}), we have
\begin{align}
\langle  \tilde{\pi }^+(x_1\otimes u_j^+),w_{-j-1}^-\rangle _{-j-1}^{j+1}=-\langle u_j^+,\tilde{\pi }^-(x_1\otimes w_{-j-1}^-)\rangle _{-j}^j\label{eq;well-pair_0_0_eq3}
\end{align}
for any $x_1\in V_1$, $u_j^+\in U_j^+$ and $w_{-j-1}\in {\cal U}_{-j-1}^-$.
We can show that the bilinear map $\langle \cdot,\cdot\rangle ^{j+1}_{-j-1}$ satisfies the equation (\ref{pr;welldef_pair_5}) from the equation (\ref{eq;well-pair_0_0_eq3}) and that it also satisfies the equations (\ref{pr;welldef_pair_3}) and (\ref{pr;welldef_pair_4}) by the same argument to the argument of the case where $j=0$.
\end{proof}
By Proposition \ref{pr;welldef_bra}, we can obtain pairings $\langle \cdot,\cdot\rangle _{-j}^j$ for all $j\geq 0$ inductively.
Then, we can define a pairing between $(\tilde{\pi }^+,\tilde{U}^+)$ and $(\tilde{\varpi }^-,\tilde{\cal U}^-)$.
\begin{defn}\label{defn;bra_posext_negext}
We define a bilinear map $\langle \cdot,\cdot\rangle :\tilde{U}^+\times \tilde{{\cal U}}^-\rightarrow \mathfrak{k}$ by:
\begin{align}
\langle u_n^+,w_{-m}^-\rangle :=\begin{cases}\langle u_n^+,w_{-n}^-\rangle ^n_{-n}&(n=m)\\0&(n\neq m)\end{cases}
\end{align}for any $n,m\geq 0$, $u_n^+\in U_n^+\subset \tilde{U}^+$ and $w_{-m}^-\in {\cal U}_{-m}^-\subset \tilde{{\cal U}}^-$.
\end{defn}
By Definition \ref{defn;pairing_stap_rep} and Proposition \ref{pr;welldef_bra}, we have that $\langle \cdot ,\cdot\rangle $ satisfies 
\begin{align}
\langle \tilde{\pi }^+(z_{j}\otimes \tilde{u}^+),\tilde{w}^-\rangle =-\langle \tilde{u}^+,\tilde{\pi }^-(z_j\otimes \tilde{w}^-)\rangle \label{eq;extpair_eq0}
\end{align}
for $j=0,\pm 1$ and any $z_j\in V_j$, $\tilde{u}^+\in \tilde{U}^+$, $\tilde{w}^-\in \tilde{\cal U}^-$.
\begin{pr}\label{pr;pairing_Linv_nondeg}
The bilinear form $\langle \cdot,\cdot\rangle :\tilde{U}^+\times \tilde{{\cal U}}^-\rightarrow \mathfrak{k}$ is non-degenerate and $L(\mathfrak{g},\rho,V,{\cal V},B_0)$-invariant (cf. \cite [Definition 1.4 and Remark 1.4]{Shen}).
\end{pr}
\begin{proof}
First, let us show that the bilinear form $\langle \cdot,\cdot\rangle $ is non-degenerate.
For this, it is sufficient to show that the bilinear map $\langle \cdot,\cdot\rangle ^j_{-j}:U_j^+\times {\cal U}_{-j}\rightarrow \mathfrak{k}$ is non-degenerate for each $j\geq 0$.
We show it by induction on $j$.
For $j=0$, it follows that $\langle \cdot,\cdot\rangle ^0_0$ is non-degenerate from the assumption.
For $j+1$, we take an element $u_{j+1}^+\in U_{j+1}^+$ which satisfies $\langle u_{j+1}^+,r_{-j}^-(y_{-1}\otimes w_{-j}^-)\rangle ^{j+1}_{-j-1}=0$ for any $y_{-1}\in V_{-1}$ and $w_{-j}^-\in {\cal U}^-_{-j}$.
Then, we have
\begin{align*}
0=\langle u_{j+1}^+,r_{-j}^-(y_{-1}\otimes w_{-j}^- )\rangle ^{j+1}_{-j-1}=-\langle \tilde{\pi }^+(y_{-1}\otimes u_{j+1}^+),w_{-j}^-\rangle ^{j}_{-j}=-\langle u_{j+1}^+(y_{-1}),w_{-j}^-\rangle ^{j}_{-j}.
\end{align*}
By the induction hypothesis that $\langle \cdot ,\cdot \rangle ^j_{-j}$ is non-degenerate, we can obtain that $u_{j+1}^+(y_{-1})=0$ for any $y_{-1}\in V_{-1}$, and, thus, we have $u_{j+1}^+=0\in U^+_{j+1}\subset \mathrm {Hom}(V_{-1},U^+_j)$.
Similarly, we can show that an element $w_{-j-1}^-\in {\cal U}_{-j-1}^-$ which satisfies $\langle r_{j}^+(x_1\otimes u_{j}^+),w_{-j-1}^-\rangle ^{j+1}_{-j-1}=0$ for any $x_1\in V_1$ and $u_i^+\in U_i^+$ is $0$ by (\ref{eq;extpair_eq0}). 
Summarizing the above argument, we can obtain that the map $\langle \cdot,\cdot\rangle ^{j+1}_{-j-1}$ is non-degenerate.
Therefore, by induction, we can obtain that the bilinear map $\langle \cdot,\cdot\rangle :\tilde{U}^+\times \tilde{{\cal U}}^-\rightarrow \mathfrak{k}$ is non-degenerate.
\par Next, let us show that the bilinear map $\langle \cdot,\cdot\rangle $ is $L(\mathfrak{g},\rho,V,{\cal V},B_0)$-invariant.
For this, it is sufficient to show that the following equation holds:
\begin{align}
\langle \tilde{\pi }^+(x_j\otimes u_n^+),w_{-n-j}^-\rangle ^{n+j}_{-n-j}+\langle u_n^+,\tilde{\pi }^-(x_j\otimes w_{-n-j}^-)\rangle ^{n}_{-n}=0\label{eq;pair_L_inv}
\end{align}
for any $j,n\in \mathbb{Z}$, $x_j\in V_j$, $u_n\in U_n^+$ and $w_{-n-j}^-\in {\cal U}_{-n-j}^-$.
We shall show it by induction on $j$.
Assume that $j\geq 0$.
For $j=0,1$, the equation (\ref{eq;pair_L_inv}) follows from (\ref{eq;extpair_eq0}) immediately.
For $j+1$, by induction hypothesis, we have
\begin{align}
&\langle \tilde{\pi }^+([v_1,x_j]\otimes u_n^+),w_{-n-j-1}^-\rangle ^{n+j+1}_{-n-j-1}\notag\\
&\quad =\langle \tilde{\pi }^+(v_1\otimes \tilde{\pi }^+(x_j\otimes u_n^+)),w_{-n-j-1}^-\rangle ^{n+j+1}_{-n-j-1}-\langle \tilde{\pi }^+(x_j\otimes \tilde{\pi }^+(v_1\otimes u_n^+)),w_{-n-j-1}^-\rangle ^{n+j+1}_{-n-j-1}\notag\\
&\quad =-\langle \tilde{\pi }^+(x_j\otimes u_n^+), \tilde{\pi }^-(v_1\otimes w_{-n-j-1}^-)\rangle ^{n+j}_{-n-j}+\langle \tilde{\pi }^+(v_1\otimes u_n^+),\tilde{\pi }^-(x_j\otimes w_{-n-j-1}^-)\rangle ^{n+1}_{-n-1}\notag\\
&\quad =\langle u_n^+, \tilde{\pi }^-(x_j\otimes \tilde{\pi }^-(v_1\otimes w_{-n-j-1}^-))\rangle ^{n}_{-n}-\langle u_n^+,\tilde{\pi }^-(v_1\otimes \tilde{\pi }^-(x_j\otimes w_{-n-j-1}^-))\rangle ^{n}_{-n}\notag\\
&\quad =-\langle u_n^+, \tilde{\pi }^-([v_1,x_j]\otimes w_{-n-j-1}^-)\rangle ^{n}_{-n}
\end{align}
for any $n\in\mathbb{Z}$, $x_1\in V_1$, $v_j\in V_j$, $u_n^+\in U_n^+$ and $w_{-n-j-1}^-\in {\cal U}_{-n-j-1}^-$.
Thus,  by induction, we can show the equation (\ref{eq;pair_L_inv}) for all $j\geq 0$.
Similarly, we can obtain the equation (\ref{eq;pair_L_inv}) for all $j\leq 0$.
Thus, we have the equation (\ref{eq;pair_L_inv}) for all $j\in\mathbb{Z}$.
Therefore the bilinear map $\langle \cdot,\cdot\rangle :\tilde{U}^+\times \tilde{{\cal U}}^-\rightarrow \mathfrak{k}$ is $L(\mathfrak{g},\rho,V,{\cal V},B_0)$-invariant.
\end{proof}
By Proposition \ref{pr;pairing_Linv_nondeg}, we can regard $\tilde{\cal U}^-$ as an $L(\mathfrak{g},\rho,V,{\cal V},B_0)$-submodule of $\mathrm {Hom }(\tilde{U}^+,\mathfrak{k})$.

\subsection{The $\Phi $-map between $(\tilde{\pi }^+,\tilde{U}^+)$ and $(\tilde{\varpi }^-,\tilde{{\cal U}}^-)$}\label{sec;stapext_2}
We retain to assume that a pentad $(\mathfrak{g},\pi,U,{\cal U},B_0)$ is standard and the bilinear form $B_0$ is symmetric.
As I proved in section \ref{sec;stapext_1}, a pentad $(L(\mathfrak{g},\rho,V,{\cal V},B_0),\tilde{\pi }^{+ },\tilde{U}^+,\tilde{{\cal U}}^-,B_L)$ satisfies the condition (\ref{defn;condi_stap_pairing}).
Let us construct the $\Phi $-map of the pentad $(L(\mathfrak{g},\rho,V,{\cal V},B_0),\tilde{\pi }^{+ },\tilde{U}^+,\tilde{{\cal U}}^-,B_L)$ and show that it is standard.
\begin{defn}\label{defn;tildephi}
Assume that pentads $(\mathfrak{g},\rho,V,{\cal V},B_0)$ and $(\mathfrak{g},\pi,U,{\cal U},B_0)$ are standard and that $B_0$ is symmetric.
We define a linear map $\tilde{\Phi }^0_0:U_0^+\otimes {\cal U}_0^-\rightarrow V_0$ as:
\begin{align}
&\tilde{\Phi }^0_0(u_0^+\otimes w_0^-):=\Phi _{\pi }(u_0^+\otimes w_0^-)\label{eq;tildephi_phimap_0_0}
\end{align}where $x_1\in V_1$, $y_{-1}\in V_{-1}$, $u_0^+\in U_0^+$, $w_0^-\in {\cal U}_0^-$ and $\Phi _{\pi }$ is the $\Phi $-map of $(\mathfrak{g},\pi,U,{\cal U},B_0)$.
\par Moreover, for each $i\geq 0$, we inductively define a linear map $\tilde{\Phi }^{i+1}_{0}:U_{i+1}^+\otimes {\cal U}_{0}^-\rightarrow V_{i+1}$ by:
\begin{align}
\tilde{\Phi }^{i+1}_{0}(r_{i}^+(x_1\otimes u_i^+)\otimes w_{0}^-):=[x_1,\tilde{\Phi }^i_{0}(u_i^+\otimes w_{0}^-)],\label{eq;tildephi_phimap_i+1_0}
\end{align}
where $x_1\in V_1$, $y_{-1}\in V_{-1}$, $u_i^+\in U_i^+$ and $w_{0}^-\in {\cal U}_{0}^-$.
\par Assume that an integer $j\geq 0$ satisfies a condition that we have linear maps $\tilde{\Phi }^k_{-j}:U_k^+\otimes {\cal U}_{-j}^-\rightarrow V_{k-j}$ for all $k\geq 0$.
Then, for any $k\geq 0$, we define a linear map $\tilde{\Phi }^{k}_{-j-1}:U_{k}^+\otimes {\cal U}_{-j-1}^-\rightarrow V_{k-j-1}$ by:
\begin{align}
&\tilde{\Phi }^{k}_{-j-1}(u_k^+\otimes r_{-j}^-(y_{-1}\otimes w_{-j}^-))\notag\\
&\quad:=\begin{cases}[y_{-1},\tilde{\Phi }^0_{-j}(u_0^+\otimes w_{-j}^-)]&(k=0)\\ [y_{-1},\tilde{\Phi }^k_{-j}(u_k^+\otimes w_{-j}^-)]-\tilde{\Phi }^{k-1}_{-j}(\tilde{\pi }^+(y_{-1}\otimes u_k^+)\otimes w_{-j}^-)&(k\geq 1)\end{cases}\label{eq;tildephi_phimap_k_-i-1}
\end{align}
where $y_{-1}\in V_{-1}$, $u_k^+\in U_k^+$ and $w_{-j}^-\in {\cal U}_{-j}^-$.
\par Consequently, we can define linear maps $\tilde{\Phi }^i_{-j}:U_i^+\otimes {\cal U}_{-j}^-\rightarrow V_{i-j}$ for all $i,j\geq 0$.
\end{defn}
\begin{pr}\label{pr;tildephi}
The linear map $\tilde{\Phi }^i_{-j}$ is well-defined and satisfies the following equation:
\begin{align}
B_L(a_{-i+j},\tilde{\Phi }^i_{-j}(u_i^+\otimes w_{-j}^-))=\langle \tilde{\pi }^+(a_{-i+j}\otimes u_i^+),w_{-j}^-\rangle \label{pr;tildephi_phimap}
\end{align}for any $i,j\geq 0$, $a_{-i+j}\in V_{-i+j}$, $u_i^+\in U_{i}^+$ and $w_{-j}^-\in {\cal U}_{-j}^-$.
\end{pr}
\begin{proof}
Let us show that the linear maps defined by the equations (\ref{eq;tildephi_phimap_0_0}), (\ref{eq;tildephi_phimap_i+1_0}) and (\ref{eq;tildephi_phimap_k_-i-1}) satisfy our claim by induction.
First, let us show that the linear map $\tilde{\Phi }^{i+1}_{0}$ $(i\geq 0)$ defined in (\ref{eq;tildephi_phimap_i+1_0}) is well-defined by induction on $i$.
For $i=0$, under the above notation, we have
\begin{align}
B_L(a_{-1},[x_1,\tilde{\Phi }^0_{0}(u_0^+\otimes w_{0}^-)])=&B_L([a_{-1},x_1],\tilde{\Phi }^0_{0}(u_0^+\otimes w_{0}^-))\notag=\langle \tilde{\pi }^+([a_{-1},x_1]\otimes u_0^+),w_0^-\rangle \\
=&\langle r_0^+(x_1\otimes u_0^+)(a_{-1}),w_0^-\rangle =\langle \tilde{\pi }^+(a_{-1}\otimes r_0^+(x_1\otimes u_0^+)),w_0^-\rangle \label{proof;tildephi_phimap_0+1_0}
\end{align}for any $a_{-1}\in V_{-1}$.
Thus, if $x_1^1,\ldots, x_1^{l}\in V_1$ and $u_0^{+,1},\ldots ,u_0^{ +,l}\in U_0^+$ satisfy $\sum _{s=1}^{l }r_0^+(x_1^s\otimes u_0^{+,s})=0$, then we have
\begin{align}
\sum _{s=1}^{l }B_L(a_{-1},[x_1^s,\tilde{\Phi }^0_{0}(u_0^{+,s}\otimes w_{0}^-)])=0
\end{align}
for any $a_{-1}\in V_{-1}$.
Since the restriction of $B_L$ to $V_{-1}\times V_1$ is non-degenerate, we have 
\begin{align}
\sum _{s=1}^{l }[x_1^s,\tilde{\Phi }^0_{0}(u_0^{+,s}\otimes w_{0}^-)]=0,
\end{align}and, thus, the map $\tilde{\Phi }^1_0$ is well-defined.
The equation (\ref{pr;tildephi_phimap}) follows from (\ref{proof;tildephi_phimap_0+1_0}).
\par For $i\geq 1$, under the notation of (\ref{eq;tildephi_phimap_i+1_0}), we have
\begin{align}
&B_L(a_{-i-1},[x_1,\tilde{\Phi }^i_{0}(u_i^+\otimes w_{0}^-)])=B_L([a_{-i-1},x_1],\tilde{\Phi }^i_{0}(u_i^+\otimes w_{0}^-))\notag=\langle \tilde{\pi }^+([a_{-i-1},x_1]\otimes u_i^+),w_0^-\rangle \\
&\quad =\langle \tilde{\pi }^+(a_{-i-1}\otimes \tilde{\pi }^+(x_1\otimes u_i^+)),w_0^-\rangle -\langle \tilde{\pi }^+(x_1\otimes \tilde{\pi }^+(a_{-i-1}\otimes u_i^+)),w_0^-\rangle \notag\\
&\quad =\langle \tilde{\pi }^+(a_{-i-1}\otimes r_i^+(x_1\otimes u_i^+)),w_0^-\rangle 
\end{align}
by the induction hypothesis for any $a_{-i-1}\in V_{-i-1}$.
Thus, by the same argument to the argument of the case where $i=0$, we have the well-definedness of $\tilde{\Phi }^{i+1}_0$ and that $\tilde{\Phi }^{i+1}_0$ satisfies the equation (\ref{pr;tildephi_phimap}).
Therefore, by induction, we can obtain our claim on $\tilde{\Phi }^{i+1}_0$ for all $i\geq 0$.
\par Let us show that the linear maps defined in (\ref{eq;tildephi_phimap_k_-i-1}) are well-defined.
We assume that an integer $i\geq 0$ satisfies the condition that we have linear maps $\tilde{\Phi }^k_{-i}:U_k^+\otimes {\cal U}_{-i}^-\rightarrow V_{k-i}$ for all $k\geq 0$ which satisfy the equation (\ref{pr;tildephi_phimap}).
When $i=0$, it has been shown that this assumption holds.
Then, we can show the well-definedness of the linear maps $\tilde{\Phi }_{-1}^k$ $(k\geq 0)$ by induction on $k$.
When $k=0$, we can show that $\tilde{\Phi }_{-1}^0$ is well-defined and satisfies (\ref{pr;tildephi_phimap}) by a similar argument to the argument of (\ref{eq;tildephi_phimap_i+1_0}).
When $k\geq 1$, we have
\begin{align}
&B_L(a_{-k+1},[y_{-1},\tilde{\Phi }^k_{0}(u_k^+\otimes w_{0}^-)]-\tilde{\Phi }^{k-1}_{0}(\tilde{\pi }^+(y_{-1}\otimes u_k^+)\otimes w_{0}^-))\notag\\
&\quad =B_L([a_{-k+1},y_{-1}],\tilde{\Phi }^k_{0}(u_k^+\otimes w_{0}^-))-B_L(a_{-k+1},\tilde{\Phi }^{k-1}_{0}(\tilde{\pi }^+(y_{-1}\otimes u_k^+)\otimes w_{0}^-))\notag\\
&\quad =\langle \tilde{\pi }^+([a_{-k+1},y_{-1}]\otimes u_k^+),w_{0}^-\rangle-\langle \tilde{\pi }^+(a_{-k+1}\otimes \tilde{\pi }^+(y_{-1}\otimes u_k^+)),w_{0}^-\rangle\notag\\
&\quad =-\langle \tilde{\pi }^+(y_{-1}\otimes \tilde{\pi }^+(a_{-k+1}\otimes u_k^+)),w_{0}^-\rangle\notag\\
&\quad =\langle \tilde{\pi }^+(a_{-k+1}\otimes u_k^+),\tilde{\pi }^-(y_{-1}\otimes w_{0}^-)\rangle =\langle \tilde{\pi }^+(a_{-k+1}\otimes u_k^+),r_{0}^-(y_{-1}\otimes w_{0}^-)\rangle 
\end{align}for any $k\geq 1$ and $a_{-k+1}\in V_{-k+1}$ under the notation of (\ref{eq;tildephi_phimap_k_-i-1}).
Thus, by a similar argument to the argument of (\ref{eq;tildephi_phimap_i+1_0}), we have the well-definedness of $\tilde{\Phi }^k_{-1}$ for all $k\geq 1$ and that $\tilde{\Phi }^k_{-1}$ satisfies the equation (\ref{pr;tildephi_phimap}).
For $i\geq 1$, by the same argument to the argument of the case where $i=0$, we have the well-definedness of $\tilde{\Phi }^{k}_{-i-1}$ for all $k\geq 0$ and that $\tilde{\Phi }^{k}_{-i-1}$ satisfies the equation (\ref{pr;tildephi_phimap}).
Thus, by induction, we have linear maps $\tilde{\Phi }^i_{-j}$ for all $i,j\geq 0$ which satisfies the equation (\ref{pr;tildephi_phimap}).
This completes the proof.
\end{proof}

As a corollary of Propositions \ref{pr;pairing_Linv_nondeg} and \ref{pr;tildephi}, we have the following theorem.
\begin{theo}\label{theo;stapstep}
Let $(\mathfrak{g},\rho,V,{\cal V},B_0)$ and $(\mathfrak{g},\pi,U,{\cal U},B_0)$ be standard pentads and assume that $B_0$ is symmetric.
Then a pentad $(L(\mathfrak{g},\rho,V,{\cal V},B_0),\tilde{\pi }^+,\tilde{U}^+,\tilde{{\cal U}}^-,B_L)$ is also a standard pentad whose $\Phi $-map, denoted by $\tilde{\Phi }_{\pi }^+$, is defined by:
\begin{align}
\tilde{\Phi }_{\pi }^+(u_i^+\otimes w_{-j}^-):=\tilde{\Phi }^i_{-j}(u_i^+\otimes w_{-j}^-)\label{th;eq;phimap}
\end{align}for any $i,j\geq 0$, $u_{i}^+\in U_i^+$ and $w_{-j}^-\in {\cal U}_{-j}^-$, where $\tilde{\Phi }^i_{-j}$ is the linear map defined in Definition \ref{defn;tildephi}.
\end{theo}

\subsection{Chain rule}\label{sec;chainrule}
Under the assumptions of sections \ref{sec;stapext_1} and \ref{sec;stapext_2}, let us construct the Lie algebra associated with a standard pentad of the form $(L(\mathfrak{g},\rho,V,{\cal V},B_0),\tilde{\pi }^+,U^+,{\cal U}^-,B_L)$.
To find the structure of the Lie algebra $L(L(\mathfrak{g},\rho,V,{\cal V},B_0),\tilde{\pi }^+,U^+,{\cal U}^-,B_L)$, we give the following theorem.

\begin{theo}[chain rule]\label{theo;chainrule}
Let $(\mathfrak{g},\rho,V,{\cal V},B_0)$ and $(\mathfrak{g},\pi,U,{\cal U},B_0)$ be standard pentads.
Assume that $B_0$ is symmetric.
Then a pentad $(L(\mathfrak{g},\rho,V,{\cal V},B_0),\tilde{\pi }^+,\tilde{U}^+,\tilde{\cal U}^-,B_L)$ is also a standard pentad and the Lie algebra associated with it is isomorphic to $L(\mathfrak{g},\rho\oplus \pi,V\oplus U,{\cal V}\oplus {\cal U},B_0)$, i.e. we have
\begin{align}
L(L(\mathfrak{g},\rho,V,{\cal V},B_0),\tilde{\pi }^+,\tilde{U}^+,\tilde{\cal U}^-,B_L)\simeq L(\mathfrak{g},\rho\oplus \pi,V\oplus U,{\cal V}\oplus {\cal U},B_0)\label{eq;chainrule}
\end{align}as Lie algebras.
\end{theo}
\begin{proof}
Note that the pentad $(\mathfrak{g},\rho\oplus \pi,V\oplus U,{\cal V}\oplus {\cal U},B_0)$ is a standard pentad whose $\Phi $-map $\Phi _{\rho\oplus \pi }$ is given by:
$$
\Phi _{\rho\oplus \pi }((v,u)\otimes (\phi ,\psi))=\Phi _{\rho }(v\otimes \phi )+\Phi _{\pi }(u\otimes \psi )
$$
where $v\in V$, $\phi \in {\cal V}$, $u\in U$, $\psi \in {\cal U}$ and $\Phi _{\rho }$ and $\Phi _{\pi }$ are the $\Phi $-maps of the pentads $(\mathfrak{g},\rho,V,{\cal V},B_0)$ and $(\mathfrak{g},\pi,U,{\cal U},B_0)$ respectively.
It has been already shown in Theorem \ref{theo;stapstep} that the pentad $(L(\mathfrak{g},\rho,V,{\cal V},B_0),\tilde{\pi }^+,\tilde{U}^+,\tilde{\cal U}^-,B_L)$ is standard.
We denote the $n$-graduations of $(\mathfrak{g},\rho,V,{\cal V},B_0)$ and  $(L(\mathfrak{g},\rho,V,{\cal V},B_0),\tilde{\pi }^+,\tilde{U}^+,\tilde{\cal U}^-,B_L)$ by $V_n$ and $(\tilde{U}^+)_n$, i.e. 
\begin{align}
L(\mathfrak{g},\rho,V,{\cal V},B_0)=\bigoplus _{n\in\mathbb{Z}}V_n,
\quad L(L(\mathfrak{g},\rho,V,{\cal V},B_0),\tilde{\pi }^+,\tilde{U}^+,\tilde{\cal U}^-,B_L)=\bigoplus _{m\in\mathbb{Z}}(\tilde{U}^+)_m.
\end{align}
Moreover, we denote $(\tilde{U}^+)_{1} $ and $(\tilde{U}^+)_{-1} $ by:
\begin{align}
(\tilde{U}^+)_{1} =\bigoplus _{i\geq 0} U_i^+,
\quad (\tilde{U}^+)_{-1} =\bigoplus _{j\geq 0}{\cal U}_{-j}^-.
\end{align}
Denote a bilinear form on $L(L(\mathfrak{g},\rho,V,{\cal V},B_0),\tilde{\pi }^+,\tilde{U}^+,\tilde{\cal U}^-,B_L)$ defined in {\rm Definition} \ref{pr;stap_bilinear_exist} by $\overline {B}_L$.
By Lemmas \ref{lem;derivationext} and \ref{lemma;der_Binv}, we can define derivations $\alpha $ and $\beta $ on $L(L(\mathfrak{g},\rho,V,{\cal V},B_0),\tilde{\pi }^+,U^+,{\cal U}^-,B_L)$ which satisfy
\begin{align}
\alpha (v_n)=nv_n,\quad\alpha (u_i^+)=iu_i^+,\quad\alpha (w_{-j}^-)=-jw_{-j}^-, \quad\beta (\tilde{u}^+_m)=m\tilde{u}^+_m
\end{align}
and
\begin{align}
\overline {B}_L(\alpha (\overline{z}),\overline{\omega })+\overline {B}_L(\overline{z},\alpha (\overline{\omega }))=\overline {B}_L(\beta (\overline{z}),\overline{\omega })+\overline {B}_L(\overline{z},\beta (\overline{\omega }))=0\label{overlineB_inv_der}
\end{align}
for any $n,m\in\mathbb{Z}$, $i,j\geq 0$, $v_n\in V_n$, $u_i^+\in U_i^+\subset (\tilde{U}^+)_{1}$, $w_{-j}^-\in {\cal U}_{-j}^-\subset (\tilde{U}^+)_{-1}$, $\tilde{u}_m^+\in (\tilde{U}^+)_m$ and $\overline {z},\overline {\omega }\in L(L(\mathfrak{g},\rho,V,{\cal V},B_0),\tilde{\pi }^+,\tilde{U}^+,\tilde{\cal U}^-,B_L)$.
Since $L(L(\mathfrak{g},\rho,V,{\cal V},B_0),\tilde{\pi }^+,\tilde{U}^+,\tilde{\cal U}^-,B_L) $ is generated by $L(\mathfrak{g},\rho,V,{\cal V},B_0)$ and $(\tilde{U}^+)_{\pm 1}$ and since $L(\mathfrak{g},\rho,V,{\cal V},B_0)$ and $(\tilde{U}^+)_{\pm 1}$ are generated by $V_0$, $V_{\pm 1}$, $U=U_0^+$ and ${\cal U}={\cal U}_0^-$, we have that $L(L(\mathfrak{g},\rho,V,{\cal V},B_0),\tilde{\pi }^+,\tilde{U}^+,\tilde{\cal U}^-,B_L) $ is generated by  $V_0$, $V_{\pm 1}$, $U_0^+$ and ${\cal U}_0^-$.
Put 
\begin{align*}
W_{(n,m)}:=\{\overline {X}\in L(L(\mathfrak{g},\rho,V,{\cal V},B_0),\tilde{\pi }^+,\tilde{U}^+,\tilde{\cal U}^-,B_L)\mid \alpha (\overline {X})=n\overline {X},\ \beta (\overline {X})=m\overline {X}\} 
\end{align*}
for any $n,m\in\mathbb{Z}$.
Then we can easily show that all eigenvalues of $\alpha $ and $\beta $ are integers by induction and that $[W_{(n,m)},W_{(k,l)}]\subset W_{(n+k,m+l)}$.
Thus, we can obtain the following $\mathbb{Z}$-grading of $L(L(\mathfrak{g},\rho,V,{\cal V},B_0),\tilde{\pi }^+,\tilde{U}^+,\tilde{\cal U}^-,B_L)$ induced by the eigenspace decomposition of $\gamma :=\alpha +\beta$:
\begin{align*}
L(L(\mathfrak{g},\rho,V,{\cal V},B_0),\tilde{\pi }^+,\tilde{U}^+,\tilde{\cal U}^-,B_L)=\bigoplus _{k\in \mathbb{Z}}(\bigoplus _{n+m=k}W_{(n,m)}).
\end{align*}
If we put $W^{\gamma }_k:=\{\overline{X}\mid \gamma (\overline{X})=k\overline{X}\}$, then we have $W^{\gamma }_k=\bigoplus _{n+m=k}W_{(n,m)}$ and, thus, we can obtain the following $\mathbb{Z}$-grading of $L(L(\mathfrak{g},\rho,V,{\cal V},B_0),\tilde{\pi }^+,\tilde{U}^+,\tilde{\cal U}^-,B_L)$:
\begin{align*}
L(L(\mathfrak{g},\rho,V,{\cal V},B_0),\tilde{\pi }^+,\tilde{U}^+,\tilde{\cal U}^-,B_L)=\bigoplus _{k\in\mathbb{Z}}W^{\gamma }_k.
\end{align*}
In particular, 
\begin{align}
W^{\gamma }_0=V_0,\quad  W^{\gamma }_1=V_1\oplus U_0^+, \quad W^{\gamma }_{-1}=V_{-1}\oplus {\cal U}_0^-.
\end{align}
We can easily show that $W^{\gamma }_{k+1}=[W^{\gamma }_{1},W^{\gamma }_{k}]$, $W^{\gamma }_{-k-1}=[W^{\gamma }_{-1},W^{\gamma }_{-k}]$ for all $k\geq 1$ and that the restriction of $\overline {B}_L$ to $W^{\gamma }_k\times W^{\gamma }_{-k}$ is non-degenerate for any $k\in \mathbb{Z}$ from (\ref{overlineB_inv_der}).
Therefore, by Theorem \ref{th;universality_stap}, we have the equation (\ref{eq;chainrule}).
\end{proof}

\begin{ex}\label{ex;3}
We retain to use the notations of Examples \ref{ex;1} and \ref{ex;2}.
Put ${\cal U}:=\mathbb{C}$ and define a representation $\varpi :\mathfrak{g}\otimes {\cal U}\rightarrow {\cal U}$ and a bilinear map $\langle \cdot,\cdot \rangle _U:U\times {\cal U}\rightarrow \mathbb{C}$ by:
$$
\varpi ((a,b,A)\otimes w):=-aw,\quad \langle u,w\rangle _U:=uw.
$$
We can identify ${\cal U}$ with $\mathrm {Hom}(U,\mathbb{C})$ via $\langle \cdot,\cdot \rangle _U$.
Then pentads $(L(\mathfrak{g},\rho,V,{\cal V},B_0),\tilde{\pi }^+,\tilde{U }^+,\tilde{\cal U}^-,B_L)$ and $(\mathfrak{g},\rho\oplus \pi,V\oplus U,{\cal V}\oplus {\cal U},B_0)$ are standard.
Let us show that the Lie algebra $L(\mathfrak{g},\rho\oplus \pi,V\oplus U,{\cal V}\oplus {\cal U},B_0)$ is isomorphic  to $\mathfrak{sl}_4$.
Put elements
\begin{align*}
H_0:=\begin{pmatrix}\frac{5}{4}&&&\\&\frac{1}{4}&&\\&&\frac{-3}{4}&\\&&&\frac{-3}{4}\end{pmatrix},\quad H_1:=\begin{pmatrix}\frac{3}{4}&&&\\&\frac{-1}{4}&&\\&&\frac{-1}{4}&\\&&&\frac{-1}{4}\end{pmatrix},\quad H_2:=\begin{pmatrix}\frac{1}{2}&&&\\&\frac{1}{2}&&\\&&\frac{-1}{2}&\\&&&\frac{-1}{2}\end{pmatrix}\in \mathfrak{sl}_4.
\end{align*}
Then we can obtain a $\mathbb{Z}$-grading of $\mathfrak{sl}_4$ by the eigenspace decomposition of $\mathrm {ad}\ H_0$:
\begin{align}
\mathfrak{sl}_4=\bigoplus _{i=-2}^2\mathfrak{l}_i \quad (\mathfrak{l}_i:=\{X\in\mathfrak{sl}_4\mid [H_0,X]=iX\}).\label{ex;grading}
\end{align}
In particular, 
\begin{align*}
&\mathfrak{l}_0=\{ \begin{pmatrix} a &0&0&0\\0&b &0&0\\0&0&\multicolumn{2}{c}{\raisebox{1.5ex}[0pt]{}} \\0&0&\multicolumn{2}{c}{\raisebox{1.5ex}[0pt]{\Large $A$}} \end{pmatrix}\mid a,b\in\mathbb{C}, A\in \mathfrak{gl}_2, a+b+\mathrm {Tr}(A)=0\}\simeq \mathfrak{gl}_1\oplus \mathfrak{gl}_1\oplus \mathfrak{sl}_2,\\
&\mathfrak{l}_1=\{ \begin{pmatrix} 0&u&0&0\\0&0&v_1&v_2\\0&0&0&0\\0&0&0&0\end{pmatrix}\mid u, v_1,v_2\in \mathbb{C}\},\quad \mathfrak{l}_{-1}=\{ \begin{pmatrix} 0&0&0&0\\\psi &0&0&0\\0&\phi _1&0&0\\0&\phi _2&0&0\end{pmatrix}\mid \psi , \phi _1,\phi _2\in \mathbb{C}\}.
\end{align*}
Then, we have that $\mathfrak{l}_0\simeq \mathbb{C}H_1\oplus \mathbb{C}H_2\oplus \mathfrak{sl}_2$ and that the restriction of a bilinear form $T$, defined by $T(X,X^{\prime }):=\mathrm{Tr}(XX^{\prime })$ $(X,X^{\prime }\in \mathfrak{sl}_4)$, to $\mathfrak{l}_0\times \mathfrak{l}_0$ satisfies:
$$
T\mid _{\mathfrak{l}_0\times \mathfrak{l}_0}((a,b,A),(a^{\prime},b^{\prime},A^{\prime}))=\frac{3}{4}aa^{\prime }+bb^{\prime}+\frac{1}{2}(ab^{\prime }+a^{\prime }b)+ \mathrm{Tr}(AA^{\prime}),
$$
where $a,a^{\prime }\in \mathbb{C}H_1$, $b,b^{\prime }\in \mathbb{C}H_2$, $A,A^{\prime }\in \mathfrak{sl}_2$.
Thus, we can easily show that the grading (\ref{ex;grading}) and the Killing form of $\mathfrak{sl}_4$, denoted by $K_{\mathfrak{sl}_4}$, satisfy the assumptions of Theorem \ref{th;universality_stap}
and that a pentad $(\mathfrak{l}_0,\mathrm {ad},\mathfrak{l}_1,\mathfrak{l}_{-1},K_{\mathfrak{sl}_4}\mid _{\mathfrak{l}_0\times \mathfrak{l}_0})$ is equivalent to $(\mathfrak{g},\rho\oplus \pi,V\oplus U,{\cal V}\oplus {\cal U},B_0)$ (cf. \cite[the theory of prehomogeneous vector spaces of parabolic type]{ru-1,ru-2,ru-3}).
Thus, by Theorems \ref{th;universality_stap} and \ref{theo;chainrule}, we have
$$
L(L(\mathfrak{g},\rho,V,{\cal V},B_0),\tilde{\pi }^+,\tilde{U }^+,\tilde{\cal U}^-,B_L)\simeq L(\mathfrak{g},\rho\oplus \pi,V\oplus U,{\cal V}\oplus {\cal U},B_0)\simeq \mathfrak{sl}_4.
$$
\par In this case, we can directly check that the Lie algebra $L(L(\mathfrak{g},\rho,V,{\cal V},B_0),\tilde{\pi }^+,\tilde{U }^+,\tilde{\cal U}^-,B_L)$ is isomorphic to $\mathfrak{sl}_4$ using Examples \ref{ex;sa_sl}, \ref{ex;1} and \ref{ex;2}.
In fact, by the results of Examples \ref{ex;1} and \ref{ex;2}, we have that the pentad $(L(\mathfrak{g},\rho,V,{\cal V},B_0),\tilde{\pi }^+,\tilde{U }^+,\tilde{\cal U}^-,B_L)$ is equivalent to the pentad $(\mathfrak{gl}_1\oplus \mathfrak{sl}_3,  \Lambda _1 ,\mathbb{C}^3,\mathbb{C}^3, \kappa _3)$, which is defined in Example \ref{ex;sa_sl}.
Thus, we have that the Lie algebra $L(L(\mathfrak{g},\rho,V,{\cal V},B_0),\tilde{\pi }^+,\tilde{U }^+,\tilde{\cal U}^-,B_L)$ is isomorphic to $\mathfrak{sl}_4$.
\end{ex}

\medskip
\begin{flushleft}
Nagatoshi Sasano\\
Institute of Mathematics-for-Industry\\
Kyushu University\\
744, Motooka, Nishi-ku, Fukuoka 819-0395\\
Japan\\
E-mail: n-sasano@imi.kyushu-u.ac.jp
\end{flushleft}

\end{document}